\documentclass[a4paper, 12pt]{amsart}
\usepackage{amsmath,amsthm,amssymb}
\usepackage[T1]{fontenc}
\usepackage{url}
\usepackage{amsmath,amsthm,amssymb}
\usepackage{verbatim}
\usepackage{mathrsfs}
\usepackage{graphics}
\usepackage{graphicx}
\usepackage{subfigure}
\usepackage{hyperref}
\usepackage{mathtools}
\usepackage{color}
\usepackage{wasysym}
\allowdisplaybreaks[1]

\newtheorem{theorem}{Theorem}[section]
\newtheorem{proposition}[theorem]{Proposition}
\newtheorem{lemma}[theorem]{Lemma}
\newtheorem{corollary}[theorem]{Corollary}

\theoremstyle{remark}
\newtheorem{remark}[theorem]{Remark}
\newtheorem{definition}{Definition}

\numberwithin{equation}{section}

\newcommand{\vep}{\varepsilon}
\newcommand{\R}{{\mathbb{R}}}

\newcommand{\Q}{{\mathbb{Q}}}
\newcommand{\C}{{\mathbb{C}}}
\newcommand{\Z}{{\mathbb{Z}}}
\newcommand{\N}{{\mathbb{N}}}

\newcommand{\sgn}{\operatorname{sgn}}

\newcommand{\erf}{\operatorname{erf}}
\newcommand{\erfc}{\operatorname{erfc}}

\newcommand{\rp}{\operatorname{Re}}
\newcommand{\ip}{\operatorname{Im}}

\begin{document}

\title[On resonant energy sets]{On resonant energy sets for Hamiltonian systems with reflections}
\author[Krzysztof Fr\k{a}czek]{Krzysztof Fr\k{a}czek}
\address{Faculty of Mathematics and Computer Science, Nicolaus
Copernicus University, ul. Chopina 12/18, 87-100 Toru\'n, Poland}
\email{fraczek@mat.umk.pl}

\date{\today}

%\thanks{Research was partially supported by  the Narodowe Centrum Nauki Grant 2022/45/B/ST1/00179}
\maketitle
%\begin{abstract}
%
%\end{abstract}

\begin{abstract}
We study two uncoupled oscillators, one horizontal and one vertical, moving in a rectilinear polygon (with only vertical and horizontal sides) and undergoing elastic reflections at its boundary.
The main purpose of the article is to analyze the occurrence of resonance in such systems, depending on the shape of the analytic potentials that determine the oscillators.
We define resonant energy levels; roughly speaking, these are levels for which the resonance phenomenon occurs for a large set of values of the parameter.
We focus on unimodal analytic potentials whose unique minimum is at zero.
The most important result of the work describes the size of the set of resonance levels in the form of the following trichotomy: it is either empty, a singleton, or large, namely non-empty and open.
{In the latter case, we show that an abundance of resonant orbits occurs only when the potentials are of a special type; we denote this family by $\mathcal{SP}$. This result can be regarded as a distant analogue of the classical Bertrand's theorem (1873), which characterizes centrally symmetric potentials in the presence of an abundance of periodic orbits.}

%We also indicate which classes of potentials each of the three possibilities can occur in. From this point of view, the last case (strongly resonant) is the most interesting. Then, the potentials belong to a special class of potentials, denoted by $\mathcal{SP}$,  which seems unknown in the literature. The presented results appear to be new, even in the simplest case, when  the  uncoupled oscillators are not trapped in any set.
\end{abstract}

\section{Introduction}
Let us consider the motion of a single point mass in the plane $\R^2$. Suppose that if $(q_1,q_2)$ are the position coordinates and $(p_1,p_2)$ are the corresponding  momenta, then  $p_1^2/2+p_2^2/2$ is the kinetic energy of the particle and $V(q_1,q_2)$ is its potential energy. We only deal with potentials
of the form $V(q_1,q_2)=V_1(q_1)+V_2(q_2)$, where $V_1,V_2:\R\to\R_{\geq 0}$ are unimodal analytic potentials that increase away from zero and tend to infinity as the argument tends to plus or minus infinity. %For simplicity, assume that zero is the only local minimum for both $V_1$ and $V_2$.
Then
\begin{equation}\label{eq:Hamfun}
H(p_1,p_2,q_1,q_2)=\frac{p_1^2}{2}+\frac{p_2^2}{2}+V_1(q_1)+V_2(q_2)
\end{equation}
is the Hamiltonian associated with our problem, and the corresponding Hamiltonian equation is of the form
\begin{equation}\label{eq:Ham2}
\frac{dp_1}{dt}=-V_1'(q_1), \quad\frac{dq_1}{dt}=p_1,\quad \frac{dp_2}{dt}=-V_2'(q_2), \quad\frac{dq_2}{dt}=p_2.
\end{equation}
In addition to the usual rules describing the behavior of a particle given by \eqref{eq:Ham2}, we assume that they move in a rectilinear polygon.  Denote by $\mathcal{RP}$ the family of rectilinear polygons; this is the family of bounded polygons (not necessarily connected or simply connected) whose boundary consists of finitely many bounded vertical and horizontal segments.
A rectilinear polygon has corners of two types: corners in which the smaller angle ($90^\circ$) is interior to the polygon are called convex, and corners in which the larger angle ($270^\circ$) is interior are called concave.
We also assume that the particle undergoes an elastic reflection upon meeting a wall. More precisely, if a trajectory meets a vertical segment at $(p_1,p_2,q_1,q_2)$, then it jumps to $(-p_1,p_2,q_1,q_2)$ and continues to evolve according to \eqref{eq:Ham2}. If a trajectory meets a horizontal segment at $(p_1,p_2,q_1,q_2)$, then it jumps to $(p_1,-p_2,q_1,q_2)$. Moreover, if a trajectory meets a {convex} corner at $(p_1,p_2,q_1,q_2)$, then it jumps to $(-p_1,-p_2,q_1,q_2)$.
This classical reflection law does not apply at a concave corner.
%This is the natural law of reflection in the classical sense, which does not apply when encountering a concave corner.
In that case, the problem of reflection should be considered in the framework of the geometrical theory of diffraction proposed by Keller in \cite{Kel}. When an incident ray meets a concave corner, it undergoes diffraction, which leads to the emergence of four directions: $(p_1,p_2)$, $(-p_1,p_2)$, $(p_1,-p_2)$, and $(-p_1,-p_2)$.
These directions correspond to singularities of the diffraction coefficient
%They are singularities of a diffraction coefficient
and indicate the preferred directions of energy propagation. Since exactly
one of these directions points outside the polygon,
%one direction suggests propagation outside the polygon
 we will ignore it and regard the remaining three as the natural (in the diffractive sense) continuation of the dynamics.

Denote by $(\varphi_t)_{t\in\R}=(\varphi^P_t)_{t\in\R}$ the Hamiltonian flow describing the behavior of a particle
in the polygon $P\in\mathcal{RP}$ when the Hamiltonian is given by \eqref{eq:Hamfun}. Each periodic orbit of the flow $(\varphi^P_t)_{t\in\R}$ {(including those propagating between corners)} is called a \emph{resonant orbit}.
If a resonant orbit avoids the corners of the polygon, then it is \emph{regular}; otherwise, it is \emph{singular}.
{ Regular periodic orbits give rise to resonance in the classical sense, whereas singular periodic orbits generate resonance in the sense of diffraction theory. In this case, the wave front travels back and forth between the corners. A modern survey of resonances for waves reflected from corners and of the role of diffractive geodesics  can be found in \cite{Hi-Wu}.}

%As we will see in Section~\ref{sec:genequ}, each singular resonant orbit connects two corners (not necessarily different) by a piecewise smooth arc while the particle moves back and forth between the corners along the arc.

For a given energy level $E\geq 0$ and any $0\leq \theta\leq E$ let
\[S_{E,\theta}:=\Big\{(p_1,p_2,q_1,q_2)\in\R^4:\frac{p_1^2}{2}+V_1(q_1)=\theta,\frac{p_2^2}{2}+V_2(q_2)=E-\theta,(q_1,q_2)\in P\Big\}.\]
Then the phase space of the flow $(\varphi_t)_{t\in\R}$, i.e., $S=\R^2\times P$ is foliated by invariant sets $\{S_{E,\theta}:E\geq 0,0\leq \theta\leq E\}$.

If the restriction of $(\varphi_t)_{t\in\R}$ to $S_{E,\theta}$ has a resonant orbit, then the pair $(E,\theta)$ is called resonant.
As we will see in Section~\ref{sec:genequ} (see Remark~\ref{rem:cyl}), if $(E,\theta)$ is  resonant and $S_{E,\theta}$ has a regular resonant orbit, then $S_{E,\theta}$ contains an open cylinder of regular periodic orbits surrounding the resonant orbit. Moreover, the boundary of the cylinder consists of a chain of singular resonant orbits. Conversely, if $S_{E,\theta}$ has a singular resonant orbit connecting { only} convex corners, it is surrounded by a cylinder of regular periodic orbits in $S_{E,\theta}$. For other types of singular resonant orbits, this phenomenon may not occur, {but in that case an isolated diffraction resonance may arise}.

\begin{definition}
We say that $E>0$ is a \emph{resonant energy level} if $(E,\theta)$ is a resonant pair  for uncountably many $\theta\in (0,E)$. Denote by $\mathcal{E}=\mathcal{E}(P,V_1,V_2)\subset (0,+\infty)$ the set of resonant energy levels.
\end{definition}
%The above definition of resonant energy levels can be slightly modified by weakening the claim of uncountability of resonant pairs by requiring that the set of resonant pairs is not locally finite, i.e., there exists $\vep>0$ such that $(E,\theta)$ is a resonant pair  for infinitely many $\theta\in [\vep,E-\vep]$.
%This modification of the definition does not change the arguments used in the evidence of the main results.of the article. Thus, all the results presented in the paper are also true for the modified version of the concept of resonant energy levels.

\subsection{Main results}\label{sec:mainresults}
From now on, we will deal only with analytic potentials $V:\R\to\R_{\geq 0}$ such that
\begin{gather}\label{eq:i}
V(0)=0,\quad y\cdot V'(y)>0\text{ for }y\neq 0\quad\text{and}\quad \lim_{y\to\pm\infty}V(y)=+\infty.
\end{gather}
We denote  the family of such potentials by $\mathcal{UM}$. For any $V\in\mathcal{UM}$,
let $m=\deg(V,0)\geq 2$ be the degree of the holomorphic extension of $V$ at $0$, i.e.,
\[V^{(m)}(0)\neq 0\text{ and }V'(0)=\ldots=V^{(m-1)}(0)=0.\]
In view of \eqref{eq:i}, $m$ is even.
Assume that $V\in\mathcal{UM}$ and let $m=\deg(V,0)$. In view of \cite[Sec.~3.12.5]{Sa-Ge}, there exists a bi-analytic map $V_*:\R\to\R$ (i.e., analytic bijection with analytic inverse)
such that $V^m_*=V$.

\begin{definition}
Let $\mathcal{SP}$ denote the \emph{special class} of analytic $\mathcal{UM}$-potentials $V:\R\to\R_{\geq 0}$ such that:
\begin{itemize}
\item[($i$)] the degree $\deg(V,0)=2$;
\item[($ii$)] the inverse of the bi-analytic map $V_*$ satisfies $V_*^{-1}(x)={c}x+f(x^2)$ for some ${c}\neq 0$ and an analytic $f:\R\to\R$.
\end{itemize}
\end{definition}

{
\begin{remark}
Let us note that special potentials are in one-to-one correspondence with the set of pairs
$(c,g)$, where $g:\R\to\R$ is a bounded odd  analytic function and $c$
is a positive real number such that $c+g$ is a positive function. This correspondence is given by
\[V\mapsto (c,g),\quad\text{where}\quad (V^{-1}_*)'(x)=c+g(x).\]
As noted in Remark~\ref{rem:SP}, when we limit ourselves to even potentials, the only elements of $\mathcal{SP}$ are quadratic functions which correspond to the pairs of the form $(c,0)$.
\end{remark}
\begin{remark}
As we will see in Sections~\ref{sec:propertiesa}~and~\ref{sec:geomqper}, $\mathcal{SP}$ potentials also admit a natural characterization in the context of one-dimensional barrierless oscillators and their isochronicity. For any energy level $E>0$, let $\omega(E)$ denote the oscillation frequency of the oscillator determined by a potential $V\in\mathcal{UM}$ at energy $E$. Then $\mathcal{SP}$ potentials are the only potentials for which the frequency function $E\mapsto\omega(E)$ is constant (or, equivalently, constant on some open interval).
\end{remark}}

The following theorem indicates that resonant energy levels can only appear when both potentials have degree two.
Moreover, we show a family of pairs of degree-two potentials for which no resonance occurs. %(see the irrationality condition), for which there are no resonant energy levels.

\begin{theorem}\label{thm:main1}
Let $V_1,V_2:\R\to\R_{\geq 0}$ be two $\mathcal{UM}$-potentials and let $m_1:=\deg(V_1,0)$ and $m_2:=\deg(V_2,0)$.
 Suppose that
\begin{itemize}
\item[$(a)$] at least one of the degrees $m_1$ and $m_2$ is greater than $2$; or
\item[$(b)$] $m_1=m_2=2$ and $V_1(x)=\omega V_2(\tau x)$ for some $\omega>0$ and $\tau\neq 0$ such that {$\sqrt{\tfrac{V''_1(0)}{V''_2(0)}}=|\tau|\sqrt{\omega}$} is irrational.
\end{itemize}
Then $\mathcal{E}(P,V_1,V_2)$ is empty {for any polygon $P\in \mathcal{RP}$}.
\end{theorem}

Finally, we present a deeper analysis of the set of resonant energies when the degrees of the two potentials are equal to two.
We show that the set of resonant energies is  generally not rich, i.e., it is empty or has only one element.
If the set of resonant energy levels has at least two elements, then it must be large (open, so uncountable), and this exceptional phenomenon occurs only when both potentials are in the special class $\mathcal{SP}$.

\begin{theorem}\label{thm:main2}
For any pair $V_1,V_2$ of $\mathcal{UM}$-potentials  and any polygon $P\in\mathcal{RP}$,
the set of resonant energy levels $\mathcal{E}(P,V_1,V_2)$ is bounded and satisfies the following trichotomy:
\begin{itemize}
\item[$(a)$] either $\mathcal{E}(P,V_1,V_2)$ is empty;
\item[$(b)$] or $\mathcal{E}(P,V_1,V_2)$ is a singleton, in which case $m_1=m_2=2$;
\item[$(c)$] or $\mathcal{E}(P,V_1,V_2)$ is non-empty and open, in which case $V_1,V_2$ belong to $\in\mathcal{SP}$ and $\sqrt{\frac{V''_1(0)}{V''_2(0)}}$ is rational.
\end{itemize}
\end{theorem}

\noindent
{Summarizing the case where both potentials $V_1$, $V_2$ are of degree $2$, we have:
\begin{itemize}
\item if at least one potential is not $\mathcal{SP}$, then we can have at most one resonant energy level;
\item if exactly one potential is $\mathcal{SP}$, then there are no resonant energy levels (it follows from \eqref{eq:ga4} in Proposition~\ref{prop:twobad});
\item if neither potential belongs to $\mathcal{SP}$, then we can have exactly one resonant energy level (some examples are constructed in Appendix~\ref{sec:badpot}).
\end{itemize}
If both potentials are in $\mathcal{SP}$, then:
\begin{itemize}
\item the rationality of $\sqrt{\frac{V''_1(0)}{V''_2(0)}}$ may yield  plenty of resonant energy levels, whereas
\item the irrationality of $\sqrt{\frac{V''_1(0)}{V''_2(0)}}$ implies the absence of resonant energy levels at all (it follows also from \eqref{eq:ga4} in Proposition~\ref{prop:twobad} and \eqref{eq:spab}).
\end{itemize}
If, in addition, the potentials are related by the rescaling condition $V_1(x)=\omega V_2(\tau x)$, then both are in $\mathcal{SP}$ or both are not. In the former case, as we have already seen, the resonance behavior is determined by whether $\sqrt{\tfrac{V''_1(0)}{V''_2(0)}}=|\tau|\sqrt{\omega}$ is rational.
 In the latter case, when both are not in $\mathcal{SP}$, we have:
\begin{itemize}
\item if $|\tau|\sqrt{\omega}$ is irrational or $|\tau|\sqrt{\omega}$ is rational with $\omega\neq 1$, then there are no resonant levels (see part ($c$) in Proposition~\ref{prop:twobad});
\item if $\omega=1$ and $\tau$ is rational, then we can have exactly one resonant energy level (see Remark~\ref{rmk:relat} in Appendix~\ref{sec:badpot}).
\end{itemize}}

Detecting the gap above for the size of the resonance energy set and understanding the role of the special class $\mathcal{SP}$ in its study is the most outstanding  achievement of the article.
This phenomenon seems new even in the simplest case when the uncoupled oscillators are not trapped in any set.
{
In this case, the third part of Theorem~\ref{thm:main2} can be regarded as an analogue of the classical Bertrand's theorem \cite{Ber}, which characterizes centrally symmetric potentials in the situation where all bounded orbits are periodic. Recall that in this setting the centrally symmetric potential must be either the harmonic oscillator
($V(r)=ar^2$)
 or the Kepler potential
($V(r)=-a/r$). In our setting, when we ignore the assumptions about the symmetry of potentials $V_1$ and $V_2$, special potentials play the role of these two potentials. More precisely, we should take into account only the harmonic oscillator, since in our approach we have no possibility of considering potentials with singularities.}

Hamiltonian systems with elastic reflections have been studied extensively in both integrable and hyperbolic settings; see, for example, \cite{Do,Du,Ko-Tr,Li,Woj,Zh} for example.

{
An important example of an integrable physical system with reflections is the Boltzmann system, in which a single particle moves in the plane under a gravitational field and undergoes elastic reflection from a horizontal line (see \cite{GJ}). In this system, for appropriate values of the first integrals, resonances (periodic orbits) arise, and are completely described in \cite{GR}.}

{The study of dynamics in integrable and near-integrable cases with the potentials considered here and impacts with
vertical and horizontal walls was introduced in \cite{Pn-RK18,Pn-RK21,Pn-RK22,Ya-RK}. The quasi-integrable
situation, with impacts for general rectilinear polygons that
is studied here was introduced in \cite{BEFCPRK}, where a conjugacy to motion on billiards and flat surfaces
via action-angle coordinates was constructed.
Its ergodic properties were studied for very simple (star-shaped) rectilinear polygons and even potentials
in \cite{Fr-RK}.}
%The study of dynamics in the quasi-integrable situation was initiated in recent years, see, for example \cite{BEFCPRK,Fr-RK,Pn-RK18,Pn-RK21,Pn-RK22,Ya-RK}, and involved the Hamiltonian systems considered in this article but for very simple rectilinear  polygons. In this context, the most advanced results were achieved in \cite{Fr-RK}, where precise analysis of invariant measures for star-shaped rectilinear polygons and even potentials was performed.
The main tool used in \cite{Fr-RK} was to pass to the framework of  translation surfaces and use the techniques developed to study curves in the moduli space of translation surfaces introduced in \cite{Mi-We14} and developed in \cite{Fr-Sh-Ul} and \cite{Fr}.

In the present article, we deal with a more general situation in which the polygon is an arbitrary rectilinear polygon, and the potentials need not be even. This severely limits the use of the tools built so far for studying  translation surface curves. This time, the primary motivation is to question the dynamics (persistence) of the system derived from our original system after small perturbations of the Hamiltonian \eqref{eq:Hamfun}.
Then, to apply techniques that mimic classical KAM, we need precise information about the occurrence of resonances for the unperturbed system. Understanding the occurrence of resonances gives us hope to formulate appropriate Diophantine properties and attempt an attack via KAM. This is a challenging task in the case of quasi-integrable systems, as shown in the ground-breaking articles by Marmi-Moussa-Yoccoz \cite{Ma-Mo-Yo,Ma-Mo-Yo2}.

{To conclude this section, we present a summarizing theorem that describes the dynamics of uncoupled oscillators in the case where they are not confined by any polygonal barriers. In this setting, for each pair $(E,\theta)$, the flow on $S_{E,\theta}$ is conjugate to a linear translation on a torus; hence it is either minimal (all orbits are dense) or completely periodic (all orbits have the same period).
\begin{theorem}
Suppose $V_1,V_2$ are $\mathcal{UM}$-potentials. If at least one of $m_1=\deg(V_1,0)$ and $m_2=\deg(V_2,0)$ is greater than $2$, then
for every $E>0$ the set of $\theta\in(0,E)$ for which the flow on $S_{E,\theta}$ is completely periodic is countable and dense. If $m_1=m_2=2$, then we have the following trichotomy:
\begin{itemize}
\item[$(a)$] for every $E>0$ the set of $\theta\in(0,E)$ for which the flow on $S_{E,\theta}$ is completely periodic is countable and dense;
\item[$(b)$] or there exists exactly one energy level $E_0>0$ such that for every $\theta\in(0,E_0)$ the flow on $S_{E_0,\theta}$ is completely periodic, and any other energy level behaves as in point $(a)$;
\item[$(c)$] or for every $E>0$ and for every $\theta\in(0,E)$ the flow on $S_{E,\theta}$ is completely periodic and all periods are the same (do not depend on $E$ and $\theta$). This situation occurs if and only if $V_1,V_2\in\mathcal{SP}$ and $\sqrt{\frac{V''_1(0)}{V''_2(0)}}$ rational.
\end{itemize}
\end{theorem}
Since the proof of this result is essentially scattered throughout the proofs of the previous theorem and does not contain any new concepts, we decided to omit it.}

\subsection{Structure and main tools of the paper}
The first standard step, performed in Section~\ref{sec:osctobil}, is to move to the framework of translation surfaces. Following the arguments from \cite{Fr-RK}, we show that the flow  $(\varphi_t)_{t\in\R}$ on $S_{E,\theta}$ is conjugated to a translation flow on a certain surface.
A complete description of how the surface parameters depend on $(E,\theta)$ can be found in Section~\ref{sec:osctobil}.
The remainder of the proof can be divided into three independent components: an analytic component (in Section~\ref{sec:propertiesa}), a translation component (in Section~\ref{sec:genequ}) and a Fourier component (in Section~\ref{sec:geomqper}). In Section~\ref{sec:propertiesa},
we prove a kind of independence of functions giving the parameters determining translation surfaces found in Section~\ref{sec:osctobil}.
The main tool for proving independence is to analyze the behavior of the mentioned functions at the ends of their domains.
We show that a sufficiently high-order derivative of the function has a singularity at the end of its domain, which is the most novel achievement of the analytic component of the paper. In Section~\ref{sec:genequ},  we  prove a simple criterion (Theorem~\ref{thm:main-min}) to show the absence of resonance for the {directional} flow on translation surfaces tiled by rectilinear polygons. In fact, we benefit
here from ideas developed in \cite{Fr}. The third component (Section~\ref{sec:geomqper}) is the most challenging. Here, we carry out an in-depth analysis of geometrically quasi-periodic analytic functions that arise as parameters of translation surfaces. This part is the most technically advanced and novel, but also the most valuable, as it is essential to show the existence of a gap when studying the size of the set of resonant energy levels. Finally, Section~\ref{sec:appl} combines all three components to prove the main theorems.

\subsection*{Acknowledgements}
The author thanks the anonymous reviewers for their critical comments, which made it possible to significantly improve the presentation of the results. Research was partially supported by  the Narodowe Centrum Nauki Grant 2022/45/B/ST1/00179

\section{From two-dimensional oscillations to polygonal billiards}\label{sec:osctobil}
In this section, we show that the flow $(\varphi^P_t)_{t\in\R}$ restricted to $S_{E,\theta}$ is conjugated to a billiard flow on a rectilinear polygon $P_{E,\theta}\in\mathcal{RP}$ so that the directions of the billiard orbits are $\pm\pi/4$, $\pm 3\pi/4$. The transition to billiards relies on the arguments in Section~3 of \cite{Fr-RK}. Although only even potentials were considered in \cite{Fr-RK}, all the arguments used there also apply to $\mathcal{UM}$-potentials, with minor modifications. The transition to billiard flows on rational polygons (which are rectilinear) provides an opportunity to use fruitfully some basic properties of linear flows on compact translation surfaces, which are formed from polygons in the unfolding procedure. In Section~\ref{sec:genequ}, we briefly introduce the theory of translation surfaces.

For every  $V\in\mathcal{UM}$, let $V^{-1}:[0, +\infty)\to[0,+\infty)$ be the inverse of its positive branch.
Then $V^{-1}$ is continuous and analytic on $(0, +\infty)$ with $V^{-1}(0)=0$ and $(V^{-1})'(y)>0$ for $y>0$.
Denote by $\bar{V}:\R\to\R_{\geq 0}$ the reflected version of $V$, i.e.,  $\bar{V}(y)=V(-y)$ for $y\in\R$. Then $\bar{V}$ also belongs to $\mathcal{UM}$.

\medskip

Suppose that $V_1,V_2:\R\to\R_{\geq 0}$ are $\mathcal{UM}$-potentials  and let  $P\in\mathcal{RP}$.
We study the behavior of Hamiltonian flow related to
the Hamiltonian equation of the form \eqref{eq:Ham2} with the additional rule that the particle wanders inside the polygon $P$
and bounces off its walls elastically.

Recall that $(\varphi_t)_{t\in\R}=(\varphi^P_t)_{t\in\R}$ is the  Hamiltonian flow describing the behavior of a particle
in the polygon $P$.
Its phase space is foliated by invariant sets $\{S_{E,\theta}:E\geq 0,0\leq \theta\leq E\}$.
We denote by $(\varphi^{P,E,\theta}_t)_{t\in\R}$ the restriction of  $(\varphi_t)_{t\in\R}$ to $S_{E,\theta}$.

By definition,
\[S_{E,\theta}\subset\R^2\times\left[-\bar{V}_1^{-1}(\theta),V_1^{-1}(\theta)\right]\times
\left[-\bar{V}_2^{-1}(E-\theta),V_2^{-1}(E-\theta)\right].\]
Let us consider new coordinates on
\begin{gather*}
\left[-\bar{V}_1^{-1}(\theta),V_1^{-1}(\theta)\right]\times\left[-\bar{V}_2^{-1}(E-\theta),V_2^{-1}(E-\theta)\right]\text{ and }\\
\R^2\times\left[-\bar{V}_1^{-1}(\theta),V_1^{-1}(\theta)\right]\times
\left[-\bar{V}_2^{-1}(E-\theta),V_2^{-1}(E-\theta)\right]
\end{gather*}
given by {scaled angle coordinates}
\begin{align*}
  \eta(q_1,q_2):=(\eta_1(q_1),\eta_2(q_2)) & =\left(\int_{0}^{q_1}\frac{dy}{\sqrt{2}\sqrt{\theta-V_1(y)}},\int_{0}^{q_2}\frac{dy}{\sqrt{2}\sqrt{E-\theta-V_2(y)}}\right) \\
  \bar{\eta}(p_1,p_2,q_1,q_2) & = \big(\eta'_1(q_1)p_1,\eta'_2(q_2)p_2, \eta_1(q_1),\eta_2(q_2)\big).
\end{align*}
Then, the flow $\big(\bar{\eta}\circ\varphi^{P,E,\theta}_t\circ\bar{\eta}^{-1}\big)_{t\in\R}$ coincides with
the billiard flow on
\[P_{E,\theta}:=\eta\left(P\cap \left(\left[-\bar{V}_1^{-1}(\theta),V_1^{-1}(\theta)\right]\times\left[-\bar{V}_2^{-1}(E-\theta),V_2^{-1}(E-\theta)\right]\right)\right)\]
so that the directions of its orbits are $\pm\pi/4$, $\pm 3\pi/4$. Indeed, by the definition of $S_{E,\theta}$,
\begin{align*}
\frac{d}{dt}\eta_1(q_1)&=\eta'_1(q_1)p_1=\sqrt{\frac{p_1^2/2}{\theta-V_1(q_1)}}=\pm 1,\\
\frac{d}{dt}\eta_2(q_2)&=\eta'_2(q_2)p_2=\sqrt{\frac{p_2^2/2}{E-\theta-V_2(q_2)}}=\pm 1.
\end{align*}
This calculation also shows that $\bar{\eta}$ extends continuously to the boundary of its domain.

%{It also shows that the mapping $\bar{\eta}$ is well defined (its first two coordinate functions), including on the boundary of its domain.}

As
$P\cap \left(\left[-\bar{V}_1^{-1}(\theta),V_1^{-1}(\theta)\right]\times\left[-\bar{V}_2^{-1}(E-\theta),V_2^{-1}(E-\theta)\right]\right)$
is a rectilinear polygon
and $\eta$ sends vertical and horizontal segments to vertical and horizontal segments, respectively,  $P_{E,\theta}$ is also a
rectilinear polygon.
We now determine some valuable numerical data of $P_{E,\theta}$ for any $E>0$ and $\theta\in(0,E)$.

\medskip

To every $P\in\mathcal{RP}$, we assign four finite (or empty) subsets of $\R_{\geq 0}$ defined as follows:
\begin{itemize}
\item let $X_P^+\subset \R_{\geq 0}$ be the set of non-negative first coordinates of  vertical sides in $P$;
\item let $X_P^-\subset \R_{> 0}$ be the set of absolute values of the negative first coordinates of  vertical sides in $P$;
\item let $Y_P^+\subset \R_{\geq 0}$ be the set of non-negative second coordinates of  horizontal sides in $P$;
\item let $Y_P^-\subset \R_{> 0}$ be the set of absolute values of the negative second coordinates of  horizontal sides in $P$.
\end{itemize}
Let
\[x_P^+:=\max X_P^+,\quad x_P^-:=\max X_P^-,\quad y_P^+:=\max Y_P^+,\quad y_P^-:=\max Y_P^-. \]
These are the parameters of the extreme sides of $P$, with the convention that the maximum of the empty set is zero.

For any $E>0$  and any $\xi>0$, let us consider eight continuous maps:
\begin{gather*}
a,\bar{a}:(0,+\infty)\to\R,\  a_{\xi}:[V_1(\xi),+\infty)\to\R, \ \bar{a}_{\xi}:[\bar{V}_1(\xi),+\infty)\to\R,  \\
b=b_E,\bar{b}=\bar b_E:[0,E)\to\R, \ b_{\xi}=b_{E,\xi}:[0,E-V_2(\xi)]\to\R \
(\text{if }V_2(\xi)<E)\text{ and }\\
\bar{b}_{\xi}=\bar{b}_{E,\xi}:[0,E-\bar{V}_2(\xi)]\to\R \
(\text{if }\bar{V}_2(\xi)<E)
\end{gather*}
given by
\begin{gather}\label{def:a}
a(\theta)=\int_0^{V^{-1}_1(\theta)}\frac{1}{\sqrt{2}\sqrt{\theta-V_1(y)}}dy,
\quad a_{\xi}(\theta)=\int_0^{\xi}\frac{1}{\sqrt{2}\sqrt{\theta-V_1(y)}}dy,\\
\nonumber
\bar{a}(\theta)=\int_0^{\bar{V}^{-1}_1(\theta)}\frac{1}{\sqrt{2}\sqrt{\theta-\bar{V}_1(y)}}dy,
\quad \bar{a}_{\xi}(\theta)=\int_0^{\xi}\frac{1}{\sqrt{2}\sqrt{\theta-\bar{V}_1(y)}}dy,\\
\nonumber
b_E(\theta)=\int_0^{V^{-1}_2(E-\theta)}\frac{1}{\sqrt{2}\sqrt{E-\theta-V_2(y)}}dy,
\quad  b_{E,\xi}(\theta)=\int_0^{\xi}\frac{1}{\sqrt{2}\sqrt{E-\theta-V_2(y)}}dy,\\
\nonumber
\bar{b}_E(\theta)=\int_0^{\bar{V}^{-1}_2(E-\theta)}\frac{1}{\sqrt{2}\sqrt{E-\theta-\bar{V}_2(y)}}dy,
\quad \bar{b}_{E,\xi}(\theta)=\int_0^{\xi}\frac{1}{\sqrt{2}\sqrt{E-\theta-\bar{V}_2(y)}}dy.
\end{gather}
\begin{figure}[h]
\includegraphics[width=0.5 \textwidth]{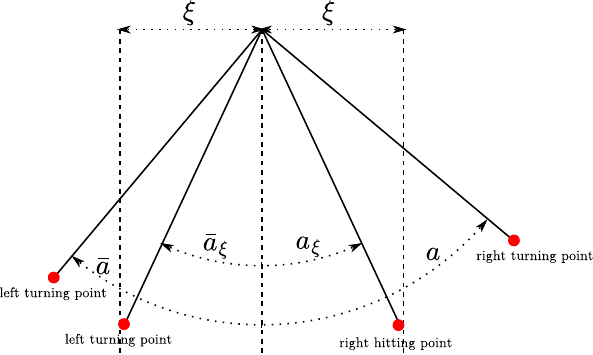}
\caption{The maps $a$, $\bar{a}$, $a_\xi$, $\bar{a}_\xi$ for non-even $V_1$.}
\label{fig:osc}
\end{figure}
The maps $a$, $\bar{a}$, $b$, $\bar{b}$ are ``quarter periods'' of oscillators without barriers, that is, they describe the time required for the oscillator to travel from the neutral position to the turning point, moving either to the right or to the left, or upward or downward. For polynomial potentials, these are elliptic integrals. In contrast, the functions $a_\xi$, $\bar{a}_\xi$, $b_\xi$, $\bar{b}_\xi$
describe the time to the first impact with a horizontal or vertical barrier located at a distance $\xi$ from the neutral point, see Figure~\ref{fig:osc}.
In view of Propositions~4.1~and~4.2 in \cite{Fr-RK},
\begin{align}\label{eq:analytic}
\begin{split}
&a,\bar{a}:(0,+\infty)\to\R,\ a_{\xi}:(V_1(\xi),+\infty)\to\R,\
\bar{a}_{\xi}:(\bar{V}_1(\xi),+\infty)\to\R,\\
&b,\bar{b}:[0,E)\to\R,\ b_{\xi}:[0,E-V_2(\xi))\to\R,\
\bar{b}_{\xi}:[0,E-\bar{V}_2(\xi))\to\R\\
&\text{ are all analytic.}
\end{split}
\end{align}
By the definition of $\eta$, for every $E>0$ and $\theta\in(0,E)$, the polygon $P_{E,\theta}\in\mathcal{RP}$ is rectilinear with:
\begin{align*}
X^+_{P_{E,\theta}}&\subset\{a(\theta)\}\cup\{a_x(\theta):x\in X^+_P,\, V_1(x)\leq \theta\},\\
X^-_{P_{E,\theta}}&\subset\{\bar{a}(\theta)\}\cup\{\bar{a}_x(\theta):x\in X^-_P,\, \bar{V}_1(x)\leq \theta\},\\
Y^+_{P_{E,\theta}}&\subset\{b(\theta)\}\cup\{b_y(\theta):y\in Y^+_P,\, V_2(y)\leq E-\theta\},\\
Y^-_{P_{E,\theta}}&\subset\{\bar{b}(\theta)\}\cup\{\bar{b}_y(\theta):y\in Y^-_P,\, \bar{V}_2(y)\leq E-\theta\}.
\end{align*}

\begin{figure}[h]
\includegraphics[width=0.8 \textwidth]{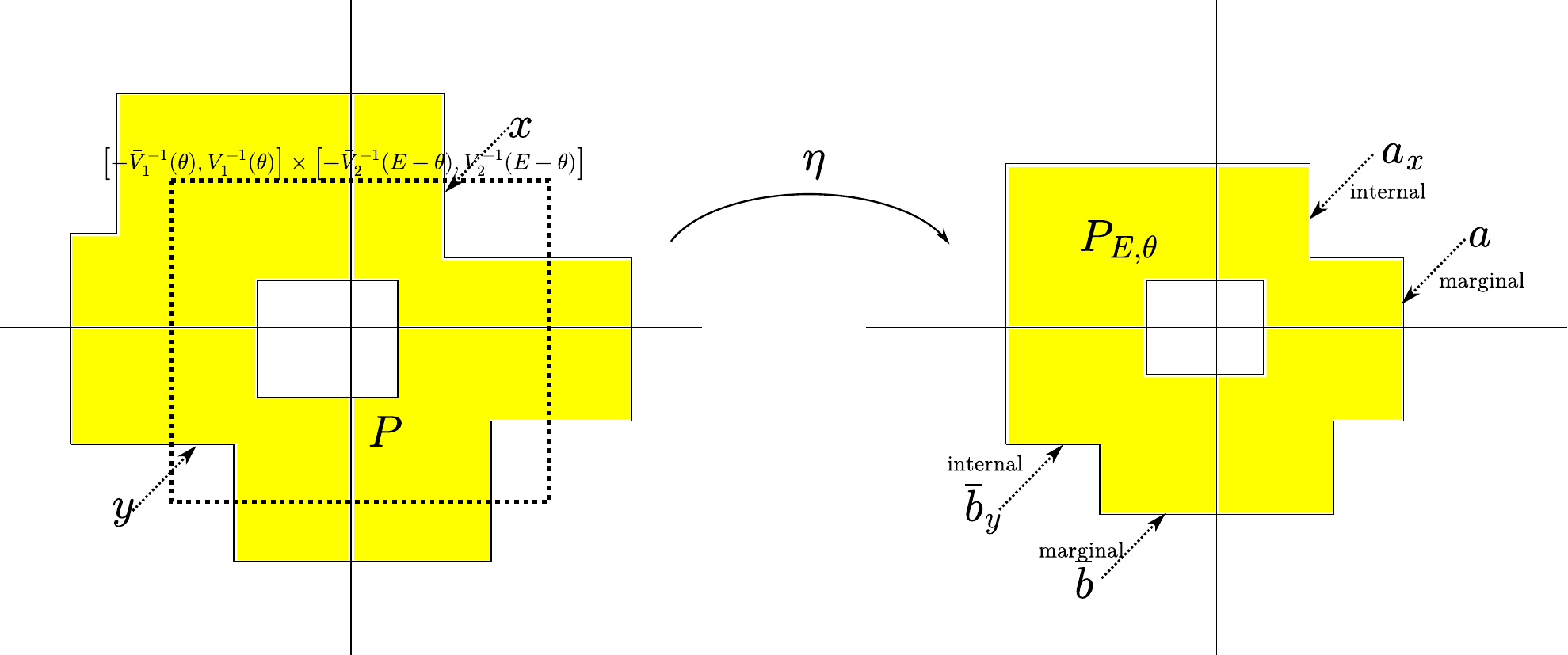}
\caption{Internal and marginal sides of $P_{E,\theta}$.}
\label{fig:sides}
\end{figure}
\begin{remark}\label{rmk:marg}
By the definition of $P_{E,\theta}$, its sides can be divided into two types: \emph{internal} and \emph{marginal}.  The first class consists of the images, under $\eta$, of the sides of the
polygon $P$.
The second class consists of the images of pieces of sides of the rectangle $[-\bar{V}_1^{-1}(\theta),V_1^{-1}(\theta)]\times
[-\bar{V}_2^{-1}(E-\theta),V_2^{-1}(E-\theta)]$, { see Figure~\ref{fig:sides}}. All marginal sides are extreme. Moreover, ${a}_x(\theta)$, $\bar{a}_x(\theta)$, ${b}_y(\theta)$, $\bar{b}_y(\theta)$ are the parameters of the internal sides, while ${a}(\theta)$,  $\bar{a}(\theta)$, ${b}(\theta)$, $\bar{b}(\theta)$ are the parameters of the marginal sides. In exceptional cases, a side can be both internal and marginal.
\end{remark}

For any $E>0$, let us consider the finite partition $\mathcal{J}_E$ (into open intervals) of the interval $[0,E]$ determined by the numbers
\[V_1(x),\ x\in X^+_P;\quad \bar{V}_1(x), \ x\in X^-_P;\quad E-V_2(y),\ y\in Y^+_P;\quad  E-\bar{V}_2(y),\ y\in Y^-_P.\]
Then, the partition points $\theta$ are those for which the polygon $P_{E,\theta}$ has sides that are both internal and marginal.

\begin{remark}\label{rem:XYI}
Fix $I\in \mathcal{J}_E$. Then $I\ni \theta\mapsto P_{E,\theta}\in\mathcal{RP}$ is a smooth curve of billiard tables in $\mathcal{RP}$.
{Moreover, the sets $X^+_I:=\{x\in X^+_P:V_1(x)<\theta\}$
%, \quad X^-_I:=\{x\in X^-_P:\bar{V}_1(x)<\theta\},\\
%&Y^+_I:=\{y\in Y^+_P:E-V_2(y)>\theta\}, \quad Y^-_I:=\{y\in Y^-_P:E-\bar{V}_2(y)>\theta\}
and its counterparts $X^-_I$, $Y^+_I$, $Y^-_I$ do not depend on the choice of $\theta\in I$.
Since any side that is not extreme must be internal, for every $\theta \in I$,
\begin{gather}\label{eq:XY+-1}
\begin{aligned}
X^+_{P_{E,\theta}}\setminus\{x^+_{P_{E,\theta}}\}\subset\{a_x(\theta):x\in X^+_I\},&\quad
%X^-_{P_{E,\theta}}\setminus\{x^-_{P_{E,\theta}}\}\subset\{\bar{a}_x(\theta):x\in X^-_I\},\\
%Y^+_{P_{E,\theta}}\setminus\{y^+_{P_{E,\theta}}\}\subset\{b_y(\theta):y\in Y^+_I\},&\quad
%Y^-_{P_{E,\theta}}\setminus\{y^-_{P_{E,\theta}}\}\subset\{\bar{b}_y(\theta):y\in Y^-_I\}.
\end{aligned}
\end{gather}
with the same properties for the remaining counterparts.
On the other hand, any extreme side can be marginal or internal, so
\begin{align}\label{eq:XY+-2}
\begin{aligned}
&x^+_{P_{E,\theta}}\!=a(\theta) \text{ for all }\theta\in I \text{ or } x^+_{P_{E,\theta}}\!=a_{x}(\theta) \text{ for all }\theta\in I
\text{ for some }x\in X^+_I,\\
%&x^-_{P_{E,\theta}}\!=\bar{a}(\theta) \text{ for all }\theta\in I \text{ or } x^-_{P_{E,\theta}}\!=\bar{a}_{x}(\theta) \text{ for all }\theta\in I
%\text{ for some }x\in X^-_I,\\
%&y^+_{P_{E,\theta}}\!=b(\theta) \text{ for all }\theta\in I \text{ or } y^+_{P_{E,\theta}}\!=b_{y}(\theta) \text{ for all }\theta\in I
%\text{ for some }y\in Y^+_I,\\
%&y^-_{P_{E,\theta}}\!=\bar{b}(\theta) \text{ for all }\theta\in I \text{ or } y^-_{P_{E,\theta}}\!=\bar{b}_{y}(\theta) \text{ for all }\theta\in I
%\text{ for some }y\in Y^-_I.
\end{aligned}
\end{align}
with the same properties  for the remaining counterparts.
More precisely, if an extreme side of $P_{E,\theta}$ is internal, then it must be the image of an extreme side of $P$. Hence,
\begin{align}\label{eq:XY+-3}
\begin{aligned}
\text{if }\ V_1(x^+_P)\leq \theta, \ \text{ then }\ & x^+_{P_{E,\theta}}=a_{x}(\theta)\ \text{ for some }\ x\in X^+_I;\\
%\text{if }\ \bar{V}_1(x^-_P)\leq \theta,\ \text{ then }\ & x^-_{P_{E,\theta}}=\bar{a}_{x}(\theta)\ \text{ for some }\ x\in X^-_I;\\
%\text{if }\ V_2(x^+_P)\leq E-\theta,\ \text{ then }\ & y^+_{P_{E,\theta}}=b_{y}(\theta)\  \text{ for some }\ y\in Y^+_I;\\
%\text{if }\ \bar{V}_2(x^+_P)\leq E-\theta,\ \text{ then }\ & y^-_{P_{E,\theta}}=\bar{b}_{y}(\theta)\ \text{ for some }\ y\in Y^-_I.
\end{aligned}
\end{align}
with the same properties  for the remaining counterparts.}

%\begin{itemize}
%\item every staircase lengths of $P^{\varsigma_1\varsigma_2}_{E,\theta}$ is of the form $a_x(\theta)$ for some $x\in X_I$;
%\item every staircase heights of $P^{\varsigma_1\varsigma_2}_{E,\theta}$ is of the form $b_y(\theta)$ for some $y\in Y_I$;
%\item if $V_1(x_k^{\varsigma_1\varsigma_2})<\theta$ for all
%$1\leq k\leq k(\bar{x}^{\varsigma_1\varsigma_2})$, then the width of $P^{\varsigma_1\varsigma_2}_{E,\theta}$ is of the form $a_x(\theta)$ for some $x\in X_I$;
%\item otherwise, the width of $P^{\varsigma_1\varsigma_2}_{E,\theta}$ is of the form $a(\theta)$;
%\item if $E-V_2(y_k^{\varsigma_1\varsigma_2})>\theta$ for all
%$1\leq k\leq k(\bar{x}^{\varsigma_1\varsigma_2})$, then the height of $P^{\varsigma_1\varsigma_2}_{E,\theta}$ is of the form $b_y(\theta)$ for some $y\in Y_I$;
%\item otherwise, the height  of $P^{\varsigma_1\varsigma_2}_{E,\theta}$ is of the form $b(\theta)$.
%\end{itemize}
\end{remark}

\section{Properties of functions $a$, $b$, $a_{\xi}$ and $b_{\xi}$}\label{sec:propertiesa}
In the previous section, we identified some parameters of the polygons $P_{E,\theta}$, which are determined by the functions of the form $a$, $b$, $a_{\xi}$, $b_{\xi}$ and their reflections.
In order to prove the absence of resonances in $S_{E,\theta}$ (in Section~\ref{sec:appl}), or equivalently for the billiard flow on $P_{E,\theta}$, we will need a form of independence for finite families of such analytic functions, see Proposition~\ref{prop:indep0}.
For this purpose, we perform an in-depth analysis of their behavior at the endpoints of their domains  in this section.

{ In what follows, for any interval $I\subset \R$, we denote by $C(I)$ the space of real-valued continuous functions on $I$, and by
$C^{\omega}(I)$ the space of real-valued analytic functions.}

Assume that $V:\R\to\R_{\geq 0}$ is a $\mathcal{UM}$-potential.
Let us consider
$a:(0,  +\infty)\to\R_{>0}$ and $a_{\xi}:[V(\xi),  +\infty)\to\R_{>0}$ ($\xi>0$) defined by \eqref{def:a}.
As noted above, $a$ and $a_{\xi}$ on $(V(\xi),  +\infty)$ are analytic, and $a_{\xi}$ is continuous at $V(\xi)$.
Let $m:=\deg(V,0)$  and let $V_*:\R\to\R$ be a bi-analytic map  (i.e., analytic bijection with analytic inverse)
such that $V^m_*=V$. We can choose $V_*$ so that $V_*'(x)>0$ for all $x\in \R$. Denote by $W:\R\to\R$ the inverse of $V_*$.
Using integration by substitution twice, we have
\begin{align}\label{eq:formofa}
\begin{aligned}
\sqrt{2}a(\theta)&=\int_0^{W(\theta^{\frac{1}{m}})}\frac{1}{\sqrt{\theta-V(y)}}dy=\left|\begin{array}{c}
                                                                                  x=V_*(y),\,y=W(x) \\
                                                                                  dy=W'(x)\,dx
                                                                                \end{array}\right|\\
&=\int_0^{\theta^{\frac{1}{m}}}\frac{W'(x)}{\sqrt{\theta-x^m}}dx
=\left|\begin{array}{c}
                                                                                  s=\frac{x}{\theta^{\frac{1}{m}}} \\
                                                                                  dx=\theta^{\frac{1}{m}}\,ds
                                                                                \end{array}\right|
=\frac{1}{\theta^{\frac{1}{2}-\frac{1}{m}}}\int_0^{1}\frac{W'(\theta^{\frac{1}{m}} s)}{\sqrt{1-s^m}} ds
\end{aligned}
\end{align}
and
\[
\sqrt{2}a_{\xi}(\theta)=\int_0^{\xi}\frac{1}{\sqrt{\theta-V(y)}}dy
=\int_0^{V_*(\xi)}\frac{W'(x)}{\sqrt{\theta-x^m}}dx
=\frac{1}{\theta^{\frac{1}{2}-\frac{1}{m}}}
\int_0^{\frac{V_*(\xi)}{\theta^{\frac{1}{m}}}}\frac{W'(\theta^{\frac{1}{m}} s)}{\sqrt{1-s^m}} ds
\]
for every $\theta\geq V(\xi)$. It follows that
\begin{equation}\label{eq:limfin}
\lim_{\theta\searrow 0} \theta^{\frac{1}{2}-\frac{1}{m}}a(\theta)=\frac{W'(0)}{\sqrt{2}}\int_0^{1}\frac{ds}{\sqrt{1-s^m}}>0.
\end{equation}

For any $c\geq 0$, we denote by $C^{\omega}(c,+\infty)$ the space of real analytic maps on $(c,+\infty)$.
For any $\alpha\geq 0$ and $c\geq 0$, let $D_{\alpha}:C^{\omega}(c,+\infty)\to C^{\omega}(c,+\infty)$ be the linear differential operator defined by
\[D_{\alpha}(f)(\theta)=m\frac{d}{d\theta}(\theta^\alpha f(\theta)).\]
Fix $\xi>0$ and let $s_0:=V_*(\xi)$.  For any $\alpha\geq 0$ and $k\in\Z_{\geq 0}$, let $D_{\alpha,k}:C^{\omega}(s_0^m,+\infty)\to C^{\omega}(s_0^m,+\infty)$ (recall that $s_0^m=V_*^m(\xi)=V(\xi)$) be the affine differential operator defined by
\[D_{\alpha,k}(f)(\theta)=D_{\alpha}(f)(\theta)+\frac{s^{k+1}_0}{\theta^{\frac{k+1}{m}
+\frac{1}{2}}}\frac{W^{(k+1)}(s_0)}{\sqrt{\theta-s_0^m}}.\]
%For brevity, we drop the subscript $\gamma$ if $\gamma=1$, i.e., $D_{\alpha,k}=D_{\alpha,k,1}$.
Then, for all $f,g\in C^{\omega}(s_0^m,+\infty)$, we have
\begin{equation}\label{eq:diffDs}
D_{\alpha,k}(f)-D_{\alpha,k}(g)=D_\alpha(f-g).
\end{equation}
For any $\alpha\geq 0$ and $k\in\Z_{\geq 0}$, let $f_{\alpha,k}\in C^{\omega}(s_0^m,+\infty)$  be given by
\begin{equation}\label{eq:diffD}
f_{\alpha,k}(\theta)=\frac{1}{\theta^{\alpha}}\int_0^{s_0\theta^{-\frac{1}{m}}}\frac{W^{(k+1)}(\theta^{\frac{1}{m}} s)s^k}{\sqrt{1-s^m}} ds.
\end{equation}
Of course, $\sqrt{2}a_{\xi}=f_{\frac{1}{2}-\frac{1}{m},0}$.

%\begin{lemma}\label{lem:fgamma}
By a standard chain-rule argument, for any $\alpha\geq 0$ and $k\in\Z_{\geq 0}$, we have
\begin{equation}\label{eq:Dak}
D_{\alpha,k}(f_{\alpha,k})= f_{1-\frac{1}{m},k+1}.
\end{equation}
%\end{lemma}
%
%\begin{proof}
%By standard chain rule argument, we have
%\begin{align*}
%\frac{1}{m}&D_\alpha(f_{\alpha,k})(\theta)=\frac{d}{d\theta}\int_0^{s_0\theta^{-\frac{1}{m}}}\frac{W^{(k+1)}(\theta^{\frac{1}{m}} s)s^k}{\sqrt{1-s^m}} ds\\
%&=\frac{d}{d\theta}(s_0\theta^{-\frac{1}{m}})\frac{W^{(k+1)}(s_0)(s_0\theta^{-\frac{1}{m}})^k}{\sqrt{1-s^m_0\theta^{-1}}}+
%\int_0^{s_0\theta^{-\frac{1}{m}}}\frac{d}{d\theta}\frac{W^{(k+1)}(\theta^{\frac{1}{m}} s)s^k}{\sqrt{1-s^m}} ds\\
%&=-\frac{1}{m}\frac{ s^{k+1}_0}{\theta^{\frac{k+1}{m}+\frac{1}{2}}}\frac{W^{(k+1)}(s_0)}{\sqrt{\theta-s^m_0}}+
%\frac{1}{m}\frac{1}{\theta^{1-\frac{1}{m}}}\int_0^{s_0\theta^{-\frac{1}{m}}}\frac{W^{(k+2)}(\theta^{\frac{1}{m}} s)s^{k+1}}{\sqrt{1-s^m}} ds.
%\end{align*}
%This immediately gives \eqref{eq:Dak}.
%\end{proof}

%\begin{corollary}
%For every $k\geq 1$ we have
%\[(D_{1-\frac{1}{m},k-1,\gamma}\circ \ldots \circ D_{1-\frac{1}{m},1,\gamma}\circ D_{\frac{1}{2}-\frac{1}{m},0,\gamma})(\gamma \sqrt{2}a_{\xi})(\theta)=
%\gamma f_{1-\frac{1}{m},k}(\theta).\]
%\end{corollary}

%\begin{proof}
%By \eqref{eq:Dak}, we have
%\[(D_{1-\frac{1}{m},k-1}\circ \ldots \circ D_{1-\frac{1}{m},1}\circ D_{\frac{1}{2}-\frac{1}{m},0})(a_{\xi})(\theta)=f_{1-\frac{1}{m},k}(\theta).\]
%Moreover,
%\begin{align*}\lim_{\theta\searrow V(\xi)}f_{1-\frac{1}{m},k}(\theta)&=
%\lim_{\theta\searrow V(\xi)}\frac{1}{\theta^{1-\frac{1}{m}}}\int_0^{s_0\theta^{-\frac{1}{m}}}\frac{W^{(k+1)}(\theta^{\frac{1}{m}} s)s^k}{\sqrt{1-s^m}} ds\\
%&=\frac{V_*(\xi)}{V(\xi)}\int_0^{1}\frac{W^{(k+1)}(V_*(\xi) s)s^k}{\sqrt{1-s^m}} ds.
%\end{align*}
%\end{proof}

Let us pass to the reflected version $\bar{V}:\R\to\R_{\geq 0}$ of $V$, i.e., $\bar{V}(y)=V(-y)$ for $y\in\R$.
Let us consider bi-analytic maps $\bar{V}_*:\R\to \R$ and $\bar{W}:\R\to\R$ given by
\[\bar{V}_*(y)=-V(-y)\text{ and }\bar{W}(x)=-W(-x).\]
Then $\bar{V}_*$ is strictly increasing,  $\bar{V}_*^m=\bar{V}$, and $\bar{W}:\R\to\R$ is the inverse of $\bar{V}_*$.

For any $\xi>0$, let $\bar{\xi}>0$ be the unique number such that $\bar{V}(\bar{\xi})=V(\xi)=s_0^m$.
By definition, $\bar{a}_{\bar{\xi}}:[V(\xi),+\infty)\to\R$ is given by
\[\sqrt{2}\bar{a}_{\bar{\xi}}(\theta)=\int_0^{\bar{\xi}}\frac{1}{\sqrt{\theta-\bar{V}(y)}}dy
=\frac{1}{\theta^{\frac{1}{2}-\frac{1}{m}}}
\int_0^{\frac{\bar{V}_*(\bar{\xi})}{\theta^{\frac{1}{m}}}}\frac{\bar{W}'(\theta^{\frac{1}{m}} s)}{\sqrt{1-s^m}} ds
\text{ for all } \theta\geq \bar{V}(\bar{\xi}).
\]
As $\bar{V}_*(\bar{\xi})=V_*(\xi)=s_0$, we have
\[\sqrt{2}\bar{a}_{\bar{\xi}}(\theta)
=\frac{1}{\theta^{\frac{1}{2}-\frac{1}{m}}}
\int_0^{s_0\theta^{-\frac{1}{m}}}\frac{\bar{W}'(\theta^{\frac{1}{m}} s)}{\sqrt{1-s^m}} ds.
\]
For any $\alpha\geq 0$ and $k\in\Z_{\geq 0}$, let $\bar{f}_{\alpha,k}\in C^{\omega}(s_0^m,+\infty)$ be given by
\[\bar{f}_{\alpha,k}(\theta)=\frac{1}{\theta^{\alpha}}\int_0^{s_0\theta^{-\frac{1}{m}}}\frac{\bar{W}^{(k+1)}(\theta^{\frac{1}{m}} s)s^k}{\sqrt{1-s^m}} ds.\]
Of course, $\sqrt{2}\bar{a}_{\bar{\xi}}=\bar{f}_{\frac{1}{2}-\frac{1}{m},0}$.

%\begin{lemma}\label{lem:fbar}
By a standard chain-rule argument, for any $\alpha\geq 0$, $k\in\Z_{\geq 0}$, and $\gamma\in\R$, if $W^{(k+1)}(s_0)=\gamma \bar{W}^{(k+1)}(s_0)$, then
\begin{equation}\label{eq:Dak1}
D_{\alpha,k}(\gamma\bar{f}_{\alpha,k})=\gamma\bar{f}_{1-\frac{1}{m},k+1}.
\end{equation}
%\end{lemma}
%\begin{proof}
%The proof uses the same argument as in the proof of \eqref{eq:Dak}:
%\begin{align*}
%\frac{1}{m}&D_\alpha(\gamma\bar{f}_{\alpha,k})(\theta)=\gamma\frac{d}{d\theta}\int_0^{s_0\theta^{-\frac{1}{m}}}\frac{\bar{W}^{(k+1)}(\theta^{\frac{1}{m}} s)s^k}{\sqrt{1-s^m}} ds\\
%&=\frac{d}{d\theta}(s_0\theta^{-\frac{1}{m}})\frac{\gamma\bar{W}^{(k+1)}(s_0)(s_0\theta^{-\frac{1}{m}})^k}{\sqrt{1-s^m_0\theta^{-1}}}+
%\gamma\int_0^{s_0\theta^{-\frac{1}{m}}}\frac{d}{d\theta}\frac{\bar{W}^{(k+1)}(\theta^{\frac{1}{m}} s)s^k}{\sqrt{1-s^m}} ds\\
%&=-\frac{1}{m}\frac{s^{k+1}_0}{\theta^{\frac{k+1}{m}+\frac{1}{2}}}\frac{W^{(k+1)}(s_0)}{\sqrt{\theta-s^m_0}}+
%\frac{1}{m}\frac{\gamma}{\theta^{1-\frac{1}{m}}}\int_0^{s_0\theta^{-\frac{1}{m}}}\frac{\bar{W}^{(k+2)}(\theta^{\frac{1}{m}} s)s^{k+1}}{\sqrt{1-s^m}} ds.
%\end{align*}
%\end{proof}

\begin{remark}\label{rem:exN}
For any real $\gamma$, we denote by $\gamma:\R\to\R$ the linear map $\gamma(x)=\gamma\cdot x$.
Then, for every $\gamma\neq 0$, we have $V\circ \gamma\in\mathcal{UM}$. Obviously, $\bar{V}=V\circ(-1)$.

Suppose that $\bar{V}\neq V\circ \gamma $. It follows that $\bar{V}_*\neq V_*\circ \gamma $, and hence $W\neq \gamma \bar{W}$. Since both $\bar{W}$ and $W$ are analytic,
there exists $N\in\N$ such that
\begin{equation}\label{def:N}
W^{(k)}(s_0)=\gamma \bar{W}^{(k)}(s_0)\text{ for all $1\leq k<N$ and }W^{(N)}(s_0)\neq \gamma \bar{W}^{(N)}(s_0).
\end{equation}
Indeed, otherwise ${W}^{(k)}(s_0)=\gamma \bar W^{(k)}(s_0)$ for all $k\geq 1$, and hence $W-\gamma \bar{W}$ is constant.
As $W(0)=\bar{W}(0)=0$, this gives $W=\gamma \bar{W}$ and a contradiction.
\end{remark}

Suppose that $\gamma\in\R$ and $N\geq 2$ satisfy \eqref{def:N}. In view of \eqref{eq:Dak1},  we have
\[D_{\frac{1}{2}-\frac{1}{m},0}(\gamma\bar{f}_{\frac{1}{2}-\frac{1}{m},0})=\gamma\bar{f}_{1-\frac{1}{m},1}
\text{ and }D_{1-\frac{1}{m},k}(\gamma\bar{f}_{1-\frac{1}{m},k})=\gamma\bar{f}_{1-\frac{1}{m},k+1}\]
for all $1\leq k\leq N-2$. Moreover, by \eqref{eq:Dak},  we have
\[D_{\frac{1}{2}-\frac{1}{m},0}({f}_{\frac{1}{2}-\frac{1}{m},0})={f}_{1-\frac{1}{m},1}
\text{ and }D_{1-\frac{1}{m},k}({f}_{1-\frac{1}{m},k})={f}_{1-\frac{1}{m},k+1}\]
for all $1\leq k\leq N-2$. Hence, by \eqref{eq:diffDs},
\begin{align}\label{eq:indf}
\begin{aligned}
{f}_{1-\frac{1}{m},1}-\gamma\bar{f}_{1-\frac{1}{m},1}&=
D_{\frac{1}{2}-\frac{1}{m}}({f}_{\frac{1}{2}-\frac{1}{m},0}-\gamma\bar{f}_{\frac{1}{2}-\frac{1}{m},0})\\
{f}_{1-\frac{1}{m},k+1}-\gamma\bar{f}_{1-\frac{1}{m},k+1}&=
D_{1-\frac{1}{m}}({f}_{1-\frac{1}{m},k}-\gamma\bar{f}_{1-\frac{1}{m},k})\text{ for }1\leq k\leq N-2.
\end{aligned}
\end{align}

For any $n\in\Z_{\geq 0}$ and $\xi>0$, let $D^{(n)}:C^\omega(V(\xi),+\infty)\to C^\omega(V(\xi),+\infty)$ be a linear operator given by
\[D^{(n)}:=\left\{
\begin{array}{rcl}
\underbrace{D_{1-\frac{1}{m}}\circ\ldots\circ D_{1-\frac{1}{m}}}_{n-1\text{ times}}\circ D_{\frac{1}{2}-\frac{1}{m}}&\text{if}& n\geq 1\\
Id &\text{if}& n=0.
\end{array}
\right.\]

The following technical result shows the behavior of the analytic maps $a_{\xi},\bar{a}_{\bar{\xi}}:(V(\xi),  +\infty)\to\R_{>0}$ at the end of their domain. It is crucial in proving Proposition~\ref{prop:indep0} regarding independence.

\begin{proposition}\label{prop:gamma}
Suppose that $\gamma$ is a real number such that $\bar{V} \neq V\circ \gamma$ and let $N\geq 1$ be defined by \eqref{def:N}. Then
\[\lim_{\theta\searrow V(\xi)}D^{(N)}(a_{\xi}-\gamma \bar{a}_{\bar{\xi}})(\theta)=\pm\infty.\]
\end{proposition}

\begin{proof}
{
Since $\sqrt{2}{a}_{\xi}={f}_{\frac{1}{2}-\frac{1}{m},0}$ and $\sqrt{2}\bar{a}_{\bar{\xi}}=\bar{f}_{\frac{1}{2}-\frac{1}{m},0}$,
by \eqref{eq:indf}, \eqref{eq:Dak}, and~\eqref{eq:Dak1}, we have
%
%, if $N\geq 2$, then
%\[\sqrt{2}D^{(N-1)}(a_{\xi}-\gamma \bar{a}_{\bar{\xi}})=D^{(N-1)}({f}_{\frac{1}{2}-\frac{1}{m},0}-\gamma\bar{f}_{\frac{1}{2}-\frac{1}{m},0})
%={f}_{1-\frac{1}{m},N-1}-\gamma\bar{f}_{1-\frac{1}{m},N-1},\]
%and if $N= 1$, then
%\[\sqrt{2}D^{(N-1)}(a_{\xi}-\gamma \bar{a}_{\bar{\xi}})={f}_{\frac{1}{2}-\frac{1}{m},0}-\gamma\bar{f}_{\frac{1}{2}-\frac{1}{m},0}.\]
%In view of \eqref{eq:Dak}~and~\eqref{eq:Dak1}, for every $\alpha\geq 0$ and $k\in\Z_{\geq 0}$, we have
%\begin{align*}
%D_\alpha({f}_{\alpha,k})(\theta)&=-\frac{s^{k+1}_0}{\theta^{\frac{k+1}{m}+\frac{1}{2}}}\frac{ W^{(k+1)}(s_0)}{\sqrt{\theta-s^m_0}}
%+{f}_{1-\frac{1}{m},k+1}(\theta)\\
%D_\alpha(\gamma\bar{f}_{\alpha,k})(\theta)&=-\frac{ s^{k+1}_0}{\theta^{\frac{k+1}{m}+\frac{1}{2}}}\frac{\gamma\bar{W}^{(k+1)}(s_0)}{\sqrt{\theta-s^m_0}}
%+\gamma\bar{f}_{1-\frac{1}{m},k+1}(\theta).
%\end{align*}
%It follows that

\begin{align*}
\sqrt{2}D^{(N)}(a_{\xi}-\gamma \bar{a}_{\bar{\xi}})=\frac{s^{N}_0}{\theta^{\frac{N}{m}+\frac{1}{2}}}\frac{\gamma \bar{W}^{(N)}(s_0)-W^{(N)}(s_0)}{\sqrt{\theta-V(\xi)}}+({f}_{1-\frac{1}{m},N}(\theta)-\gamma\bar{f}_{1-\frac{1}{m},N}(\theta)).
\end{align*}
Since the right limits of ${f}_{1-\frac{1}{m},N}$ and $\bar{f}_{1-\frac{1}{m},N}$ at $V(\xi)$ are finite, we obtain}
%\begin{align*}
%\lim_{\theta\searrow V(\xi)}{f}_{1-\frac{1}{m},N}(\theta)&=\lim_{\theta\searrow V(\xi)}
%\frac{1}{\theta^{1-\frac{1}{m}}}\int_0^{s_0\theta^{-\frac{1}{m}}}\frac{W^{(N+1)}(\theta^{\frac{1}{m}} s)s^k}{\sqrt{1-s^m}} ds\\
%&=\frac{V_*(\xi)}{V(\xi)}\int_0^{1}\frac{W^{(N+1)}(V_*(\xi) s)s^k}{\sqrt{1-s^m}} ds
%\end{align*}
%is finite, and similarly the limit
%\[\lim_{\theta\searrow V(\xi)}\bar{f}_{1-\frac{1}{m},N}(\theta)=\frac{\bar{V}_*(\bar{\xi})}{\bar{V}(\bar{\xi})}\int_0^{1}\frac{\bar{W}^{(N+1)}(\bar{V}_*(\bar{\xi}) s)s^k}{\sqrt{1-s^m}} ds
%\]
%is also finite, we obtain
\[\lim_{\theta\searrow V(\xi)}D^{(N)}(a_{\xi}-\gamma \bar{a}_{\bar{\xi}})(\theta)=
\left\{
\begin{array}{rcl}
+\infty&\text{if}& \gamma \bar{W}^{(N)}(s_0)>W^{(N)}(s_0);\\
-\infty&\text{if}& \gamma \bar{W}^{(N)}(s_0)<{W}^{(N)}(s_0).
\end{array}
\right.\]
\end{proof}

The following simple lemma says that for non-even potentials the assumption $\bar{V} \neq V\circ \gamma$ is always satisfied, except for the obvious case $\gamma=-1$. If $V$ is even, the additional exceptional case $\gamma=1$ must also be excluded.

\begin{lemma}\label{lem:gamma}
If $V$ is not even, then $\bar{V} \neq V\circ \gamma$ for every $\gamma\neq -1$.
\end{lemma}

\begin{proof}
Suppose, contrary to our claim, that for some $\gamma\neq -1$, we have $V(\gamma x)=V(-x)$ for all $x\in\R$.
If $|\gamma|<1$, then
\[V(x)=V((-\gamma)^nx)\to V(0)=0\text{ as }n\to+\infty,\]
and
if $|\gamma|>1$, then
\[V(x)=V((-\gamma)^{-n}x)\to V(0)=0\text{ as }n\to+\infty\]
for every $x\in\R$. This gives $|\gamma|=1$.
As $\gamma\neq -1$, it follows that $\gamma=1$. Hence, $V(-x)=V(\gamma x)$ contradicts the
assumption that $V$ is not even.
\end{proof}

Suppose that $V_1,V_2:\R\to\R_{\geq 0}$ are $\mathcal{UM}$-potentials. Recall that the corresponding (reflected)
maps $\bar{V}_1,\bar{V}_2$ are also in $\mathcal{UM}$.
The following key result states a kind of independence for finite families of functions of the form $a$, $a_x$, $b$, $b_y$, $\bar{a}$, $\bar{a}_x$, $\bar{b}$, $\bar{b}_y$  in the case of non-even potentials. In the case of even potentials, we must restrict ourselves to families of functions of the form $a$, $a_x$, $b$, $b_y$  because $\bar{a}=a$, $\bar{a}_x=a_x$, $\bar{b}=b$, $\bar{b}_y=b_y$. This result is crucial in proving Theorem~\ref{thm:main1} in Section~\ref{sec:appl}.

{ Although the following result is formulated and proved for non-even potentials, an analogous version is also valid when at least one of the potentials is even. In that case, the statement and the proof are essentially the same and, in fact, slightly simpler. For the sake of clarity, however, we discuss the even case only briefly in a remark following the proof.}

\begin{proposition}\label{prop:indep0}
Let  $V_1,V_2:\R\to\R_{\geq 0}$ be {non-even} $\mathcal{UM}$-potentials.
Let $m_1=\deg (V_1,0)$ and $m_2=\deg (V_2,0)$.
Fix an energy level $E>0$.
Assume that $K,L,\bar{K},\bar{L}\in\Z_{\geq 0}$ and
\begin{equation}\label{eq:seqxy}
0<x_1<\!\ldots\!<x_K,\ 0<y_1<\!\ldots\!<y_L,\
0<\bar{x}_1<\!\ldots\!<\bar{x}_{\bar{K}},\ 0<\bar{y}_1<\!\ldots\!<\bar{y}_{\bar{L}}
\end{equation}
 are such that
\[v_1+v_2<E,\text{ where }v_1=\max\{V_1(x_K),\bar{V}_1(\bar{x}_{\bar{K}})\}, v_2=\max\{V_2(y_L),\bar{V}_2(\bar{y}_{\bar{L}})\}.\]
%If $V_1$ is even, then we, additionally, assume that $\bar{K}=0$; similarly, if $V_2$ is even, then we assume that $\bar{L}=0$.
%Suppose that both potentials $V_1$ and $V_2$ are not even.
Let $\Delta_0$ be the finite subset of $C([v_1,E-v_2])\cap C^{\omega}(v_1,E-v_2)$  consisting of
\begin{gather*}
 a_{x_i}\text{ for }1\leq i\leq K,\quad \bar{a}_{\bar{x}_i}\text{ for }1\leq i\leq \bar{K},\quad b_{y_i}\text{ for }1\leq i\leq L,\quad \bar{b}_{\bar{y}_i}\text{ for }1\leq i\leq \bar{L}.
\end{gather*}
Denote by $\Delta$ the set arising from $\Delta_0$ by extending it by $a$, $b$, $\bar{a}$ and $\bar{b}$.

Let $\gamma_\delta$ for $\delta\in\Delta$ be real coefficients, not all zero such that $\gamma_a\cdot\gamma_{\bar{a}}\geq 0$ and $\gamma_b\cdot\gamma_{\bar{b}}\geq 0$.
If at least one $\gamma_{\delta}$ is non-zero for some $\delta\in\Delta_0$, then
\begin{equation}\label{neq:ga1}
\sum_{\delta\in \Delta}\gamma_\delta\delta(\theta)\neq 0\text{ for all but countably many }\theta\in (v_1,E-v_2).
\end{equation}
Suppose that at least one $\gamma_a$, $\gamma_{\bar{a}}$, $\gamma_b$, $\gamma_{\bar{b}}$ is non-zero with  $\gamma_a\cdot\gamma_{\bar{a}}\geq 0$ and $\gamma_b\cdot\gamma_{\bar{b}}\geq 0$. If $m_1>2$ or $m_2>2$,
then
\begin{equation}\label{eq:ga2}
\gamma_a a(\theta)+\gamma_{\bar{a}}\bar{a}(\theta)+\gamma_b b(\theta)+\gamma_{\bar{b}}\bar{b}(\theta)\neq 0\text{ for all but finitely many }\theta\in (v_1,E-v_2).
\end{equation}
\end{proposition}

{
\begin{remark}
As we have already mentioned, \eqref{neq:ga1} and \eqref{eq:ga2} are conditions resembling independence of functions  appearing in these sums, and the outlines of their proofs by contradiction are quite simple. First, assuming that \eqref{neq:ga1} and \eqref{eq:ga2} are not satisfied, using analyticity we obtain equalities on entire intervals on which the functions are well defined. Next, we identify the left endpoint of the interval, which plays a key role in reaching a contradiction. For \eqref{neq:ga1}, we find one or two functions whose derivatives $D^{(n)}$ have singularities at this point. The remaining functions are regular there. Using these singularities and Proposition~\ref{prop:gamma}, we obtain a trivialization of the resulting equality. For \eqref{eq:ga2}, the left endpoint is zero, which is a singular point for the functions $a$ and $\bar{a}$, and regular for $b$ and $\bar{b}$.
\end{remark}}
\begin{proof}[Proof of Proposition~\ref{prop:indep0}]
{Since the proof is technical in nature, for greater clarity we divide it into parts and cases.}
%We provide the proof only in cases where both potentials are not even.
%In other cases, the proof runs along similar lines. Some of the arguments are even a little simpler.

\medskip

\noindent
\textbf{Part 1.} Contrary to our claim, suppose \eqref{neq:ga1} does not hold.
Since all $\delta\in \Delta$ are analytic on $(v_1,E-v_2)$ and continuous on $[v_1,E-v_2]$,
we have
\begin{equation}\label{eq:gamma}
\sum_{\delta\in \Delta}\gamma_\delta\delta(\theta)= 0\ \text{ for all  }\ \theta\in [v_1,E-v_2].
\end{equation}
Assume that at least one $\gamma_\delta$  is non-zero for some $\delta\in\Delta_0$.
Without loss of generality, we can assume that:
\begin{align*}
&\gamma_{a_{x_K}}\neq 0\text{ if }V_1({x_K})=v_1\text{, and }\gamma_{\bar{a}_{\bar{x}_{\bar{K}}}}\neq 0\text{ if }\bar{V}_1({\bar{x}_{\bar{K}}})=v_1\text{, and }\\
&\gamma_{b_{y_L}}\neq 0\text{ if }{V}_2({y_L})=v_2\text{, and }\gamma_{\bar{b}_{\bar{y}_{\bar{L}}}}\neq 0\text{ if }\bar{V}_2({\bar{y}_{\bar{L}}})=v_2.
\end{align*}
Otherwise, one can artificially truncate the sequences in \eqref{eq:seqxy}.

Suppose that $K\geq 1$ and $V_1({x_K})=v_1$, so $\gamma_{a_{x_K}}\neq 0$. In other cases (i.e., $\bar{V}_1({\bar{x}_{\bar{K}}})=v_1$ or ${V}_2({y_L})=v_2$ or $\bar{V}_2({\bar{y}_{\bar{L}}})=v_2$), the proof is the same, so we skip it.

\noindent
\textbf{Case 1.} Suppose that $\bar{V}_1({\bar{x}_{\bar{K}}})<v_1=V_1({x_K})$.
Then
\[a_{x_K}(\theta)=-\sum_{\delta\in\Delta\setminus \{a_{x_K}\}}\frac{\gamma_\delta}{\gamma_{a_{x_K}}}\delta(\theta)\text{ for all }
\theta \in (v_1,E-v_2).\]
In view of \eqref{eq:analytic}, all $\delta\in\Delta\setminus \{a_{x_K}\}$ are also analytic on $(\tilde{v}_1,E-v_2)$,
where
\[\tilde{v}_1=\max\{V_1(x_{K-1}),\bar{V}_1(\bar{x}_{\bar{K}})\}<v_1.\]
Since for every $\alpha\geq 0$, the operator
$D_\alpha$ maps $C^{\omega}(\tilde{v}_1,E-v_2)$ to $C^{\omega}(\tilde{v}_1,E-v_2)$, it follows that the limit $\lim_{\theta\searrow v_1}D^{(n)}(\delta)(\theta)$ exists and is finite for every $\delta\in\Delta\setminus\{a_{x_K}\}$ and $n\geq 1$. Hence,
\begin{equation}\label{eq:limf}
\text{the limit }\lim_{\theta\searrow v_1}D^{(n)}(a_{x_K})(\theta)\text{ exists and is finite for every $n\geq 1$.}
\end{equation}
As $W_1'((V_1)_*(x_K))>0$, we can apply
Proposition~\ref{prop:gamma} to $V=V_1$ and $\gamma=0$. Since $V_1(x_K)=v_1$, we obtain
\begin{equation}\label{eq:-inf}
\lim_{\theta\searrow v_1}D^{(1)}(a_{x_K})(\theta)=-\infty,
\end{equation}
contrary to \eqref{eq:limf}.

\noindent
\textbf{Case 2.}  Suppose that $\bar{V}_1({\bar{x}_{\bar{K}}})=v_1=V_1({x_K})$.
Then
\[a_{x_K}(\theta)-\gamma\bar{a}_{\bar{x}_{\bar{K}}}(\theta)=-\sum_{\delta\in\Delta\setminus \{a_{x_K},\bar{a}_{\bar{x}_{\bar{K}}}\}}\frac{\gamma_\delta}{\gamma_{a_{x_K}}}\delta(\theta)\text{ for all }
\theta \in (v_1,E-v_2),\]
where $\gamma=-\tfrac{\gamma_{\bar{a}_{\bar{x}_{\bar{K}}}}}{\gamma_{a_{x_K}}}$.
Moreover,  all $\delta\in\Delta\setminus \{a_{x_K},\bar{a}_{\bar{x}_{\bar{K}}}\}$ are also analytic on $(\tilde{v}_1,E-v_2)$,
where $\tilde{v}_1=\max\{V_1(x_{K-1}),\bar{V}_1(\bar{x}_{\bar{K}-1})\}<v_1$. It follows that the limit of $D^{(n)}(\delta)(\theta)$
as $\theta\searrow v_1$ exists and is finite for every $\delta\in\Delta\setminus\{a_{x_K},\bar{a}_{\bar{x}_{\bar{K}}}\}$ and $n\geq 1$. Hence,
\begin{equation}\label{eq:limf1}
\text{the limit }\lim_{\theta\searrow v_1}D^{(n)}(a_{x_K}-\gamma\bar{a}_{\bar{x}_{\bar{K}}})(\theta)\text{ exists and is finite for every $n\geq 1$.}
\end{equation}

\noindent
\textbf{Case 2.1.} Assume that $\gamma\neq -1$. Since $V_1$ is not even, by Lemma~\ref{lem:gamma} and Remark~\ref{rem:exN}, there exists $N\geq 1$
such that
\begin{equation*}%\label{def:n}
W_1^{(k)}(v^{\frac{1}{m_1}}_1)=\gamma \bar{W}_1^{(k)}(v^{\frac{1}{m_1}}_1)\text{ for all $1\leq k<N$ and }W_1^{(N)}(v^{\frac{1}{m_1}}_1)\neq \gamma \bar{W}_1^{(N)}(v^{\frac{1}{m_1}}_1).
\end{equation*}
By Proposition~\ref{prop:gamma} applied to $V=V_1$ and $s_0=v^{\frac{1}{m_1}}_1$, since $V_1(x_K)=\bar{V}_1({\bar{x}_{\bar{K}}})=v_1$, we obtain $\lim_{\theta\searrow v_1}D^{(N)}(a_{x_K}-\gamma\bar{a}_{\bar{x}_{\bar{K}}})(\theta)=\pm\infty$,
contrary to \eqref{eq:limf1}.

\noindent
\textbf{Case 2.2.} Assume that $\gamma= -1$. In view of \eqref{eq:-inf},
\[\lim_{\theta\searrow v_1}D^{(1)}(a_{x_K})(\theta)=-\infty\text{ and }\lim_{\theta\searrow v_1}D^{(1)}(\bar{a}_{\bar{x}_{\bar{K}}})(\theta)=-\infty.\]
Hence, $\lim_{\theta\searrow v_1}D^{(1)}(a_{x_K}-\gamma\bar{a}_{\bar{x}_{\bar{K}}})(\theta)=-\infty$,
contrary to \eqref{eq:limf1}.

This completes the proof of \eqref{neq:ga1}.

\medskip

\noindent
\textbf{Part 2.} Suppose that at least one $\gamma_a$, $\gamma_{\bar{a}}$, $\gamma_b$, $\gamma_{\bar{b}}$ is non-zero, $\gamma_a\cdot\gamma_{\bar{a}}\geq 0$, $\gamma_b\cdot\gamma_{\bar{b}}\geq 0$, $m_1>2$, and \eqref{eq:ga2} does not hold. Since $a$, $\bar{a}$, $b$, $\bar{b}$ are analytic on $(0,E)$, we have
\begin{equation*}
\gamma_a a(\theta)+\gamma_{\bar{a}}\bar{a}(\theta)+\gamma_b b(\theta)+\gamma_{\bar{b}}\bar{b}(\theta)= 0\text{ for all }\theta\in (0,E).
\end{equation*}
As $b$, $\bar{b}$ are continuous on $[0,E)$, the limits of $b(\theta)$ and $\bar{b}(\theta)$ as $\theta\searrow 0$ exist and are finite.
It follows that
\begin{equation*}
\lim_{\theta\searrow 0}\theta^{\frac{1}{2}-\frac{1}{m_1}}(\gamma_a a(\theta)+\gamma_{\bar{a}}\bar{a}(\theta))=
-\lim_{\theta\searrow 0}\theta^{\frac{1}{2}-\frac{1}{m_1}}(\gamma_b b(\theta)+\gamma_{\bar{b}}\bar{b}(\theta))=0.
\end{equation*}
On the other hand, by \eqref{eq:limfin} and $\bar{W}_1'(0)=W_1'(0)$,
\begin{align*}
\lim_{\theta\searrow 0}\theta^{\frac{1}{2}-\frac{1}{m_1}}(\gamma_a a(\theta)+\gamma_{\bar{a}}\bar{a}(\theta))&=
\frac{\gamma_a W_1'(0)+\gamma_{\bar{a}} \bar{W}_1'(0)}{\sqrt{2}}\int_0^{1}\frac{s\,ds}{\sqrt{1-s^{m_1}}}\\
&=\sqrt{2}(\gamma_a +\gamma_{\bar{a}}){W}_1'(0)\int_0^{1}\frac{s\,ds}{\sqrt{1-s^{m_1}}}.
\end{align*}
It follows that $\gamma_a +\gamma_{\bar{a}}=0$. As  $\gamma_a$, $\gamma_{\bar{a}}$ have the same sign, we obtain $\gamma_a=\gamma_{\bar{a}}=0$.
Therefore, $\gamma_b b(\theta)+\gamma_{\bar{b}}\bar{b}(\theta)= 0$ for all $\theta\in (0,E)$. As $\gamma_b$ and $\gamma_{\bar{b}}$ have the same sign and $b(\theta)$ and $\bar{b}(\theta)$ are positive for all $\theta\in (0,E)$, this gives $\gamma_b=\gamma_{\bar{b}}=0$, and hence a contradiction.

This completes the proof of \eqref{eq:ga2}.
\end{proof}

{
\begin{remark}
Assume that the potential $V_1$ is even; the case of an even potential $V_2$ is treated analogously.
Then $\bar{V}_1=V_1$ and the vertical polygon parameters coming from the positive and negative parts have the same form; that is, $\bar{a}=a$ and $\bar{a}_x=a_{x}$. Therefore, we no longer distinguish the positive- and negative-side parameters by considering one common family of (positive) parameters. Hence, we can assume that $\bar{K}=0$, i.e., there are no parameters $\bar{x}_i$.

If $V_1$  is even, then $a=\bar{a}$ and that is why we can ignore in all sums the terms involving $\bar{a}$.  Accordingly,
we may assume that $\gamma_{\bar{a}}=0$, which makes the assumption $\gamma_{{a}}\cdot\gamma_{\bar{a}}\geq 0$ unnecessary.
%
%If $V_1$ (or $V_2$) is even, then $a=\bar{a}$ ($b=\bar{b}$ resp.) and $\gamma_a=\gamma_{\bar{a}}$ ($\gamma_b=\gamma_{\bar{b}}$ resp.), so the assumption is automatically satisfied.
\end{remark}}

\subsection{Degree-two potentials}\label{sec:potdeg2}
In this section, we discuss  condition \eqref{eq:ga2} in the case of $\deg(V_1,0)=\deg(V_2,0)=2$.
In this context, the special class  $\mathcal{SP}$ of potentials naturally appears.
Recall that $V:\R\to\R_{\geq 0}$ belongs to $\mathcal{SP}$ if:
\begin{itemize}
\item[($i$)] $V\in\mathcal{UM}$ with the degree $\deg(V,0)=2$ and
\item[($ii$)] the corresponding bi-analytic map $W:\R\to\R$ is of the form $W(x)=cx+f(x^2)$ or equivalently $W^{(2n+1)}(0)=0$ for $n\geq 1$.
\end{itemize}

\begin{remark}\label{rem:SP}
Notice that if $V\in \mathcal{SP}$, then
\begin{equation}\label{eq:SP}
W'(x)+\bar{W}'(x)=2W'(0) \text{ for all }x\in\R.
\end{equation}
Indeed, as $\bar{W}(x)=-W(-x)=cx-f(x^2)$, we have $W(x)+\bar{W}(x)=2cx$. Hence, $W'(x)+\bar{W}'(x)=2c=2W'(0)$.

In fact, condition ($ii$) is equivalent to the constancy of $W'+\bar{W}'$. Indeed, if $W'+\bar{W}'$ is constant,
then for every $n\geq 1$, we have
\[W^{(2n+1)}(x)=-\bar{W}^{(2n+1)}(x)=-W^{(2n+1)}(-x).\]
Hence, $W^{(2n+1)}(0)=0$ for all $n\geq 1$.

Moreover, if $V\in \mathcal{SP}$ is even, then $V$ is a quadratic map. Indeed, if $V$ is even, then
\[cx+f(x^2)=W(x)=-W(-x)=cx-f(x^2),\text{ so }f=0.\]
It follows that $V(x)=(V_*(x))^2=\big(\tfrac{x}{c}\big)^2$.
\end{remark}

\begin{definition}
Let $V_1,V_2:\R\to\R_{\geq 0}$ be $\mathcal{UM}$-potentials with $\deg(V_1,0)=\deg(V_2,0)=2$.
We denote by $\mathfrak E=\mathfrak E(V_1,V_2)\subset (0,+\infty)$ the set of all energy levels $E>0$ for which there exists an interval $I\subset (0,E)$ and integers
$\gamma_{a}$,  $\gamma_{\bar{a}}$, $\gamma_{b}$, $\gamma_{\bar{b}}$  not all zero with
$\gamma_{a}\cdot\gamma_{\bar{a}}\geq 0$ and $\gamma_{b}\cdot\gamma_{\bar{b}}\geq 0$ such that
\begin{equation}\label{eq:ga3}
\gamma_a a(\theta)+\gamma_{\bar{a}}\bar{a}(\theta)+\gamma_b b(\theta)+\gamma_{\bar{b}}\bar{b}(\theta)= 0\text{
for uncountably many  }\theta \in I.
\end{equation}
\end{definition}
As we will see in Section~\ref{sec:genequ} (see Corollary~\ref{cor:EE}), we have $\mathcal E(P,V_1,V_2)\subset \mathfrak E(V_1,V_2)$.
\begin{proposition}\label{prop:twobad}
Let $V_1,V_2:\R\to\R_{\geq 0}$ be $\mathcal{UM}$-potentials with $\deg(V_1,0)=\deg(V_2,0)=2$.
If $E\in\mathfrak{E}$ and \eqref{eq:ga3} holds, then $\gamma_{a}+\gamma_{\bar{a}}\neq 0$, $\gamma_{b}+\gamma_{\bar{b}}\neq 0$, and
\begin{equation}\label{eq:ga4}
(\gamma_a+\gamma_{\bar{a}}) (a(\theta)+\bar{a}(\theta))+(\gamma_b+\gamma_{\bar{b}})(b(\theta)+\bar{b}(\theta))= 0\text{
for all  }\theta \in (0,E).
\end{equation}
\begin{itemize}
\item[$(a)$] If $\mathfrak E$ has at least two elements, then  $V_1,V_2\in \mathcal{SP}$.
\item[$(b)$] If $E\in\mathfrak E$ and $V_1(x)=\omega V_2(\tau x)$ for some $\omega>0$ and $\tau\neq 0$, then
\[\frac{\gamma_{a}+\gamma_{\bar{a}}}{\gamma_{b}+\gamma_{\bar{b}}}=-|\tau|\sqrt{\omega}=-\sqrt{\frac{V''_1(0)}{V''_2(0)}}.\]
{\item[$(c)$] If additionally $\omega\neq 1$, then $V_1,V_2\in \mathcal{SP}$.}
% with
%$\tfrac{\gamma_{a}+\gamma_{\bar{a}}}{\gamma_{b}+\gamma_{\bar{b}}}=-\sqrt{\tfrac{V''_1(0)}{V''_2(0)}}$.
\end{itemize}
\end{proposition}

\begin{remark}
The proof of part ($a$) in Proposition~\ref{prop:twobad} uses some profound results on geometrically quasi-periodic maps formulated and proved in Section~\ref{sec:geomqper}. Due to their technical nature and lengthy proof, we decided to postpone them in a separate section.
\end{remark}

\begin{proof}[Proof of Proposition~\ref{prop:twobad}]
Suppose that $E\in\mathfrak{E}$.
Since $a$, $\bar{a}$, $b$, $\bar{b}$ are analytic on $(0,E)$, we have
\begin{equation}\label{eq:sumzero}
\gamma_a a(\theta)+\gamma_{\bar{a}}\bar{a}(\theta)+\gamma_b b(\theta)+\gamma_{\bar{b}}\bar{b}(\theta)= 0\text{ for all  }\theta\in(0,E).
\end{equation}
Notice that $\gamma_{a}+\gamma_{\bar{a}}\neq 0$ and $\gamma_{b}+\gamma_{\bar{b}}\neq 0$. Indeed, suppose that
$\gamma_{a}+\gamma_{\bar{a}}= 0$. As $\gamma_{a}\cdot\gamma_{\bar{a}}\geq 0$, it follows that $\gamma_{a}=\gamma_{\bar{a}}=0$.
By assumption, at least one of $\gamma_{b}$ and $\gamma_{\bar{b}}$ is non-zero. As $\gamma_{b}\cdot\gamma_{\bar{b}}\geq 0$ and the maps
$b$ and $\bar{b}$ take only positive values, $\gamma_b b(\theta)+\gamma_{\bar{b}}\bar{b}(\theta)\neq 0$ for all $\theta\in(0,E)$,
contrary to \eqref{eq:sumzero}.

Let us consider an analytic map $c:(0,+\infty)\to\R$ given by
\[c(\theta)=-\int_0^1\frac{\gamma_bW'_2(\sqrt{\theta}s)+\gamma_{\bar{b}}\bar{W}'_2(\sqrt{\theta}s)}{\sqrt{2}\sqrt{1-s^2}}ds\text{ for }\theta>0.\]
As $c(E-\theta)=-\gamma_b b(\theta)-\gamma_{\bar{b}}\bar{b}(\theta)$ for all  $\theta\in(0,E)$ (see \eqref{eq:formofa}), we have
\begin{equation*}\label{eq:aac}
\gamma_a a(\theta)+\gamma_{\bar{a}}\bar{a}(\theta)=c(E-\theta)\text{ for every }E\in\mathfrak{E}\text{  and for every }\theta\in(0,E).
\end{equation*}
It follows that $\gamma_a a+\gamma_{\bar{a}}\bar{a}$ and $c$ extend analytically to $\R$.
In view of Lemma~\ref{lem:evenU}, if $V_1$ is not even, then $\gamma_a=\gamma_{\bar{a}}$ and if  $V_2$ is not even, then  $\gamma_b=\gamma_{\bar{b}}$.
Therefore,
\begin{gather*}
\gamma_a a(\theta)+\gamma_{\bar{a}}\bar{a}(\theta)=\frac{\gamma_a+\gamma_{\bar{a}}}{2} (a(\theta)+\bar{a}(\theta)),\
\gamma_b b(\theta)+\gamma_{\bar{b}}\bar{b}(\theta)=\frac{\gamma_b+\gamma_{\bar{b}}}{2} (b(\theta)+\bar{b}(\theta))
\end{gather*}
regardless of whether the functions $V_1$ and $V_2$ are even or not.
In view of \eqref{eq:sumzero}, this gives \eqref{eq:ga4}. It follows that
\begin{equation}\label{eq:cform}
a(\theta)+\bar{a}(\theta)=-\frac{\gamma_b+\gamma_{\bar{b}}}{\gamma_a+\gamma_{\bar{a}}}\int_0^1\frac{W'_2(\sqrt{E-\theta}s)+\bar{W}'_2(\sqrt{E-\theta}s)}{\sqrt{2}\sqrt{1-s^2}}ds.
\end{equation}

\medskip

\noindent
\textbf{Part (a).}
Assume that $0<E_1<E_2$ are two elements of $\mathfrak E$. In view of  \eqref{eq:cform}, there are two positive numbers
$\gamma_1,\gamma_2$ such that
\begin{gather*}
a(\theta)+\bar{a}(\theta)=\gamma_1c(E_1-\theta)\text{ for }\theta\in(0,E_1),\
a(\theta)+\bar{a}(\theta)=\gamma_2c(E_2-\theta)\text{ for }\theta\in(0,E_2),
\end{gather*}
where
\[c(\theta)=\int_0^1\frac{W'_2(\sqrt{\theta}s)+\bar{W}'_2(\sqrt{\theta}s)}{\sqrt{2}\sqrt{1-s^2}}ds\text{ for }\theta>0.\]
%More precisely, if \eqref{eq:ga3} holds for some $\theta\in\mathcal{E}$, then $\gamma_1=-\tfrac{\gamma_b+\gamma_{\bar{b}}}{\gamma_a+\gamma_{\bar{a}}}$.
Hence, for every $\theta\in(0,E_1)$, we have
\[ (a+\bar{a})(\theta+E_2-E_1)=\gamma_2c(E_2-(\theta+E_2-E_1))=\gamma_2c(E_1-\theta)=\frac{\gamma_2}{\gamma_1}( a+\bar{a})(\theta).\]
Hence, $a+\bar{a}$ is geometrically $E$-quasi-periodic for $E:=E_2-E_1>0$, see Definition~\ref{def:qp}. In view of Theorem~\ref{thm:const}, we get
 $V_1,V_2\in\mathcal{SP}$.

\medskip

\noindent
\textbf{Part (b).}
Suppose that $V_1(x)=\omega V_2(\tau x)$ for some $\omega>0$ and $\tau\neq 0$. Then $(V_1)_*(x)=\sgn(\tau)\sqrt{\omega} (V_2)_*(\tau x)$.
Hence
\begin{equation}\label{eq:W_1'}
W_2(x)=\tau W_1(\sgn(\tau)\sqrt{\omega}x)\quad\text{and}\quad W'_2(x)=|\tau| \sqrt{\omega}W'_1(\sgn(\tau)\sqrt{\omega}x).
\end{equation}
In view of \eqref{eq:cform} and \eqref{eq:formofa}, for every $\theta\in(0,E)$,
\begin{align*}
-\frac{\gamma_{a}+\gamma_{\bar{a}}}{\gamma_{b}+\gamma_{\bar{b}}}(a(\theta)+\bar{a}(\theta))&
=\int_0^1\frac{W'_2(\sqrt{E-\theta}s)+{W}'_2(-\sqrt{E-\theta}s)}{\sqrt{2}\sqrt{1-s^2}}ds\\
&=|\tau| \sqrt{\omega}\int_0^1\frac{W'_1(\sqrt{\omega(E-\theta)}s)+{W}'_1(-\sqrt{\omega(E-\theta)}s)}{\sqrt{2}\sqrt{1-s^2}}ds\\
&=|\tau| \sqrt{\omega}\big(a(\omega(E-\theta))+\bar{a}(\omega(E-\theta))\big).
\end{align*}
Setting $\theta:=\tfrac{\omega}{1+\omega}E\in(0,E)$, we have $\omega(E-\theta)=\theta$. As $a$, $\bar{a}$ are positive, by \eqref{eq:W_1'},
it follows that
\[\frac{\gamma_{a}+\gamma_{\bar{a}}}{\gamma_{b}+\gamma_{\bar{b}}}=
-|\tau|\sqrt{\omega}=-\frac{W'_2(0)}{W'_1(0)}=-\sqrt{\frac{V''_1(0)}{V''_2(0)}}.\]

\noindent
{\textbf{Part (c).} If additionally  $\omega\neq 1$, then $(a+\bar{a})(\theta)=(a+\bar{a})(\omega(E-\theta))$, so
$a+\bar{a}$ is invariant under the action of the map $\theta\mapsto \omega(E-\theta)$ for which $\tfrac{\omega}{1+\omega}E$ is forward (if $\omega<1$) or backward (if $\omega>1$) attracting fixed point. It follows that $a+\bar{a}$ is constant. Therefore $V_1\in\mathcal{SP}$ and automatically $V_2\in\mathcal{SP}$.}
%More precisely, there exists $c\in\R$ such that
% $a(\theta)+\bar{a}(\theta)=\gamma_1c(\theta)=\gamma_1c$ for all $\theta>0$.
%Applying again Theorem~\ref{thm:const} to the potential $V_2$, we also have $V_2\in\mathcal{SP}$.
%
%Since for every $\theta>0$ we have
%\begin{gather*}
%\gamma_1c= a(\theta)+\bar{a}(\theta)=\int_0^1\frac{W'_1(\sqrt{\theta}s)+\bar{W}'_1(\sqrt{\theta}s)}{\sqrt{2}\sqrt{1-s^2}}ds,\\
%c=c(\theta)=\int_0^1\frac{W'_2(\sqrt{\theta}s)+\bar{W}'_2(\sqrt{\theta}s)}{\sqrt{2}\sqrt{1-s^2}}ds,
%\end{gather*}
%passing to the limit  $\theta \searrow 0$, we obtain
%\[ W'_1(0)\frac{\pi}{\sqrt{2}}=\gamma_1c\quad\text{ and }\quad W'_2(0)\frac{\pi}{\sqrt{2}}=c.\]
%Hence
%\[\frac{\gamma_a+\gamma_{\bar{a}}}{\gamma_b+\gamma_{\bar{b}}}=-\frac{1}{\gamma_1}=-\frac{W'_2(0)}{W'_1(0)}=-\sqrt{\frac{V''_1(0)}{V''_2(0)}}.\]
\end{proof}

%\begin{remark}
%In view of Theorem~\ref{thm:const}, if $V_1$ ($V_2$ resp.) in Proposition~\ref{prop:twobad} is not even, then additionally  $\gamma_{a}=\gamma_{\bar{a}}\neq 0$ ($\gamma_{b}=\gamma_{\bar{b}}\neq 0$ resp.).
%\end{remark}

\begin{remark}\label{rmk:SP}
Assume that $V_1,V_2\in \mathcal{SP}$. In view of \eqref{eq:SP}, for every $E>0$ and for all $\theta\in (0,E)$, we have
\[a(\theta)+\bar{a}(\theta)=\int_0^1\frac{W'_1(\sqrt{\theta}s)+\bar{W}'_1(\sqrt{\theta}s)}{\sqrt{2}\sqrt{1-s^2}}ds=W'_1(0)\frac{\pi}{\sqrt{2}}
=\frac{\pi}{\sqrt{V''_1(0)}}\]
and
\[b_E(\theta)+\bar{b}_E(\theta)=\int_0^1\frac{W'_2(\sqrt{E-\theta}s)+\bar{W}'_2(\sqrt{E-\theta}s)}{\sqrt{2}\sqrt{1-s^2}}ds=
W'_2(0)\frac{\pi}{\sqrt{2}}=\frac{\pi}{\sqrt{V''_2(0)}}.\]
Therefore,
\begin{equation}\label{eq:spab}
a(\theta)+\bar{a}(\theta)=\sqrt{\frac{V''_1(0)}{V''_2(0)}}(b_E(\theta)+\bar{b}_E(\theta))\text{ for all $E>0$ and $\theta\in(0,E)$}.
\end{equation}
%This shows that the condition \eqref{eq:ga3} is actually equivalent to $V_1,V_2\in \mathcal{SP}$.
\end{remark}

\section{General criterion for the absence of resonance}\label{sec:genequ}
In this section, we present and prove a simple criterion (Theorem~\ref{thm:main-min}) to show the absence of resonance for billiard flows on rectilinear polygons in directions $\pm\pi/4,\pm 3\pi/4$.
A rectilinear polygon $P\in\mathcal{RP}$ is \emph{resonant} if the billiard flow (in directions $\pm\pi/4,\pm 3\pi/4$) has an orbit joining two corners (possibly the same).
Returning to the Hamiltonian flow $(\varphi^P_t)_{t\in\R}$ describing the behavior of a particle
in the polygon $P\in\mathcal{RP}$ and its billiard representation, described in Section~\ref{sec:osctobil}, we can formulate the following principle:
\begin{align}
\begin{aligned}
\text{A pair $(E,\theta)$  is resonant for the flow $(\varphi^P_t)_{t\in\R}$}\\
\text{if and only if the polygon $P_{E,\theta}$ is resonant. }
\end{aligned}
\end{align}
This principle, together with Theorem~\ref{thm:main-min}, will allow us in Section~\ref{sec:appl} to prove the non-resonance of specific energy levels.

\medskip

Recall that for any $P\in\mathcal{RP}$, the finite sets $X^+_P$, $X^-_P$ collect the parameters of the vertical sides of $P$  and $Y^+_P$, $Y^-_P$ collect the parameters of the horizontal sides.
Let $X_P:=X^+_P \cup X^-_P$ and $Y_P:=Y^+_P \cup Y^-_P$. Moreover, $x^+_P$, $x^-_P$, $y^+_P$, $y^-_P$ are parameters of the extreme sides of $P$. Assume that $x^+_P\notin X^-_P$, $x^-_P\notin X^+_P$, $y^+_P\notin Y^-_P$, $y^-_P\notin Y^+_P$.

\begin{theorem}\label{thm:main-min}
Suppose that
for any choice of integers $n_{{x}}$, ${x}\in {X}_{{P}}\setminus\{{0}\}$ and $m_{{y}}$, ${y}\in{Y}_{{P}}\setminus\{{0}\}$
such that not all of them are zero and $n_{{x}_P^+}\cdot n_{{x}_P^-}\geq 0$, $m_{{y}_P^+}\cdot m_{{y}_P^-}\geq 0$, we have
\begin{equation}\label{cond:i}
\sum_{{x}\in{X}_{{P}}\setminus\{0\}}n_{{x}}{x}-\sum_{{y}\in{Y}_{{P}}\setminus\{{0}\}}m_{{y}}{y}\neq 0.
\end{equation}
Then the polygon $P$ is not resonant.
\end{theorem}

{The above result can be interpreted as a minimality criterion for the flow when the parameters of the polygon are rationally independent, or as an analogue of the classical Keane \cite{Kea} result on minimality. However, our condition is slightly weaker than full rational independence, since it allows for rational relations among the extremal parameters of the polygon, which is necessary when some potentials are even.}

The proof of this result uses the same arguments as in  Theorem~4.2 in \cite{Fr}. Nevertheless, we include the proof for completeness.

\subsection{Short introduction to translation surfaces}
Since the proof of Theorem~\ref{thm:main-min} uses  translation surface tools, we give a short introduction to this subject in this section. For further background material, we refer the reader to \cite{ViB}, \cite{Yo} and \cite{ZoFlat}.

\medskip

A \emph{translation surface} $(M,\omega)$ is a compact  orientable topological surface $M$, together with a finite set
of points $\Sigma$ (called \emph{singular} points) and an atlas of charts $\omega=\{\zeta_\alpha:U_\alpha\to \C:\alpha\in\mathcal{A}\}$ on $M\setminus \Sigma$ such that every transition map
$\zeta_\beta\circ\zeta^{-1}_\alpha:\zeta_\alpha(U_\alpha\cap U_\beta)\to \zeta_\beta(U_\alpha\cap U_\beta)$ is a translation, i.e., for every connected component $C$ of $U_\alpha\cap U_\beta$, there exists
$v_{\alpha,\beta}^C\in \C$ such that  $\zeta_\beta\circ\zeta^{-1}_\alpha(z)=z+v^C_{\alpha,\beta}$ for $z\in \zeta_\alpha^{-1}(C)$. All points in $M\setminus \Sigma$ are called \emph{regular}. For any point $x\in M$, the translation structure $\omega$ allows us to define the total angle around $x$.
If $x$ is regular, then the total angle is $2\pi$. If $\sigma$ is singular, then the total angle is $2\pi(k_\sigma+1)$, where $k_\sigma\in\N$ is the multiplicity of $\sigma$. Then $\sum_{\sigma\in\Sigma}k_\sigma=2g-2$, where $g$ is the genus of the surface $M$.
%Denote by $\underline{\Sigma}\subset \Sigma$ the subset of singular points with non-zero multiplicity.
Singular points with zero multiplicity are sometimes called fake singularities.
These points are sometimes treated instead as regular points.

\medskip

For any direction $\vartheta\in\R/2\pi\Z$, let $X_\vartheta$ be a tangent vector field on $M\setminus{\Sigma}$ which is the pullback of the unit constant vector field $e^{i\vartheta}$ on $\C$ through the charts of the atlas.
Since the derivative of any transition map is the identity, the vector field  $X_\vartheta$ is well defined on $M\setminus{\Sigma}$.
Denote by $(\psi^\vartheta_t)_{t\in\R}$ the corresponding local flow, called the translation flow on $(M,\omega)$ in direction $\vartheta$.
%The flow preserves the measure $\lambda_{\omega}$ which is the pullback of the Lebesgue measure on $\C$.

%We distinguish the horizontal flow $(\psi^v_t)_{t\in\R}$, i.e., for $\vartheta=\pi/2$.
%
%For every $\vartheta\in\R/2\pi\Z$ and a translation surface $(M,\omega)$ denote by $(M,e^{i\vartheta}\omega)$ the rotated translation surface,
%i.e., the  new charts in $e^{i\vartheta}\omega$ are defined by postcomposition of charts from $\omega$ with the rotation by $\vartheta$.
%Then the flow $(\psi^\vartheta_t)_{t\in\R}$ on $(M,\omega)$ coincides with the vertical flow $(\psi^v_t)_{t\in\R}$ on $(M,e^{i(\tfrac{\pi}{2}-\vartheta)}\omega)$.

A \emph{saddle connection} in direction $\vartheta$ is an orbit segment of $(\psi^\vartheta_t)_{t\in\R}$ that joins
 singularities from ${\Sigma}$ (possibly the same one) and has no interior singularities.

\begin{remark}\label{rem:yoc}
By Corollary~5.4 in \cite{Yo}, if $(M,\omega)$ has no saddle connection in direction $\vartheta$, then the flow $(\psi^\vartheta_t)_{t\in\R}$ is \emph{minimal} on every connected component of $M$, i.e., every orbit (all of them are semi-infinite or double-infinite) is dense in the connected component of $M$ containing the orbit. Moreover, if $(\psi^\vartheta_t)_{t\in\R}$ is not minimal and has a (regular) periodic orbit, then it is surrounded by a maximal open cylinder in $(M,\omega)$ consisting of periodic orbits that are homotopic to the original one.
The boundary of the cylinder consists of a chain of saddle connections, see \cite[\S 5.2]{Yo}.
\end{remark}

\subsection{Partitions of translation surfaces into polygons}\label{sec:parttranssurf}
In this section, we recall some basic concepts introduced in \cite{Fr}.

\begin{definition}\label{def:part}
Let $(M,\omega)$ be a compact translation surface. A finite partition $\mathcal{P}=\{P_\alpha:\alpha\in\mathcal A\}$ of $M$ is called a \emph{partition into polygons} if
\begin{itemize}
\item[$(i)$] each $P_\alpha$, $\alpha\in\mathcal A$ is a closed connected  subset of $M$ and  $\bigcup_{\alpha\in \mathcal A}P_\alpha=M$;
\item[$(ii)$]
for every $\alpha\in \mathcal A$, there exists a chart $\zeta_\alpha:U_\alpha\to \C$ in $\omega$ such that:
\begin{itemize}
\item $Int (P_\alpha)\subset P_\alpha\setminus\Sigma \subset U_\alpha$,
\item the restriction of $\zeta_\alpha$ to $P_\alpha\setminus\Sigma$ is a homeomorphism onto its image,
\item $\zeta_\alpha(Int (P_\alpha))\subset \C$ is the interior of a compact polygon $\widetilde{P}_{\alpha}\subset \C$ and
\item $\zeta_\alpha^{-1}:Int(\widetilde{P}_{\alpha})\to Int(P_\alpha)$ has a continuous  extension
$\bar{\zeta}^{-1}_\alpha:\widetilde{P}_{\alpha}\to P_\alpha$.
\end{itemize}
Then the $\bar{\zeta}^{-1}_\alpha$-image of any side in $\widetilde{P}_{\alpha}$ is called a side of $P_\alpha$ and the $\bar{\zeta}^{-1}_\alpha$-image of any vertex in $\widetilde{P}_{\alpha}$ is called a vertex of $P_\alpha$.
\item[$(iii)$] if $P_\alpha\cap P_\beta\neq \emptyset$, then it is the union of common sides and corners of the polygons $P_\alpha$, $P_\beta$;
\item[$(iv)$]  if $\sigma\in P_\alpha\cap \Sigma$, then $\sigma$ is a vertex of $P_\alpha$.
\end{itemize}
Let $\vartheta\in \R/2\pi\Z$. The partition $\mathcal{P}$ is \emph{$\vartheta$-admissible} if the polygons $\widetilde{P}_\alpha$, $\alpha\in\mathcal{A}$  have no sides parallel to $\vartheta$.
\end{definition}
Notice that the definition does not require $P_\alpha$ to be a topological polygon; this means $\bar{\zeta}^{-1}_\alpha:\widetilde{P}_{\alpha}\to P_\alpha$ does not have to be a homeomorphism.
Distinct vertices of $\widetilde{P}_{\alpha}$ can be mapped to the same singularities in $P_\alpha\cap\Sigma$.

\begin{definition}
For any partition into polygons $\mathcal{P}=\{P_\alpha:\alpha\in\mathcal A\}$ of the translation surface $(M,\omega)$, let:
\begin{itemize}
\item ${D}={D}(\omega,\mathcal P)$ be the set of all sides in $\mathcal P$;
\item $V=V(\omega,\mathcal P)$ be the set of triples (vertices) $(\sigma,\widetilde{\sigma},\alpha)\in{\Sigma}\times\C\times \mathcal A$ for which $\sigma\in P_\alpha\cap {\Sigma}$ and $\widetilde{\sigma}$ is a vertex
of $\widetilde{P}_\alpha$ such that $\bar{\zeta}^{-1}_\alpha(\widetilde{\sigma})=\sigma$.
\end{itemize}
We will call the pair $(D,V)$  \emph{the combinatorial data} of the partition $\mathcal{P}$.
\end{definition}

\begin{definition}
Assume that the partition $\mathcal{P}$ is $\vartheta$-admissible.
Suppose that $e\in D$ is a common side of $P_\alpha$ and $P_\beta$ and suppose that every  orbit in direction $\vartheta$ through the side $e$ passes from $P_\alpha$ to $P_\beta$.
Then, the displacement $\mathfrak{D}^{\vartheta}_{\omega}(e):=\zeta_\alpha(x)-\zeta_\beta(x)$ does not depend on the choice of $x\in e$.
For any $v=(\sigma,\widetilde{\sigma},\alpha)\in V$, let $\mathfrak{B}^{\vartheta}_{\omega}(v)=-\widetilde{\sigma}$
and
$\mathfrak{E}^{\vartheta}_{\omega}(v)=\widetilde{\sigma}$.
\end{definition}

Suppose that $\mathcal{P}$ is a $\vartheta$-admissible partition of $(M,\omega)$ and $\gamma$ is its saddle connection in direction $\vartheta$ for which $\tau=|\gamma|>0$ is its length.
By Theorem~2.12 in \cite{Fr}, we have
\begin{equation}\label{eq:saddle}
\tau e^{i\vartheta}=\mathfrak{B}^{\vartheta}_{\omega}(v_+)+\mathfrak{E}^{\vartheta}_{\omega}(v_-)+\sum_{e\in D}n_e\mathfrak{D}^{\vartheta}_{\omega}(e),
\end{equation}
where
\begin{itemize}
\item $v_+=(\sigma_+,\widetilde{\sigma}_+,\alpha)$ is such that $\sigma_+\in{\Sigma}\cap P_\alpha$ is the beginning of $\gamma$, an initial segment of $\gamma$ runs in $P_\alpha$,
and $\widetilde{\sigma}_+$ is a vertex of $\widetilde{P}_\alpha$, which is the beginning of the $\zeta_\alpha$-image of the initial segment;
\item $v_-=(\sigma_-,\widetilde{\sigma}_-,\beta)$ is such that $\sigma_-\in{\Sigma}\cap P_\beta$ is the end of $\gamma$, a final segment of $\gamma$ runs in $P_\beta$,
and $\widetilde{\sigma}_-$ is a vertex of $\widetilde{P}_\beta$, which is the end of the $\zeta_\beta$-image of the final segment;
\item for any $e\in D$, $n_e$ the number of times $\gamma$ crosses $e$.
%is the number of hits of the side $e$ by the saddle connection $\gamma$.
\end{itemize}
This observation has two consequences: it excludes saddle connections in Theorem~\ref{thm:main-min} and yields the following resonance criterion.

\begin{lemma}\label{thm:sc}
Assume that  $\mathcal P$ is a $\vartheta$-admissible partition of $(M,\omega)$.
Suppose that $\gamma$ is a saddle connection on $(M,\omega)$ in direction $\vartheta$ such that
\[\ip  e^{-i\vartheta_0}( \mathfrak{B}^{\vartheta}_{\omega}(v_+)+ \mathfrak{E}^{\vartheta}_{\omega}(v_-)+\sum_{e\in D}n_e \mathfrak{D}^{\vartheta}_{\omega}(e))= 0.\]
Then $\vartheta=\vartheta_0\operatorname{mod}\pi$.
\end{lemma}

%
%\begin{remark}
%Let $J\subset\R^d$ be an open set and $J\ni\theta\mapsto ((M,\omega(\theta)),\mathcal P(\theta))$ be a $C^\infty$-map of translation surfaces equipped $\pi/2$-partition into polygons. Suppose that the vertical on $(M,\omega(\theta_0))$ is not minimal for some $\theta_0\in J$.
%Then $(M,\omega(\theta_0))$ has a vertical saddle connection $sc_{\theta_0}$.
%Let us consider $C^\infty$ maps $f^{sc_{\theta_0}},g^{sc_{\theta_0}}:J\to\C$ given by
%\[f^{sc_{\theta_0}}(\theta)=\mathfrak{B}^{\pi/2}_{\omega(\theta)}(\sigma_+,\widetilde{\sigma}_+,\alpha),\
%g^{sc_{\theta_0}}(\theta)=\mathfrak{E}^{\pi/2}_{\omega(\theta)}(\sigma_-,\widetilde{\sigma}_-,\beta),\]
%where $(\sigma_+,\widetilde{\sigma}_+,\alpha)\in V$ represents the beginning of $sc_{\theta_0}$ and
%$(\sigma_-,\widetilde{\sigma}_-,\beta)\in V$ represents the end of $sc_{\theta_0}$.
%\end{remark}

\subsection{From billiards to translation surfaces}
For any polygon ${P}\in \mathcal{RP}$, the directional billiard flow on ${P}$ in directions $\pm\pi/4, \pm 3\pi/4$ acts on the union of four copies of ${P}$, denoted by  $P_{\pi/4}$,
$P_{-\pi/4}$, $P_{3\pi/4}$, $P_{-3\pi/4}$. Each copy $P_{\vartheta}$ for $\vartheta\in \{\pm\pi/4$, $\pm 3\pi/4\}$ represents
all unit vectors flowing in the same direction $\vartheta$. After applying the horizontal (denoted by $\gamma_h$) or vertical (denoted by $\gamma_v$) reflection (or both) to each copy separately, we can arrange all unit vectors to flow in the same direction $\pi/4$.
More precisely, after such transformations, all unit vectors in
\begin{gather*}
{P}_{++}:={P}_{\pi/4},\ P_{+-}:=\gamma_hP_{-\pi/4},\
{P}_{-+}:=\gamma_vP_{3\pi/4},\ P_{--}:=\gamma_h\circ\gamma_v P_{-3\pi/4}
\end{gather*}
follow the same direction $\pi/4$. By gluing the corresponding sides of these four polygons, we get a compact orientable surface $M$ with a translation structure $\omega$ inherited from the Euclidean plane, see Figure~\ref{fig:surface}.
Moreover, the directional billiard flow on $P$ in directions $\pm\pi/4$, $\pm 3\pi/4$ is conjugated to the translation flow $(\psi^{\pi/4}_t)_{t\in\R}$ on the translation surface $(M,\omega)$. This is an example of using the so-called unfolding procedure coming from \cite{Fox-Ker} and \cite{Ka-Ze}. Let us mention that the number of connected components of the polygon $P$ and the surface $M$ is the same.

Notice that for every convex vertex of the polygon $P$, the total angle of the corresponding point on $(M,\omega)$ is $2\pi$, so it is a fake singularity.
On the other hand, for every concave vertex of $P$, the total angle of the corresponding point is $6\pi$, so it is a singular point (from ${\Sigma}$)
with multiplicity $2$.

\begin{remark}\label{rem:cyl}
Note that the existence of an orbit segment joining corners for the billiard flow ($\pm\pi/4,\pm3\pi/4$) on $P$ (this is resonance) is equivalent to the existence of a saddle connection in direction $\pi/4$ on $(M,\omega)$.
Suppose that  resonance is revealed on $P$ by connecting two convex vertices. Then, the resonant orbit on $P$ corresponds on $(M,\omega)$ to the union of two saddle connections connecting two fake singularities. If we treat all the fake singularities on $(M,\omega)$ as regular points, then these two orbit segments together form a regular periodic orbit. By Remark~\ref{rem:yoc}, this orbit is surrounded by a cylinder of regular periodic orbits whose boundary consists of saddle connections connecting true singularities. It follows that $P$ has plenty of regular periodic orbits and some resonant orbits joining concave corners.
\end{remark}

\begin{figure}[h]
\includegraphics[width=0.8 \textwidth]{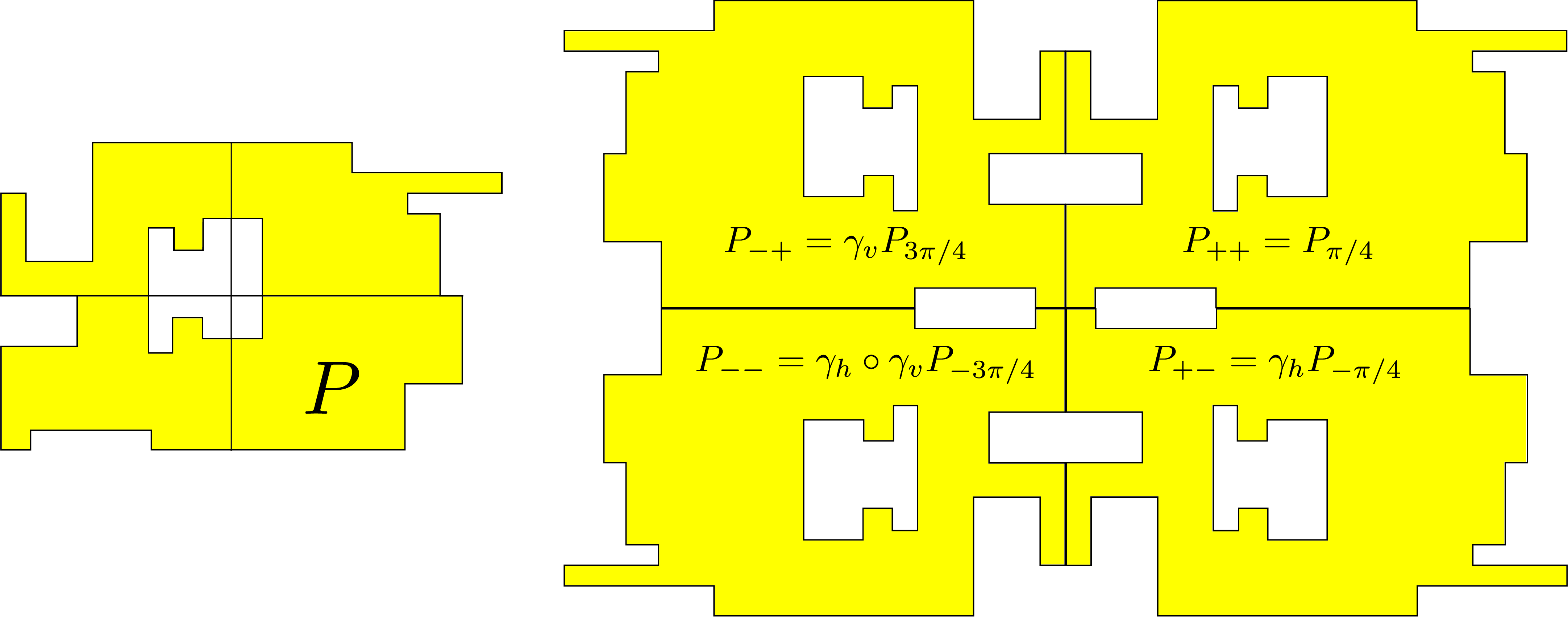}
\caption{The billiard table $P$ and  the translation surface $(M,\omega)$ - connected case. }
\label{fig:surface}
\end{figure}

The translation surface $(M,\omega)$ has a natural partition into four rectilinear polygons
\[\mathcal{P}=\{P_{++}, P_{+-},P_{-+},P_{--}\},
\]
so that $\mathcal{A}=\{++,+-,-+,--\}$.
Moreover,
\begin{gather*}
{X}_{{P}_{++}}={X}_{{P}_{+-}}={X}_{
{P}_{-+}}={X}_{{P}_{--}}={X}_{{P}},\
{Y}_{{P}_{++}}={Y}_{{P}_{+-}}={Y}_{
{P}_{-+}}={Y}_{{P}_{--}}={Y}_{{P}}.
\end{gather*}
We split the set of all sides $D$ of the partition $\mathcal{P}$ into the subsets of the vertical sides $D_v$ and the horizontal sides $D_h$.
We also distinguish the set of sides derived from the extreme sides of polygons $P_{++}, P_{+-},P_{-+},P_{--}$, which we denote by $D_{ext}$.

\medskip

Let $\vartheta$ be any direction in $(0,\pi/2)$. Then the partition $\mathcal{P}$ is $\vartheta$-admissible.
Suppose that $e\in D_v$ is a vertical side of $\mathcal{P}$, a common side of ${P}_{\alpha}$ and ${P}_{\beta}$ for some $\alpha,\beta\in\mathcal{A}$ such that ${P}_{\alpha}$ is on the left side of $e$.
For every $s\in e$, we have $\zeta_\beta(s)=-\overline{\zeta_\alpha(s)}$. Hence, the corresponding displacement
\begin{equation*}
\mathfrak{D}^{\vartheta}_{\omega}(e)=2\rp \zeta_\alpha(s)=\pm 2x(e)\text{ for some }{x(e)}\in {X}_{{P}}.
\end{equation*}
If $\rp \zeta_\alpha(s)\geq 0$, then the side $e\in D_v$ is \emph{positively oriented}; otherwise, it is \emph{negatively oriented}. We denote the set of positively (negatively) oriented vertical sides by $D_v^+$ ($D_v^-$ resp.).
Suppose that  $e\in D_v\cap D_{ext}$. Then, $x(e)=x^\pm_P$ and $e$ is located on the right side of the polygon  $P_\alpha$, so it has a non-negative first coordinate. It follows that every such (extreme) side is positively oriented.

Suppose that $e\in D_h$ is a horizontal side of $\mathcal{P}$, a common side of ${P}_{\alpha}$ and ${P}_{\beta}$ for some $\alpha,\beta\in\mathcal{A}$ such that ${P}_{\alpha}$ is below $e$.
For every $s\in e$, we have  $\zeta_\beta(s)=\overline{\zeta_\alpha(s)}$ and
\begin{equation*}
\mathfrak{D}^{\vartheta}_{\omega}(e)=2i\ip  \zeta_\alpha(s)=\pm 2iy(e)\text{ for some }y(e)\in {Y}_{{P}}.
\end{equation*}
If $\ip \zeta_\alpha(s)\geq 0$, then the side $e\in D_h$ is \emph{positively oriented}; otherwise, it is \emph{negatively oriented}. We denote the set of positively (negatively) oriented horizontal sides by $D_h^+$ ($D_h^-$ resp.). As for the vertical sides, one can easily deduce that all $e\in D_h\cap D_{ext}$ are also positively oriented with $y(e)=y^\pm_P$. In summary, we have
\begin{gather}
\label{eq:D}
\mathfrak{D}^{\vartheta}_{\omega}(e)=\pm 2x(e)\text{ if }e\in D_v^\pm;\quad
\mathfrak{D}^{\vartheta}_{\omega}(e)=\pm 2iy(e)\text{ if }e\in D_h^\pm;\\
\label{eq:ext}
D_{ext}\subset D_v^+\cup D_h^+.
\end{gather}

Note that the local coordinates of any vertex in $\mathcal{P}$ are of the form
\[x+iy\text{ or }-x+iy\text{ or }
x-iy\text{ or }-x-iy\]
for some $x\in {X}_{{P}}$ and ${y}\in {Y}_{{P}}$.
Hence, for every $v=(\sigma,\tilde{\sigma},\alpha)\in V$, we have
\[-\mathfrak{B}^{\vartheta}_{\omega}(v)=\mathfrak{E}^{\vartheta}_{\omega}(v)= \zeta_\alpha(\tilde \sigma)=\pm x\pm iy\text{ for some }x\in {X}_{{P}}, \ {y}\in {Y}_{{P}}.\]
Denote by $e_v(v)\in D_v$ the vertical side and by $e_h(v)\in D_h$  the horizontal side in $P_\alpha$ emanating from the vertex $v$. Then
\begin{align}\label{eq:BE}
\begin{aligned}
\mathfrak{B}^{\vartheta}_{\omega}(v)&=-\zeta_\alpha(\tilde \sigma)= \vep_1^b(v) x(e_v(v))+ \vep_2^b(v)iy(e_h(v)),\\
\mathfrak{E}^{\vartheta}_{\omega}(v)&=\zeta_\alpha(\tilde \sigma)=\vep_1^e(v) x(e_v(v))+ \vep_2^e(v)iy(e_h(v))
\end{aligned}
\end{align}
where the signs $\vep_1^b(v),\vep_2^b(v),\vep_1^e(v),\vep_2^e(v)\in\{\pm\}$ depend on the location of the vertex $v$ in the polygon $P_\alpha$.

Suppose that $\sigma$ is the end point of an orbit segment in the direction $\vartheta$ that runs in $P_\alpha$. If $e_v(v)$ is an extreme vertical side, then the vertex $\tilde{\sigma}$ is convex, so $\sigma$ is a fake singularity.
Moreover, $e_v(v)$ is located on the right side of the polygon $P_\alpha$, so it has a non-negative first coordinate. It follows that
\begin{equation}\label{eq:ext1}
\vep_1^e(v)=1\text{ if }e_v(v)\in D_{ext}.
\end{equation}
Similarly, if $e_h(v)$ is an extreme horizontal side, then $e_h(v)$ is located above the polygon $P_\alpha$ and
\begin{equation}\label{eq:ext2}
\vep_2^e(v)=1\text{ if }e_h(v)\in D_{ext}.
\end{equation}

Suppose that $\sigma$ is the initial point of an orbit segment in the direction $\vartheta$ that runs in $P_\alpha$. If $e_v(v)$ is an extreme vertical side, then $\sigma$ is again a fake singularity, and
 $e_v(v)$ is located on the left side of the polygon $P_\alpha$, so it has a negative first coordinate. It follows that
\begin{equation}\label{eq:ext3}
\vep_1^b(v)=1\text{ if }e_v(v)\in D_{ext}.
\end{equation}
Similarly, if $e_h(v)$ is an extreme horizontal side, then $e_h(v)$ is located below the polygon $P_\alpha$ and
\begin{equation}\label{eq:ext4}
\vep_2^b(v)=1\text{ if }e_h(v)\in D_{ext}.
\end{equation}

\medskip

Suppose that $\gamma$ is a saddle connection on $(M,\omega)$ in direction $\vartheta\in(0,\pi/2)$ such that $\tau>0$ is its length, $v_+\in V$ represents its beginning, and $v_-\in V$ represents its end. In view of \eqref{eq:saddle},
\[\tau e^{i\vartheta}=\mathfrak{B}^{\vartheta}_{\omega}(v_+)+\mathfrak{E}^{\vartheta}_{\omega}(v_-)+\sum_{e\in D}n_e\mathfrak{D}^{\vartheta}_{\omega}(e),\]
where $n_e$ is the intersection number of $e$ and $\gamma$. In view of \eqref{eq:D} and \eqref{eq:BE},
\begin{align}\label{eq:sumsum}
\begin{aligned}
\tau &e^{i\vartheta}=\sum_{e\in D_v^+}2n_ex(e)-\sum_{e\in D_v^-}2n_ex(e)+\vep_1^b(v_+)x(e_v(v_+))+\vep_1^e(v_-)x(e_v(v_-))\\
&+i\Big(\sum_{e\in D_h^+}2n_ey(e)-\sum_{e\in D_h^-}2n_ey(e)+\vep_2^b(v_+)y(e_h(v_+))+\vep_2^e(v_-)y(e_h(v_-))\Big).
\end{aligned}
\end{align}

\begin{proof}[Proof of Theorem~\ref{thm:main-min}]
Suppose, contrary to our claim, that $P$ has a resonant orbit. By Remark~\ref{rem:cyl}, the corresponding translation surface $(M,\omega)$ has a saddle connection $\gamma$ in direction $\vartheta=\pi/4$.
In view of \eqref{eq:sumsum}, we have
\begin{equation}\label{eq:tautau}
\tau e^{i\pi/4}=\sum_{x\in X_P}m^1_x x+i\sum_{y\in Y_P}m^2_y y=\sum_{x\in X_P\setminus\{0\}}m^1_x x+i\sum_{y\in Y_P\setminus\{0\}}m^2_y y,
\end{equation}
where
\begin{align*}
m^1_x&=\sum_{e\in D_v^+, x(e)=x}2n_e-\sum_{e\in D_v^-, x(e)=x}2n_e+\vep_1^b(v_+)\delta_{x(e_v(v_+)),x}+\vep_1^e(v_-)\delta_{x(e_v(v_-)),x},\\
m^2_y&=\sum_{e\in D_h^+, y(e)=y}2n_e-\sum_{e\in D_h^-, y(e)=y}2n_e+\vep_2^b(v_+)\delta_{y(e_h(v_+)),y}+\vep_2^e(v_-)\delta_{y(e_h(v_-)),y}.
\end{align*}
where $\delta$ is the standard Kronecker delta.

Let $e$ be a vertical side with $x(e)=x^+_P$. By assumption, we have $e\in D_{ext}$. By \eqref{eq:ext}, this gives $e\in D_v^+$. Moreover, if $e=e_v(v_+)$ or  $e=e_v(v_-)$, then by \eqref{eq:ext1} and \eqref{eq:ext3},
$\vep_1^b(v_+)=1$ or $\vep_1^e(v_-)=1$, respectively. It follows that
\[m^1_{x_P^+}=\sum_{e\in D_v^+, x(e)=x_P^+}2n_e+\vep_1^b(v_+)\delta_{x(e_v(v_+)),x_P^+}+\vep_1^e(v_-)\delta_{x(e_v(v_-)),x_P^+}\geq 0.\]
The same argument also show that all $m^1_{x_P^-}$, $m^2_{y_P^+}$, $m^2_{y_P^-}$ are non-negative integers.

In view of \eqref{eq:tautau}, at least one of $m^1_x$ for ${x}\in{X}_{P}\setminus\{{0}\}$ and $m^2_{{y}}$ for ${y}\in{Y}_{{P}}\setminus\{{0}\}$ is non-zero, and
\[\sum_{x\in X_P\setminus\{0\}}m^1_x x=\sum_{y\in Y_P\setminus\{0\}}m^2_y y,\]
which contradicts the assumptions of the theorem.
\end{proof}

\begin{corollary}\label{cor:EE}
For any pair $V_1,V_2:\R\to\R_{\geq 0}$ of $\mathcal{UM}$-potentials with $\deg(V_1,0)=\deg(V_2,0)=2$ and any $P\in\mathcal{RP}$, we have $\mathcal E(P,V_1,V_2)\subset \mathfrak E(V_1,V_2)$.
\end{corollary}

\begin{proof}
Suppose that $E\notin \mathfrak E(V_1,V_2)$. Then the set $J\subset (0,E)$ of all $\theta\in (0,E)$ for which there exist  integers
$\gamma_{a}$,  $\gamma_{\bar{a}}$, $\gamma_{b}$, $\gamma_{\bar{b}}$  not all zero with
$\gamma_{a}\cdot\gamma_{\bar{a}}\geq 0$ and $\gamma_{b}\cdot\gamma_{\bar{b}}\geq 0$ such that
\begin{equation*}\label{eq:ga33}
\gamma_a a(\theta)+\gamma_{\bar{a}}\bar{a}(\theta)+\gamma_b b(\theta)+\gamma_{\bar{b}}\bar{b}(\theta)= 0
\end{equation*}
is at most countable.

We will show that for any $P\in\mathcal{RP}$, the polygon $P_{E,\theta}$ is not resonant for all but countably many $\theta\in (0,E)$. This gives $E\notin \mathcal E(P,V_1,V_2)$.
Take any $I\in \mathcal J_E$. In view of \eqref{eq:XY+-1} and \eqref{eq:XY+-2}, for  every $\theta\in I$, we have
\begin{align*}
&X^+_{P_{E,\theta}}\setminus\{x^+_{P_{E,\theta}}\}\subset\{a_\xi(\theta):\xi\in X^+_I\},\
X^-_{P_{E,\theta}}\setminus\{x^-_{P_{E,\theta}}\}\subset\{\bar{a}_\xi(\theta):\xi\in X^-_I\},\\
&x^+_{P_{E,\theta}}\in \{a(\theta)\}\cup \{a_\xi(\theta):\xi\in X^+_I\},\ x^-_{P_{E,\theta}}\in \{\bar{a}(\theta)\}\cup \{\bar{a}_\xi(\theta):\xi\in X^-_I\},\\
&Y^+_{P_{E,\theta}}\setminus\{y^+_{P_{E,\theta}}\}\subset\{b_\xi(\theta):\xi\in Y^+_I\},\
Y^-_{P_{E,\theta}}\setminus\{y^-_{P_{E,\theta}}\}\subset\{\bar{b}_\xi(\theta):\xi\in Y^-_I\},\\
&y^+_{P_{E,\theta}}\in \{b(\theta)\}\cup \{b_\xi(\theta):\xi\in Y^+_I\},\ y^-_{P_{E,\theta}}\in \{\bar{b}(\theta)\}\cup \{\bar{b}_\xi(\theta):\xi\in Y^-_I\}.
\end{align*}
Therefore, by Theorem~\ref{thm:main-min}, to prove that the billiard table $P_{E,\theta}$ is non-resonant for all but countably many $\theta\in I$,
it suffices to show that
for any choice of integers $\gamma_{a}$,  $\gamma_{\bar{a}}$,  $\gamma_{a_{\xi}}$ for $\xi\in X^+_I\setminus\{0\}$,
$\gamma_{\bar{a}_{\xi}}$ for $\xi\in X^-_I$, $\gamma_{b}$,  $\gamma_{\bar{b}}$,  $\gamma_{b_{\xi}}$ for $\xi\in Y^+_I\setminus\{0\}$,
$\gamma_{\bar{b}_{\xi}}$ for $\xi\in Y^-_I$ such that at least one number is non-zero and $\gamma_{a}\cdot \gamma_{\bar{a}}\geq 0$, $\gamma_{b}\cdot \gamma_{\bar{b}}\geq 0$,
we have
\begin{align}\label{eq:gms111}
\begin{split}
\gamma_{a}&a(\theta)+ \gamma_{\bar{a}}\bar{a}(\theta)+\sum_{\xi\in X^+_I\setminus\{0\}}\gamma_{a_{\xi}}a_{\xi}(\theta)+
\sum_{\xi\in X^-_I}\gamma_{\bar{a}_{\xi}}\bar{a}_{\xi}(\theta)\\
&+\gamma_{b}b(\theta)+ \gamma_{\bar{b}}\bar{b}(\theta)
+\sum_{\xi\in Y^+_I\setminus\{0\}}\gamma_{b_{\xi}}b_{\xi}(\theta)+
\sum_{\xi\in Y^-_I}\gamma_{\bar{b}_{\xi}}\bar{b}_{\xi}(\theta)\neq 0
\end{split}
\end{align}
for all but countably many $\theta\in I$. If $V_1$ is even, then $\bar{a}=a$ and $\bar{a}_\xi=a_\xi$, so we additionally assume that
$\gamma_{\bar{a}}=\gamma_{\bar{a}_{\xi}}=0$ for all $\xi\in X^+_I\cap X^-_I$.
Similarly, if $V_2$ is even, then  we additionally assume that
$\gamma_{\bar{b}}=\gamma_{\bar{b}_{\xi}}=0$ for all $\xi\in Y^+_I\cap Y^-_I$.

Suppose that at least one integer coefficient, $\gamma_{a_{\xi}}$ for $\xi\in X^+_I\setminus\{0\}$, or
$\gamma_{\bar{a}_{\xi}}$ for $\xi\in X^-_I$, or  $\gamma_{b_{\xi}}$ for $\xi\in Y^+_I\setminus\{0\}$, or
$\gamma_{\bar{b}_{\xi}}$ for $\xi\in Y^-_I$, is non-zero. Then, \eqref{eq:gms111} follows directly from the first part of Proposition~\ref{prop:indep0}.
If all the above $\gamma$'s are zero, then at least one integer coefficient $\gamma_{a}$,  $\gamma_{\bar{a}}$,
$\gamma_{b}$,  $\gamma_{\bar{b}}$ is non-zero. In view of the definition of the countable set $J$, we have
\begin{equation*}
\gamma_{a}a(\theta)+ \gamma_{\bar{a}}\bar{a}(\theta)+\gamma_{b}b(\theta)+ \gamma_{\bar{b}}\bar{b}(\theta)\neq 0\text{ for all }\theta\in I\setminus J,
\end{equation*}
which gives  \eqref{eq:gms111} as well.
\end{proof}

\section{On geometrically quasi-periodic maps}\label{sec:geomqper}
As we have seen in the proof of the second part of Proposition~\ref{prop:twobad}, the existence of two levels of resonant energies entails that the positive map $a+\bar{a}$ satisfies a certain condition of geometric quasi-periodicity.
\begin{definition}\label{def:qp}
We say that a map $f:(0,+\infty)\to\R$ is \emph{geometrically $\tau$-quasi-periodic} ($\tau>0$) if there exists $\gamma\neq 0$ such that
\[f(x+\tau)=\gamma f(x)\quad\text{for every}\quad x>0.\]
Then, $\gamma$ is called the exponent of quasi-periodicity. If $f:(0,+\infty)\to\R$ is additionally analytic, then it has an analytic extension to $\R$, and this extension  is  geometrically $\tau$-quasi-periodic
on $\R$.
\end{definition}

In this section, we will carry out a subtle analysis of positive analytic functions satisfying the geometric quasi-periodicity condition. We show that (see Theorem~\ref{thm:const}) if $a+\bar{a}$ is geometrically quasi-periodic, then it must be constant, and consequently, the potential is in $\mathcal{SP}$. The proof of this result is elementary, although technically involved. The main tools are of Fourier nature.

\begin{theorem}\label{thm:const}
Let $V_1,V_2:\R\to\R_{\geq 0}$ be $\mathcal{UM}$-potentials such that $\deg(V_1,0)=\deg(V_2,0)=2$.
Let $E>0$ and $\gamma>0$ be such that $a(\theta)+\bar{a}(\theta)=\gamma(b_E(\theta)+\bar{b}_E(\theta))$ for $\theta\in (0,E)$.
If the map $ a+\bar{a}:(0,+\infty)\to\R$ is geometrically quasi-periodic, then $ a+\bar{a}$ and  $ b_E+\bar{b}_E$ are constant, and $V_1,V_2\in\mathcal{SP}$.
\end{theorem}

We first prove {Lemma~\ref{lem:evenU}}, which also plays an important role in the proof of Proposition~\ref{prop:twobad}.
Let $\gamma$, $\bar{\gamma}$ be non-negative coefficients such that at least one is positive.
For any $\mathcal{UM}$-potential $V:\R\to\R_{\geq 0}$  with $\deg(V,0)=2$,  consider the analytic function $U:\R\to \R_{\geq 0}$ given by
\[U(x)=\gamma W'(x)+\bar{\gamma}\bar{W}'(x)=
\gamma W'(x)+\bar{\gamma}{W}'(-x).\]
Recall that $V_*:\R\to\R$ is a bi-analytic map such that $(V_*)^2=V$ and $W=(V_*)^{-1}$.
Then $\gamma a+\bar{\gamma}\bar{a}:(0,+\infty)\to\R$ is analytic and
\begin{equation}\label{eq:U}
\gamma a(\theta)+\bar{\gamma}\bar{a}(\theta)=\frac{1}{\sqrt{2}}\int_0^1\frac{U(\sqrt{\theta}s)}{\sqrt{1-s^2}}ds\text{ for every }\theta\geq 0.
\end{equation}
If the map  $\gamma a+\bar{\gamma}\bar{a}:(0,+\infty)\to\R$ is additionally geometrically quasi-periodic, then
it has an analytic geometrically quasi-periodic extension to $\R$.

\begin{lemma}\label{lem:evenU}
Suppose that $\gamma a+\bar{\gamma}\bar{a}:(0,+\infty)\to\R$ has an analytic extension to $[0,+\infty)$.
Then, the map $U$ is even. If, additionally, $V$ is not even, then $\gamma=\bar{\gamma}$.
\end{lemma}

\begin{proof}
As $U$ is analytic, to prove that $U$ is even, it suffices to show that
\begin{equation}\label{eq:oddder}
U^{(2n+1)}(0)=0\quad \text{for every} \quad n\geq 0.
\end{equation}
We will show that there exist analytic maps $a_n:[0,+\infty)\to\R$ for $n\geq 0$ such that for every $n\geq 0$, we have
\begin{align}
\label{eq:a2n}
a_{n}(\theta)&=\int_0^1\frac{U^{(2n)}(\sqrt{\theta}s)s^{2n}}{\sqrt{1-s^2}}ds\text{ for all }\theta>0.
\end{align}
The proof is by induction on $n$. For $n=0$, \eqref{eq:a2n} follows from \eqref{eq:U} with $a_0=\sqrt{2}(\gamma a+\bar\gamma \bar a)$.

Suppose that \eqref{eq:a2n} holds for some $n\geq 0$ with $a_n$ analytic on $[0,+\infty)$. Then, after differentiating with respect to $\theta$ and multiplying by $2\sqrt{\theta}$, we
obtain
\begin{align}
\label{eq:a2n+1}
2\sqrt{\theta}a'_{n}(\theta)&=\int_0^1\frac{U^{(2n+1)}(\sqrt{\theta}s)s^{2n+1}}{\sqrt{1-s^2}}ds\text{ for all }\theta>0.
\end{align}
Repeating the same operation, we get
\[a_{n+1}(\theta):=2a'_{n}(\theta)+4\theta a''_{n}(\theta)=\int_0^1\frac{U^{(2n+2)}(\sqrt{\theta}s)s^{2n+2}}{\sqrt{1-s^2}}ds\text{ for all }\theta>0,
\]
which gives \eqref{eq:a2n} for $n+1$  with $a_{n+1}$ analytic on $[0,+\infty)$.

Since both sides of \eqref{eq:a2n+1} are continuous on $[0,+\infty)$, for every $n\geq 0$, we have
\[0=2\sqrt{0}a'_{n}(0)=\int_0^1\frac{U^{(2n+1)}(0)s^{2n+1}}{\sqrt{1-s^2}}ds=U^{(2n+1)}(0)\int_0^1\frac{s^{2n+1}}{\sqrt{1-s^2}}ds.\]
This shows that $U$ is even.

Suppose that $V$ is not even. Then $W$ is not odd, so $W^{(2n)}(0)\neq 0$ for some $n\geq 1$.
As $\bar{W}(x)=-W(-x)$, for every $n\geq 0$, we have $\bar{W}^{(2n)}(x)=-W^{(2n)}(-x)$. Hence,
\[0=U^{(2n+1)}(0)=\gamma W^{(2n)}(0)+\bar{\gamma}\bar{W}^{(2n)}(0)=(\gamma-\bar{\gamma})W^{(2n)}(0).\]
Therefore, $\gamma=\bar{\gamma}$.
\end{proof}

%As $\gamma=\bar{\gamma}$ for non-even $V$, we always have $\gamma a+\bar{\gamma}\bar{a}=\frac{\gamma+\bar{\gamma}}{2}(a+\bar{a})$.

From now on, we always assume that $U(x)=W'(x)+\bar{W}'(x)=
W'(x)+{W}'(-x)$.
Since $U$ is analytic and even, there exists a positive analytic function $u:\R_{\geq 0}\to \R$ so that $U(x)=u(x^2)$ for all $x\in\R$ and
\[ a(\theta)+\bar{a}(\theta)=\frac{1}{\sqrt{2}}\int_0^1\frac{U(\sqrt{\theta}s)}{\sqrt{1-s^2}}ds=
\frac{1}{\sqrt{2}}\int_0^1\frac{u(\theta s^2)}{\sqrt{1-s^2}}ds=
\frac{1}{\sqrt{2}}\int_0^{\pi/2} u(\theta \sin^2s)ds.\]
Moreover, $u$ has an analytic extension $u:[-\vep,+\infty)\to\R$ for some $0<\vep\leq +\infty$ such that $u(x)>0$ for $x\geq -\vep$.

\medskip

For any convex $\Omega\subset \C$ such that $\R\subset\operatorname{Int}\Omega$, we denote by $H(\Omega)$ the Fr\'echet space of  holomorphic maps on $\Omega$
equipped with the topology given by the sequence of seminorms
\[\|f\|^{\Omega}_n=\sup\{|f(z)|:z\in \Omega,|z|\leq n\}.\]
For every $\delta>0$, let
\[B_\delta=\{z\in\C:|z|\leq \delta\},\ P_\delta=\{z\in\C:\rp z\geq -\delta\},\text{ and } S_\delta=\{z\in\C:|\ip z|\leq \delta\}.\]

\medskip

For any $\vep>0$, let us consider the linear operator $A:C^{\omega}(\R_{\geq -\vep})\to C^{\omega}(\R_{\geq -\vep})$ given by
\begin{equation}\label{def:A}
A(f)(\theta)=\int_0^{\pi/2} f(\theta \sin^2s)ds=\int_0^1\frac{f(\theta s^2)}{\sqrt{1-s^2}}ds.
\end{equation}
Then, for every $\theta\geq 0$, we have
\[A(f)(\theta)=\frac{1}{2}\int_0^\theta\frac{f(s)}{\sqrt{\theta-s}\sqrt{s}}ds.\]
In fact, for every $0<\delta\leq +\infty$, the operator $A$ extends to a continuous operator
$A:H(P_\vep\cap S_\delta)\to H(P_\vep\cap S_\delta)$ given by \eqref{def:A} for every $\theta\in S_\delta$.
Indeed, for every $n\geq 1$, we have $\|A(f)\|_n^{P_\vep\cap S_\delta}\leq \tfrac{\pi}{2}\|f\|_n^{P_\vep\cap S_\delta}$.
All these extensions are one-to-one. Indeed, if $A(f)=0$ then
\[0=A(f)^{(n)}(0)=f^{(n)}(0)\int_0^{\pi/2}\sin^{2n}(s)ds=f^{(n)}(0) \pi\frac{(2n-1)!!}{(2n)!!}\]
for every $n\geq 0$. Hence, $f^{(n)}(0)=0$ for every $n\geq 0$, so $f=0$.

\begin{remark}\label{rmk:rescal}
Note that the operator $A$ commutes with any real rescaling, i.e., for any $r\in\R$, we have
$A(f\circ r)=A(f)\circ r$, where the map $r:\R\to\R$ is the linear rescaling by $r$ introduced in Remark~\ref{rem:exN}.
\end{remark}

To prove Theorem~\ref{thm:const}, we need the following result.

\begin{proposition}\label{prop:mainqper}
If $f_1,f_2\in C^{\omega}(\R_{\geq -\vep})$ for some $\vep>0$ are such that
\begin{itemize}
\item $f_1(x), f_2(x)>0$ for all $x\geq -\vep$;
\item $A(f_1)$ is geometrically quasi-periodic;
\item there exists $x_0>0$ such that $A(f_1)(x)=A(f_2)(x_0-x)$ for all $x\in \R$,
\end{itemize}
then $f_1$,  $f_2$, $A(f_1)$, and $A(f_2)$ are constant.
\end{proposition}

\begin{remark}
In view of Remark~\ref{rmk:rescal}, after some rescaling of $f$ we can assume that  $A(f)$ is geometrically $2\pi$-quasi-periodic.
Suppose that $A(f)(\theta+2\pi)=e^{2\pi\xi}A(f)(\theta)$ for some $\xi\in\R$.
As $e^{-\xi\theta}A(f)(\theta)$ is  $2\pi$-periodic and analytic,
there exists  an exponentially decaying sequence $(c_n)_{n\in\Z}$ of complex numbers such that $c_{-n}=\bar{c}_n$ for all $n\in \Z$, and
\[A(f)(\theta)=\sum_{n\in\Z}c_ne^{(\xi+in)\theta}.\]
\end{remark}

In the next part of this section, we will calculate the pre-image of the basis elements $e^{(\xi+in)\theta}$, which will form the basis for further Fourier analysis of the function  $f$. {We start by defining two  special functions, $\rho_1$ and $\rho_2$.
Let us consider  $\rho_1\in H(\C)$ given by
\[\rho_1(z)=z\int_0^ze^{-s^2}ds=\frac{\sqrt{\pi}}{2}z\erf(z)=\frac{\sqrt{\pi}}{2}z(1-\erfc(z)),
\]
where $\erf$ is the Gauss error function and  $\erfc$ is the complementary error function.
As $\rho_1(-z)=\rho_1(z)$, there exists
%Then
%\[\rho_1(z)=\sum_{n=1}^\infty\frac{(-(\xi+ik))^{n-1}z^{2n}}{(n-1)!(2n-1)}\text{ and }|\rho_1(z)|\leq |z|^2 e^{|\rp ((\xi+ik)z^2)|}.\]
%Let
$\rho_2\in H(\C)$ such that
%be given by
%\[\rho_2(z)=\sum_{n=1}^\infty\frac{(-(\xi+ik))^{n-1}z^{n}}{(n-1)!(2n-1)}.\]
$\rho_1(z)=\rho_2(z^2)$ for all $z\in\C$ and $|\rho_2(z)|\leq |z| e^{|\rp (z)|}$.
Denote by $z\mapsto \sqrt{z}$ the principal branch of the square root. Then
\begin{equation}\label{rho2 Gamma}
\rho_2(z)=\frac{\sqrt{z}}{2}(\sqrt{\pi}-\Gamma(\tfrac{1}{2},z)),
\end{equation}
where $\Gamma$ is the upper incomplete Gamma function.
%and for every $x\leq 0$ we have
%\begin{align*}
%\rho_2(x)&=\rho_1(i\sqrt{-x})={i\sqrt{-x}}\int_0^{{i\sqrt{-x}}}e^{-is^2}ds={-\sqrt{-x}}\int_0^{{\sqrt{-x}}}e^{is^2}ds\\
%&=
%\frac{-\sqrt{-x}}{2}\int_0^{-x}\frac{e^{is}}{\sqrt{s}}ds.
%\end{align*}
For any $\xi\in\R$ and $k\in\Z$, let  $\rho_{\xi,k}\in H(\C)$ be given by
\begin{equation}\label{def:rho}
\rho_{\xi,k}(z)=\frac{2}{\pi}\big(2e^{(\xi+ik)z}\rho_2((\xi+ik)z)+1\big).
\end{equation}
Then
\begin{equation}\label{eq:gamma3neq}
|\rho_{\xi,k}(z)|\leq \frac{2}{\pi}\big(2|\xi+ik||z| e^{2|\rp((\xi+ik) z)|}+1\big)\quad\text{for all}\quad
z\in\C,
\end{equation}
and, for every $x\geq 0$, we have
\begin{align}\label{eq:rho2}
\begin{split}
\rho_{\xi,k}(x)&=\frac{2}{\pi}\big(2e^{(\xi+ik)x}(\xi+ik)\sqrt{x}\int_0^{\sqrt{x}}e^{-(\xi+ik)s^2}ds+1\big)\\
&= \frac{2}{\pi}\big(e^{(\xi+ik)x}(\xi+ik)\sqrt{x}\int_0^{x}\frac{e^{-(\xi+ik)s}}{\sqrt{s}}ds+1\big).
%\rho_2(x)=\rho_1(\sqrt{x})=\sqrt{x}\int_0^{\sqrt{x}}e^{-(\xi+ik)s^2}ds=\frac{\sqrt{x}}{2}\int_0^{x}\frac{e^{-(\xi+ik)s}}{\sqrt{s}}ds.
\end{split}
\end{align}

\begin{lemma}\label{lem:Agamma}
For all $k\in\Z$ and $\xi\in\R$, we have $A(\rho_{\xi,k})(\theta)=e^{(\xi+ik)\theta}$ for every $\theta\in\C$.
\end{lemma}

\begin{proof}
As both functions are analytic, it is enough to check that $A(\rho_{\xi,k})(\theta)=e^{(\xi+ik)\theta}$ for every $\theta>0$.
Since $\int_0^a\tfrac{ds}{\sqrt{(a-s)s}}=\pi$ for all $a>0$, using \eqref{eq:rho2} and standard Fubini arguments, one can show that $A(\rho_{\xi,k})(\theta)=e^{(\xi+ik)\theta}$  for any $\theta>0$.
%
%\begin{align*}
%e^{(\xi+ik)\theta}-1&=\int_0^\theta (\xi+ik)e^{(\xi+ik)y}dy\\
%&=\frac{1}{\pi}\int_0^\theta (\xi+ik)e^{(\xi+ik)y}\int_0^{\theta-y}\frac{1}{\sqrt{\theta-x-y}\sqrt{x}}dxdy\\
%&=\frac{1}{\pi}\int_{\{(x,y):x,y\geq 0,x+y\leq \theta\}} \frac{(\xi+ik)e^{(\xi+ik)y}}{\sqrt{\theta-x-y}\sqrt{x}}dxdy\\
%&=\frac{1}{\pi}\int_{\{(x,y):0\leq x\leq y\leq \theta\}} \frac{(\xi+ik)e^{(\xi+ik)(y-x)}}{\sqrt{\theta-y}\sqrt{x}}dxdy\\
%&=\frac{1}{\pi}\int_0^\theta \frac{(\xi+ik)e^{(\xi+ik)y}}{\sqrt{\theta-y}}\int_0^{y}\frac{e^{-(\xi+ik)x}}{\sqrt{x}}dxdy\\
%&= \int_0^\theta\frac{1}{\sqrt{(\theta-y)y}}\left(\frac{\rho_{\xi,k}(y)}{2}-\frac{1}{\pi}\right)dy=A(\rho_{\xi,k})(\theta)-1.
%\end{align*}
\end{proof}

\begin{lemma}
There are two positive constants $R$, $C$ such that for any $\xi\in\R\setminus\{0\}$, $k\in \Z$, and every $x$ at least $R/|\xi|$, we have
\begin{equation}\label{eq:gammapm00}
\Big|\rho_{\xi,k}(x)-\frac{1}{\pi(\xi+ik)x}-\frac{2}{\sqrt{\pi}}e^{(\xi+ik)x}\sqrt{(\xi+ik)x}\Big|\leq \frac{C}{|\xi+ik|^2x^2}.
\end{equation}
\end{lemma}
\begin{proof}
Summarizing the discussion in \cite[Section 4.2]{Olv} regarding the expansion of the incomplete Gamma function, there are two constants $R>0$ and $C>0$ such that for any complex $z$ with $|z|\geq R$, we have
\[\Gamma(\tfrac{1}{2},z)=\frac{e^{-z}}{\sqrt{z}}\Big(1-\frac{1}{2z}+\frac{\varepsilon_2(z)}{z^2}\Big)\quad\text{with}\quad |\varepsilon_2(z)|\leq C.\]
In view of \eqref{def:rho} and \eqref{rho2 Gamma}, for any $x>0$, we have
\begin{align*}
\rho_{\xi,k}(x)&=\frac{2}{\pi}\big(e^{(\xi+ik)x}\sqrt{(\xi+ik)x}(\sqrt{\pi}-\Gamma(\tfrac{1}{2},(\xi+ik)x))+1\big)\\
& =\frac{2}{\sqrt{\pi}}e^{(\xi+ik)x}\sqrt{(\xi+ik)x}+\frac{1}{\pi(\xi+ik)x}-\frac{2\varepsilon_2((\xi+ik)x)}{\pi(\xi+ik)^2x^2}.
\end{align*}
As $x\geq R/|\xi|$, for any $k\in \Z$, we have $|(\xi+ik)x|\geq|\xi|x\geq R$, which gives \eqref{eq:gammapm00}.
\end{proof}

For any $\xi\in\R$ and $k\in\Z$, define the continuous function $\varrho_{\xi,k}:[0,+\infty)\to\C$, analytic on $(0,+\infty)$ given by
\[\varrho_{\xi,k}(x)=2\int_0^{\sqrt{x}}e^{-(\xi+ik)s^2}ds\text{ for all }x\geq 0.\]
By definition,
\begin{equation}\label{eq:relrho}
\rho_{\xi,k}(x)=
\frac{2}{\pi}\big((\xi+ik)e^{(\xi+ik)x}\sqrt{x}\varrho_{\xi,k}(x)+1\big)\text{ for }x\geq 0.
\end{equation}
%For any $z\in\C$ with $\rp z\geq 0$, we will denote by $\sqrt{z}$ its non-negative root, i.e., $\rp\sqrt{z}\geq 0$.
\begin{lemma}
For any non-zero $\xi+ik$ with $\xi\geq 0$ and $k\in\Z$, we have
\begin{gather}
%\label{eq:gammapm}
%\lim_{x\to+\infty}\varrho_{\xi,k}(x)=\frac{\sqrt{\pi}}{\sqrt{\xi+ik}},\\
\label{eq:gammapm1}
\Big|\varrho_{\xi,k}(x)-\frac{\sqrt{\pi}}{\sqrt{\xi+ik}}\Big|\leq \frac{ e^{-\xi x}}{|\xi+ik|\sqrt{x}}\text{ for all }x>0.
%\\
%\label{eq:gammapm2}
%\Big|\varrho_{\xi,0}(x)-\frac{\sqrt{\pi}}{\sqrt{\xi}}\Big|\leq \frac{e^{-\xi x}}{\xi\sqrt{x}}\text{ for all }x>0\text{ if }\xi>0.
\end{gather}
\end{lemma}
\begin{proof}
Summarizing the discussion in \cite[Section 4.2.4]{Olv}, for any complex $z$ with $\rp z\geq 0$, we have
\[\Gamma(\tfrac{1}{2},z)=\frac{e^{-z}}{\sqrt{z}}\varepsilon_0(z)\quad\text{with}\quad |\varepsilon_0(z)|\leq 1.\]
By definition,
\[\varrho_{\xi,k}(x)=\frac{\sqrt{\pi}}{\sqrt{\xi+ik}}\erf(\sqrt{(\xi+ik)x})=\frac{\sqrt{\pi}}{\sqrt{\xi+ik}}
\Big(1-\frac{\Gamma(\tfrac{1}{2},(\xi+ik)x)}{\sqrt{\pi}}\Big).\]
It follows that
\[\Big|\varrho_{\xi,k}(x)-\frac{\sqrt{\pi}}{\sqrt{\xi+ik}}\Big|\leq \Big|\frac{1}{\sqrt{\xi+ik}}\frac{e^{-(\xi+ik)x}}{\sqrt{(\xi+ik)x}}\Big|
=\frac{ e^{-\xi x}}{|\xi+ik|\sqrt{x}}.\]
\end{proof}}

\begin{lemma}\label{lem:gammac}
Suppose that $(c_k)_{k\in\Z}$ is a sequence of complex numbers decaying exponentially, i.e., there exist $\delta,C>0$ such that
$|c_k|\leq Ce^{-\delta|k|}$ for all $k\in\Z$. Then, for every $0<\vep<\delta/2$, the series $\sum_{k\in\Z}c_k\rho_{\xi,k}$
converges in $H(S_{\delta/2-\vep})$.
\end{lemma}

\begin{proof}
In view of \eqref{eq:gamma3neq}, for every $k\in\Z$,
\begin{align*}
\|\rho_{\xi,k}\|_{n}^{S_{\delta/2-\vep}}&=\sup\{|\rho_{\xi,k}(z)|:|z|\leq n,|\ip z|\leq \delta/2-\vep\}\\
&\leq 2(|\xi|+|k|)ne^{2|\xi|n}e^{|k|(\delta-2\vep)}+1.
\end{align*}
It follows that
\[\|c_k\rho_{\xi,k}\|_{n}^{S_{\delta/2-\vep}}\leq  2(|\xi|+|k|)ne^{2|\xi|n}e^{-2\vep|k|}+Ce^{-\delta|k|}\text{ for all }k\in\Z,\]
and hence
\[\sum_{k\in\Z}\|c_k\rho_{\xi,k}\|_{n}^{S_{\delta/2-\vep}}<+\infty\text{ for every }n\geq 1.\]
This gives the required convergence in $H(S_{\delta/2-\vep})$.
\end{proof}

In the next part of this section, we analyze the form and some properties of pre-images under $A$ of analytic geometrically quasi-periodic functions.

Let $g\in C^{\omega}(\R)$ be an analytic geometrically $2\pi$-quasi-periodic map with the exponent $e^{2\pi\xi}$, i.e., $g(x+2\pi)=e^{2\pi\xi}g(x)$. Then, $\theta\mapsto e^{-\xi\theta}g(\theta)$ is
an analytic $2\pi$-periodic map with Fourier coefficients $(c_k)_{k\in\Z}$ vanishing exponentially. Let $\widetilde{\rho}_{g,\xi}$ be the sum
of the series $\sum_{k\in\Z}c_k\rho_{\xi,k}$. In view of Lemma~\ref{lem:Agamma}~and~\ref{lem:gammac},
\begin{equation}\label{eq:Ag}
\widetilde{\rho}_{g,\xi}\in C^{\omega}(\R)\quad\text{and}\quad A(\widetilde{\rho}_{g,\xi})=g.
\end{equation}
%We denote by $\breve{g}:\R\to\C$ the $2\pi$-periodic analytic map given by the sum of the trigonometric series
%$\sqrt{\pi}\sum_{k\in\Z}\sqrt{\xi+ik}c_ke^{ikx}$. Note that if $g$ is real valued ($c_{-k}=\overline{c}_k$), then $\breve{g}$ is also real valued.

\begin{lemma}\label{lem:estpos}
Let $g\in C^{\omega}(\R)$ be  geometrically $2\pi$-quasi-periodic with the exponent $e^{2\pi\xi}\geq 1$
% For any $x>0$ we have
%\begin{equation}\label{eq:diffrhog}
%\Big|\widetilde{\rho}_{g,\xi}(x)-\frac{2}{\pi}(\sqrt{x}e^{\xi x}\breve{g}(x)+g(0))\Big|\leq 4\sum_{k\in\Z}|c_k||\xi+ik|\frac{1}{\sqrt{x}},
%\end{equation}
and let $(c_k)_{k\in\Z}$ be the Fourier coefficients of the $2\pi$-periodic map $\theta\mapsto e^{-\xi\theta}g(\theta)$.
If $\widetilde{\rho}_{g,\xi}(x)>0$ for $x\geq 0$, then
\begin{equation}\label{eq:rhoxik}
\sum_{k\in\Z}\sqrt{\xi+ik}c_ke^{ikx}\geq 0\text{ for }x\in\R.
\end{equation}
If, additionally, $\xi=0$, then $g$ is constant.
\end{lemma}

\begin{proof}
Recall that for every $x\in\R$,
\begin{equation}\label{eq:ghat}
\widetilde{\rho}_{g,\xi}(x)=\sum_{k\in\Z}c_k\rho_{\xi,k}(x).
\end{equation}
Let us consider the $2\pi$-periodic analytic map $\breve{g}:\R\to\R$ given by
\[
\breve{g}(x)=\sqrt{\pi}\sum_{k\in\Z}\sqrt{\xi+ik}c_ke^{ikx}.
\]
By \eqref{eq:relrho}, for every $x>0$ and $k\in \Z$,
\begin{equation*}
\rho_{\xi,k}(x)=
\frac{2}{\pi}\big((\xi+ik)e^{(\xi+ik)x}\sqrt{x}\varrho_{\xi,k}(x)+1\big).
\end{equation*}
As $\xi\geq 0$, by \eqref{eq:gammapm1}, if $\xi+ik\neq 0$, then
\begin{gather}\label{eq:gammapm2}
\Big|\rho_{\xi,k}(x)-\frac{2}{\pi}\big(\sqrt{\pi}\sqrt{\xi+ik}\sqrt{x}e^{(\xi+ik)x}+1\big)\Big|\leq 4\text{ for every }x>0.
\end{gather}
If $\xi+ik= 0$, then \eqref{eq:gammapm2} is also satisfied because its left side is zero.
%If $\xi>0$, then, by \eqref{eq:gammapm2}, for $x>0$ we have
%\begin{gather*}
%\Big|\rho_{\xi,0}(x)-\frac{2}{\pi}\sqrt{\pi}\sqrt{\xi}\sqrt{x}e^{\xi x}\Big|\leq\frac{1}{\pi\xi x}.
%\end{gather*}
%If $\xi=0$, then for $x>0$ we have
%\begin{gather*}
%\Big|\rho_{\xi,0}(x)-\frac{2}{\pi}(\sqrt{\pi}\sqrt{\xi}\sqrt{x}e^{\xi x}+1)\Big|=0.
%\end{gather*}
%It follows that setting
%\[\alpha_\xi:=\left\{
%\begin{array}{rl}
%1 & \text{if}\quad \xi=0\\
%0 & \text{if}\quad \xi>0
%\end{array}
%\right.
%\quad\text{ and }\quad
%\beta_\xi:=\left\{
%\begin{array}{rl}
%0 & \text{if}\quad \xi=0\\
%\frac{1}{\pi\xi} & \text{if}\quad \xi>0
%\end{array}
%\right.\]
%we get
%\begin{gather*}
%\Big|\rho_{\xi,0}(x)-\frac{2}{\pi}(\sqrt{\pi}\sqrt{\xi}\sqrt{x}e^{\xi x}+\alpha_\xi)\Big|\leq\frac{\beta_\xi}{x}\text{ for all }x>0\text{ and }\xi\geq 0.
%\end{gather*}
By \eqref{eq:ghat}, this gives
\[\Big|\widetilde{\rho}_{g,\xi}(x)-\frac{2}{\pi}\Big(\sqrt{x}e^{\xi x}\breve{g}(x)+\sum_{k\in\Z}c_k\Big)\Big|\leq 4\sum_{k\in\Z}|c_k|\text{ for every }x>0.\]
Since $\breve{g}$ is $2\pi$-periodic, there is $C>0$ such that for any $x\in[0,2\pi)$ and $n\in\N$,
\[\widetilde{\rho}_{g,\xi}(x+2\pi n)\leq \frac{2}{\pi}\sqrt{x+2\pi n}e^{\xi (x+2\pi n)}\breve{g}(x)+C.\]
If $\widetilde{\rho}_{g,\xi}(x)>0$ for $x>0$, then we obtain $\breve{g}(x)\geq 0$ for $x\in[0,2\pi)$, which gives \eqref{eq:rhoxik}.

Suppose that $\xi=0$. Then, $\breve{g}(x)=\sqrt{\pi}\sum_{k\in\Z}\sqrt{ik}c_ke^{ikx}$ is a non-negative map with zero integral on $[0,2\pi)$.
It follows that $\breve{g}\equiv 0$, and thus $c_k=0$ for $k\in\Z\setminus\{0\}$. Hence,
$g(x)=e^{-\xi x}g(x)=\sum_{k\in\Z}c_ke^{ikx}=c_0$,
so $g$ is constant.
\end{proof}

\begin{lemma}\label{lem:estneg}
Let $g\in C^{\omega}(\R)$ be  geometrically $2\pi$-quasi-periodic with the exponent $e^{-2\pi\xi}< 1$
and let $(c_k)_{k\in\Z}$ be the Fourier coefficients of the $2\pi$-periodic map $\theta\mapsto e^{\xi\theta}g(\theta)$.
%For every geometrically $2\pi$-quasiperiodic function $g\in C^{\omega}(\R)$ (with exponent $e^{-2\pi\xi}< 1$) and every $x\geq 4/\xi$ we have
%\begin{equation}\label{eq:diffrhog1}
%\Big|\widetilde{\rho}_{g,-\xi}(x)+\frac{2}{\pi}\sum_{k\in\Z}\frac{c_k}{\xi-ik}\frac{1}{x}\Big|\leq
%\frac{1}{7\pi\xi^2}\sum_{k\in\Z}|c_k|\frac{1}{x^2}
%+\frac{10e^4}{\pi\sqrt{\xi}}\sum_{k\in\Z}|c_k||\xi-ik|\sqrt{x}e^{-\xi x},
%\end{equation}
%where $(c_k)_{k\in\Z}$ is the sequence of Fourier coefficients of the map $e^{\xi\theta}g(\theta)$.
If $\widetilde{\rho}_{g,-\xi}(x)>0$ for $x> 0$, then
\[\sum_{k\in\Z}\frac{c_k}{\xi-ik}\leq 0.\]
\end{lemma}

\begin{proof}
Recall that for every $x\in\R$, we have
\begin{equation}\label{eq:ghat1}
\widetilde{\rho}_{g,-\xi}(x)=\sum_{k\in\Z}c_k\rho_{-\xi,k}(x).
\end{equation}
{As $\xi>0$, by \eqref{eq:gammapm00}, for every $k\in\Z$ and $x\geq R/\xi$, we have
\begin{gather*}
\Big|\rho_{-\xi,k}(x)+\frac{1}{\pi(\xi-ik)x}\Big|\leq \frac{C}{|\xi-ik|^2x^2}+2\sqrt{|\xi-ik|x}e^{-\xi x}.
\end{gather*}
By \eqref{eq:ghat1},  for every $x\geq R/\xi$, this gives
\[\Big|\widetilde{\rho}_{g,-\xi}(x)+\frac{1}{\pi}\sum_{k\in\Z}\frac{c_k}{\xi-ik}\frac{1}{x}\Big|\leq C\sum_{k\in\Z}\frac{|c_k|}{|\xi-ik|^2x^2}
+2\sum_{k\in\Z}|c_k|\sqrt{|\xi-ik|}\sqrt{x}e^{-\xi x}.\]
If $\widetilde{\rho}_{g,-\xi}(x)>0$ for $x>0$, then we get $\sum_{k\in\Z}\frac{c_k}{\xi-ik}\leq 0$.}
\end{proof}

\begin{proof}[Proof of Proposition~\ref{prop:mainqper}]
In view of Remark~\ref{rmk:rescal}, we can rescale $f_1$ so that  $A(f_1)$ is geometrically $2\pi$-quasi-periodic.
Suppose that $A(f_1)(\theta+2\pi)=e^{2\pi\xi}A(f_1)(\theta)$ for some $\xi\in\R$.
Let $(c_k)_{k\in\Z}$ be the Fourier coefficients of the analytic $2\pi$-periodic map $\theta\mapsto e^{-\xi\theta}A(f_1)(\theta)$.
Since the operator $A$ has a trivial kernel, by \eqref{eq:Ag}, we have $\widetilde{\rho}_{A(f_1),\xi}=f_1$.

\medskip
\noindent
\textbf{Zero Case: $\xi=0$.} Suppose that $\xi=0$. Then, directly by Lemma~\ref{lem:estpos} applied to $g:=A(f_1)$, the map $A(f_1)$ is constant. Since $A(1)=\frac{\pi}{2}$ and $A$ has trivial kernel, the map $f_1$ is also constant. As $A(f_2)(x)=A(f_1)(x_0-x)$ and $A(f_2)$ is analytic, $A(f_2)$ as well as $f_2$ are constant.

\medskip
\noindent
\textbf{Negative Case: $\xi<0$.} Suppose that $\xi<0$. Then, directly by Lemma~\ref{lem:estneg} applied to $g:=A(f_1)$ and $\xi:=-\xi>0$, we obtain that
\begin{equation}\label{eq:nonneg}
\sum_{k\in\Z}\frac{c_k}{\xi+ik}\geq 0.
\end{equation}
As $f_1(x)>0$ for $x>0$ and $(c_k)_{k\in\Z}$ vanishes exponentially, by the definition of $A$, we have
\[0< A(f_1)(x)=\sum_{k\in\Z}c_k e^{(\xi+ik)x}\text{ for }x>0,\]
and the series is uniformly convergent on any right half-line. It follows that
\[0<\int_0^{+\infty}A(f_1)(x)dx=-\sum_{k\in\Z}\frac{c_k}{\xi+ik},\]
which contradicts \eqref{eq:nonneg}.

\medskip
\noindent
\textbf{Positive Case: $\xi>0$.} Suppose that $\xi>0$. By assumption, $A(f_2)(x)>0$ for $x>0$ and
\[A(f_2)(x+2\pi)=A(f_1)(x_0-x-2\pi)=e^{-2\pi\xi}A(f_1)(x_0-x)=e^{-2\pi\xi}A(f_2)(x).\]
It follows that $f_2$ is an analytic map that is positive on the positive half-line and that $A(f_2)$ is geometrically $2\pi$-quasi-periodic with an exponent less than $1$.
We also get to the contradiction using the arguments from the Negative Case for the function $f_2$ instead of $f_1$.
\end{proof}

\begin{proof}[Proof of Theorem~\ref{thm:const}]
Let $c:(0,+\infty)\to\R$ be the analytic map given by
\[c(\theta)=\frac{1}{\sqrt{2}}\int_0^1\frac{W'_2(\sqrt{\theta}s)+W'_2(-\sqrt{\theta}s)}{\sqrt{1-s^2}}ds.\]
By assumption, we have $a(\theta)+\bar{a}(\theta)=\gamma c(E-\theta)$ for $\theta\in(0,E)$.
Moreover, $a+\bar{a}=\frac{1}{\sqrt{2}}A(u_1)$ and $c=\frac{1}{\sqrt{2}}A(u_2)$, where $u_1,u_2$ are analytic and positive maps on $[-\vep,+\infty)$ for some  $\vep>0$ so  that
\[u_1(x^2)=W'_1(x)+W'_1(-x)\text{ and }u_2(x^2)=W'_2(x)+W'_2(-x).\]
As $a+\bar{a}=\frac{1}{\sqrt{2}}A(u_1)$ is geometrically quasi-periodic, by {Proposition~\ref{prop:mainqper}}, we obtain that $u_1$ and $\gamma u_2$ are constant. Therefore, $W_1'+\bar{W}_1'$
and  $W_2'+\bar{W}_2'$ are constant as well.  In view of Remark~\ref{rem:SP}, this gives $V_1,V_2\in\mathcal{SP}$.
\end{proof}

\section{Proofs of the main results}\label{sec:appl}
\begin{proof}[Proof of Theorem~\ref{thm:main1}]
Fix $E>0$ and let any $I\in \mathcal{J}_E$. Then,  the polygons
\[P\cap(\left[-\bar{V}_1^{-1}(\theta),V_1^{-1}(\theta)\right]\times
\left[-\bar{V}_2^{-1}(E-\theta),V_2^{-1}(E-\theta)\right])\text{ for }\theta\in I\]
form a  smooth curve of polygons in $\mathcal{RP}$. Passing through
the change of coordinates $\eta$, we obtain a smooth curve $I\ni\theta\mapsto P_{E,\theta}\in\mathcal{RP}$ of billiard
tables.
In view of \eqref{eq:XY+-1} and \eqref{eq:XY+-2}, for every such curve and every $\theta\in I$,  we have
\begin{align*}
&X^+_{P_{E,\theta}}\setminus\{x^+_{P_{E,\theta}}\}\subset\{a_\xi(\theta):\xi\in X^+_I\},\
X^-_{P_{E,\theta}}\setminus\{x^-_{P_{E,\theta}}\}\subset\{\bar{a}_\xi(\theta):\xi\in X^-_I\},\\
&x^+_{P_{E,\theta}}\in \{a(\theta)\}\cup \{a_\xi(\theta):\xi\in X^+_I\},\ x^-_{P_{E,\theta}}\in \{\bar{a}(\theta)\}\cup \{\bar{a}_\xi(\theta):\xi\in X^-_I\},\\
&Y^+_{P_{E,\theta}}\setminus\{y^+_{P_{E,\theta}}\}\subset\{b_\xi(\theta):\xi\in Y^+_I\},\
Y^-_{P_{E,\theta}}\setminus\{y^-_{P_{E,\theta}}\}\subset\{\bar{b}_\xi(\theta):\xi\in Y^-_I\},\\
&y^+_{P_{E,\theta}}\in \{b(\theta)\}\cup \{b_\xi(\theta):\xi\in Y^+_I\},\ y^-_{P_{E,\theta}}\in \{\bar{b}(\theta)\}\cup \{\bar{b}_\xi(\theta):\xi\in Y^-_I\}.
\end{align*}
Therefore, by Theorem~\ref{thm:main-min}, to prove that the billiard table $P_{E,\theta}$ is non-resonant for all but countably many $\theta\in I$,
it suffices to show that
for every choice of integers $\gamma_{a}$,  $\gamma_{\bar{a}}$,  $\gamma_{a_{\xi}}$ for $\xi\in X^+_I\setminus\{0\}$,
$\gamma_{\bar{a}_{\xi}}$ for $\xi\in X^-_I$, $\gamma_{b}$,  $\gamma_{\bar{b}}$,  $\gamma_{b_{\xi}}$ for $\xi\in Y^+_I\setminus\{0\}$,
$\gamma_{\bar{b}_{\xi}}$ for $\xi\in Y^-_I$ such that at least one number is non-zero and $\gamma_{a}\cdot \gamma_{\bar{a}}\geq 0$, $\gamma_{b}\cdot \gamma_{\bar{b}}\geq 0$,
we have
\begin{align}\label{eq:gms}
\begin{split}
\gamma_{a}&a(\theta)+ \gamma_{\bar{a}}\bar{a}(\theta)+\sum_{\xi\in X^+_I\setminus\{0\}}\gamma_{a_{\xi}}a_{\xi}(\theta)+
\sum_{\xi\in X^-_I}\gamma_{\bar{a}_{\xi}}\bar{a}_{\xi}(\theta)\\
&+\gamma_{b}b(\theta)+ \gamma_{\bar{b}}\bar{b}(\theta)
+\sum_{\xi\in Y^+_I\setminus\{0\}}\gamma_{b_{\xi}}b_{\xi}(\theta)+
\sum_{\xi\in Y^-_I}\gamma_{\bar{b}_{\xi}}\bar{b}_{\xi}(\theta)\neq 0
\end{split}
\end{align}
for all but countably many $\theta\in I$.  If $V_1$ is even, then $\bar{a}=a$ and $\bar{a}_\xi=a_\xi$, so we additionally assume that
$\gamma_{\bar{a}}=\gamma_{\bar{a}_{\xi}}=0$ for all $\xi\in X^+_I\cap X^-_I$.
Similarly, if $V_2$ is even, then  we additionally assume that
$\gamma_{\bar{b}}=\gamma_{\bar{b}_{\xi}}=0$ for all $\xi\in Y^+_I\cap Y^-_I$.

\medskip

\noindent
\textbf{Case 1.} Suppose that at least one integer number $\gamma_{a_{\xi}}$ for $\xi\in X^+_I\setminus\{0\}$,
$\gamma_{\bar{a}_{\xi}}$ for $\xi\in X^-_I$,  $\gamma_{b_{\xi}}$ for $\xi\in Y^+_I\setminus\{0\}$, or
$\gamma_{\bar{b}_{\xi}}$ for $\xi\in Y^-_I$ is non-zero. Then \eqref{eq:gms} follows directly from the first part of Proposition~\ref{prop:indep0}.

\medskip

\noindent
\textbf{Case 2.} Suppose that all the above $\gamma$'s are zero. Then we have to assume that at least one integer number $\gamma_{a}$,  $\gamma_{\bar{a}}$,
$\gamma_{b}$,  $\gamma_{\bar{b}}$ is non-zero and then  show that
\begin{equation}\label{eq:gms1}
\gamma_{a}a(\theta)+ \gamma_{\bar{a}}\bar{a}(\theta)+\gamma_{b}b(\theta)+ \gamma_{\bar{b}}\bar{b}(\theta)\neq 0\text{ for all but countably many }\theta\in I.
\end{equation}

If  $m_1>2$ or $m_2>2$, then \eqref{eq:gms1} follows directly from the second part of  Proposition~\ref{prop:indep0}. This completes the proof of Theorem~\ref{thm:main1} under assumption $(a)$.

\medskip

Now suppose that $m_1=m_2=2$ and $V_1(x)=\omega V_2(\tau x)$ for some $\omega>0$ and $\tau\neq 0$ such that
\[|\tau|\sqrt{\omega}=\sqrt{\frac{V''_1(0)}{V''_2(0)}}\quad \text{ is irrational.}\]
Suppose, contrary to our claim, that \eqref{eq:gms1} does not meet. Then, by part $(a)$ of Proposition~\ref{prop:twobad}, we have
$\gamma_a+\gamma_{\bar{a}}\neq 0$, $\gamma_b+\gamma_{\bar{b}}\neq 0$ and
\[\sqrt{\frac{V''_1(0)}{V''_2(0)}}=|\tau|\sqrt{\omega}=-\frac{\gamma_{a}+\gamma_{\bar{a}}}{\gamma_{b}+\gamma_{\bar{b}}}\in\Q.\]
This gives a contradiction which completes the proof  of Theorem~\ref{thm:main1} under assumption $(b)$.
\end{proof}

%\begin{remark}
%Non-minimality of  $(\varphi_t^{P,E,\theta})_{t\in\R}$ on a connected component of $S_{E,\theta}$ is equivalent to the non-minimality
%of the translation flow $(\psi_t^{\pi/4})_{t\in\R}$ on the translation surface $M(\theta)$ arising from a connected polygon $P_{E,\theta}$. Then $(\psi_t^{\pi/4})_{t\in\R}$ has periodic orbits
%or a saddle connection. Non-minimality of  $(\varphi_t^{P,E,\theta})_{t\in\R}$ on some connected component of $S_{E,\theta}$ is equivalent to the existence of a periodic orbit for  $(\varphi_t^{P,E,\theta})_{t\in\R}$
%or the existence of an orbit segment joining vertices of the polygon $P$ with internal angle $3\pi/2$.
%\end{remark}
%
%
%Let $V_1,V_2:\R\to\R_{\geq 0}$ be two analytic potentials satisfying \eqref{eq:i} with $m_1:=\deg(V_1,0)$ and $m_2:=\deg(V_2,0)$.
%Assume that $P$ is any $\mathcal{RP}$-polygon.
%\begin{corollary}
%If
%\begin{itemize}
%\item[(a)] at least one degree $m_1$ or $m_2$ is greater than $2$ or;
%\item[(b)] $m_1=m_2=2$ and $V_1(x)=\omega V_2(\tau x)$ for some $\omega>0$ and $\tau\neq 0$ such that $\sqrt{\tfrac{V''_1(0)}{V''_2(0)}}$ is irrational,
%\end{itemize}
%then there are no resonance energies.
%\end{corollary}

%\begin{corollary}
%If $E\in \mathcal{E}(P,V_1,V_2)$ then $m_1=m_2=2$ and there exist natural numbers $\gamma_a,\gamma_b$
%such that
%\[\gamma_a(a(\theta)+\bar{a}(\theta))=\gamma_b(b(\theta)+\bar{b}(\theta))\text{ for all }\theta\in [0,E].\]
%\end{corollary}

\begin{proposition}
The set of resonant energy levels $\mathcal{E}(P,V_1,V_2)$ is bounded, more precisely
\[\mathcal{E}(P,V_1,V_2)\subset\big(0,\max\{V_1(x^+_P),\bar{V}_1(x^-_P)\}+\max\{V_2(y^+_P),\bar{V}_2(y^-_P)\}\big).\]
\end{proposition}

\begin{proof}
Suppose that
\[E\geq \max\{V_1(x^+_P),\bar{V}_1(x^-_P)\}+\max\{V_2(y^+_P),\bar{V}_2(y^-_P)\}.\]
We need to show that for every $I\in \mathcal{J}_E$, the billiard table $P_{E,\theta}$
is non-resonant for all but countably many $\theta\in I$. The proof of this fact uses the same arguments that were applied in the proof of Theorem~\ref{thm:main1} but in a slightly more subtle way.

For every $I\in \mathcal{J}_E$, we have
\[\max\{V_1(x^+_P),\bar{V}_1(x^-_P)\}\leq \theta\ \text{ or }\ \max\{V_2(y^+_P),\bar{V}_2(y^-_P)\}\leq E-\theta\ \text{ for all }\ \theta\in I.\]
Suppose that $\bar{V}_1(x^-_P)\leq \theta$ and $V_1(x^+_P)\leq \theta$ for all $\theta\in I$.
The proof of the second case proceeds similarly.

Due to our additional assumption, the parameters of the vertical sides cannot be of the form  $a(\theta)$ or $\bar{a}(\theta)$. This allows us to prove the absence of resonance without additional assumptions on the potentials. More precisely, in view of \eqref{eq:XY+-1}, \eqref{eq:XY+-2}, and \eqref{eq:XY+-3}, for every  $\theta\in I$,  we have
\begin{align*}
&X^+_{P_{E,\theta}}\subset\{a_\xi(\theta):\xi\in X^+_I\},\
X^-_{P_{E,\theta}}\subset\{\bar{a}_\xi(\theta):\xi\in X^-_I\},\\
&Y^+_{P_{E,\theta}}\setminus\{y^+_{P_{E,\theta}}\}\subset\{b_\xi(\theta):\xi\in Y^+_I\},\
Y^-_{P_{E,\theta}}\setminus\{y^-_{P_{E,\theta}}\}\subset\{\bar{b}_\xi(\theta):\xi\in Y^-_I\},\\
&y^+_{P_{E,\theta}}\in \{b(\theta)\}\cup \{b_\xi(\theta):\xi\in Y^+_I\},\ y^-_{P_{E,\theta}}\in \{\bar{b}(\theta)\}\cup \{\bar{b}_\xi(\theta):\xi\in Y^-_I\}.
\end{align*}
Therefore, by Theorem~\ref{thm:main-min}, to prove that  $P_{E,\theta}$ is non-resonant for all but countably many $\theta\in I$,
it suffices to show that for every choice of integers   $\gamma_{a_{\xi}}$ for $\xi\in X^+_I\setminus\{0\}$,
$\gamma_{\bar{a}_{\xi}}$ for $\xi\in X^-_I$, $\gamma_{b}$,  $\gamma_{\bar{b}}$,  $\gamma_{b_{\xi}}$ for $\xi\in Y^+_I\setminus\{0\}$,
$\gamma_{\bar{b}_{\xi}}$ for $\xi\in Y^-_I$ such that at least one number is non-zero and  $\gamma_{b}\cdot \gamma_{\bar{b}}\geq 0$,
we have
\begin{align}\label{eq:gms11}
\begin{split}
\sum_{\xi\in X^+_I\setminus\{0\}}\gamma_{a_{\xi}}a_{\xi}(\theta)+
\sum_{\xi\in X^-_I}\gamma_{\bar{a}_{\xi}}\bar{a}_{\xi}(\theta)
&+\sum_{\xi\in Y^+_I\setminus\{0\}}\gamma_{b_{\xi}}b_{\xi}(\theta)+
\sum_{\xi\in Y^-_I}\gamma_{\bar{b}_{\xi}}\bar{b}_{\xi}(\theta)\\
&+\gamma_{b}b(\theta)+ \gamma_{\bar{b}}\bar{b}(\theta)
\neq 0
\end{split}
\end{align}
for all but countably many $\theta\in I$.  If $V_1$ is even, then  we additionally assume that
$\gamma_{\bar{a}_{\xi}}=0$ for all $\xi\in X^+_I\cap X^-_I$.
Similarly, if $V_2$ is even, then  we additionally assume that
$\gamma_{\bar{b}}=\gamma_{\bar{b}_{\xi}}=0$ for all $\xi\in Y^+_I\cap Y^-_I$.

If at least one integer number $\gamma_{a_{\xi}}$ for $\xi\in X^+_I\setminus\{0\}$,
$\gamma_{\bar{a}_{\xi}}$ for $\xi\in X^-_I$,  $\gamma_{b_{\xi}}$ for $\xi\in Y^+_I\setminus\{0\}$, or
$\gamma_{\bar{b}_{\xi}}$ for $\xi\in Y^-_I$ is non-zero, then \eqref{eq:gms11} follows directly from the first part of Proposition~\ref{prop:indep0}.

If all the above $\gamma$'s are zero, then \eqref{eq:gms11} reduces to $\gamma_{b}b(\theta)+ \gamma_{\bar{b}}\bar{b}(\theta)
\neq 0$. Since $\gamma_{b}$ and $\gamma_{\bar{b}}$ have the same sign and at least one is non-zero, we have $\gamma_{b}b(\theta)+ \gamma_{\bar{b}}\bar{b}(\theta)\neq 0$ for every $\theta\in I$. This completes the proof.
\end{proof}

%\begin{theorem}
%The set of resonance energies satisfies the following trichotomy:
%\begin{itemize}
%\item[(a)] either $\mathcal{E}(P,V_1,V_2)$ is empty;
%\item[(b)] or $\mathcal{E}(P,V_1,V_2)$ has only  one element, then $m_1=m_2=2$;
%\item[(c)] or $\mathcal{E}(P,V_1,V_2)$ is non-empty and open, then $V_1,V_2\in\mathcal{SP}$.
%\end{itemize}
%\end{theorem}

\begin{proof}[Proof of Theorem~\ref{thm:main2}]
In view of Theorem~\ref{thm:main1}, we need to show that if $m_1=m_2=2$ and $\mathcal{E}(P,V_1,V_2)$ has at least two elements, then we have $V_1,V_2\in\mathcal{SP}$ with $\sqrt{\frac{V''_1(0)}{V''_2(0)}}$ rational, and the set $\mathcal{E}(P,V_1,V_2)$ is open.

By Corollary~\ref{cor:EE}, we have $\mathcal{E}(P,V_1,V_2)\subset \mathfrak{E}(V_1,V_2)$.
Suppose that $\mathcal{E}(P,V_1,V_2)$ has at least two elements. In view of Proposition~\ref{prop:twobad}, it follows that $V_1,V_2\in\mathcal{SP}$.
Moreover, by \eqref{eq:ga4} and \eqref{eq:spab}, we get that $\sqrt{\frac{V''_1(0)}{V''_2(0)}}$ is rational.

The most challenging part of the proof remains the openness of the set $\mathcal{E}(P,V_1,V_2)$.
For any $E>0$ and $\theta\in(0,E)$, we denote by  $(M_{E,\theta},\omega_{E,\theta})$ the translation surface associated with the polygon $P_{E,\theta}$ and  its natural partition into four $\mathcal{RP}$-polygons by $\mathcal{P}_{E,\theta}$.

Suppose $E_0>0$ belongs to $\mathcal{E}(P,V_1,V_2)$. We focus only on the case when both potentials $V_1$ and $V_2$ are not even. In other cases, the proof runs similarly, and some formulas are even more straightforward.

Choose an interval $I\in\mathcal{J}_{E_0}$ such that $P_{{E_0},\theta}$ is resonant for all $\theta\in I_0$ so that $I_0\subset I$ is an uncountable subset. By definition, for every $\theta\in I_0$, the translation surface $(M_{{E_0},\theta},\omega_{{E_0},\theta})$ has a saddle connection $\gamma(\theta)$ in direction $\pi/4$ with length $\tau(\theta)>0$. Since $I\in\mathcal{J}_{E_0}$, the combinatorial data of all partitions $\mathcal{P}_{{E_0},\theta}$ for $\theta\in I$ are the same. Hence, we can denote  the set of positively (negatively) oriented vertical sides by $D^{\pm}_v$,  the set of positively (negatively) oriented horizontal sides by $D^{\pm}_h$, the set of extreme sides by $D_{ext}$, and the set of vertices of $\mathcal P_{{E_0},\theta}$ by $V$, independently of $\theta\in I$. We denote by $D^\pm_{marg,v}, D^\pm_{marg,h}\subset D_{ext}$ the sets of extreme sides which come from marginal sides of the polygons $P_{{E_0},\theta}$, in the sense of Remark~\ref{rmk:marg}, such that
\begin{align*}
e\in D^+_{marg,v}&\Longleftrightarrow e\in D_v\text{ is related to the right end of the polygon }P_{{E_0},\theta},\\
e\in D^-_{marg,v}&\Longleftrightarrow e\in D_v\text{ is related to the left end of the polygon }P_{{E_0},\theta},\\
e\in D^+_{marg,h}&\Longleftrightarrow e\in D_h\text{ is related to the upper end of the polygon }P_{{E_0},\theta},\\
e\in D^-_{marg,h}&\Longleftrightarrow e\in D_h\text{ is related to the lower end of the polygon }P_{{E_0},\theta}.
\end{align*}
In view of \eqref{eq:sumsum}, for any $\theta\in I_0$,
\begin{align*}
\begin{aligned}
\tau(\theta) e^{i\frac{\pi}{4}}&=\sum_{e\in D_v^+}2n_e(\theta)x_\theta(e)-\sum_{e\in D_v^-}2n_e(\theta)x_\theta(e)\\
&\quad+\vep_1^b(v_+(\theta))x(e_v(v_+(\theta)))+\vep_1^e(v_-(\theta))x(e_v(v_-(\theta)))\\
&\quad+i\Big(\sum_{e\in D_h^+}2n_e(\theta)y_\theta(e)-\sum_{e\in D_h^-}2n_e(\theta)y_\theta(e)\\
&\quad\qquad+\vep_2^b(v_+(\theta))y_\theta(e_h(v_+(\theta)))+\vep_2^e(v_-(\theta))y_\theta(e_h(v_-(\theta)))\Big),
\end{aligned}
\end{align*}
where  $v_+(\theta)\in V$ represents the beginning and $v_-(\theta)\in V$  the end of $\gamma(\theta)$, and $n_e(\theta)$ is the meeting number of $e\in D$ with $\gamma(\theta)$ in $(M_{{E_0},\theta},\omega_{{E_0},\theta})$.
Since $I_0$ is uncountable, one can find another uncountable subset $I_1\subset I_0$ such that $v_+$, $v_-$ and $n_e$ are constant on $I_1$ for every $e\in D$. In view of \eqref{eq:XY+-1} and  \eqref{eq:XY+-2},
for every $e\in D^{\pm}_v$, we have
\begin{align*}
&x_\theta(e)=a_{\xi_e}(\theta)\text{ for some }\xi_e\in X^+_I\text{ or }x_\theta(e)=a(\theta)\text{ (if $e\in D^+_{marg,v}$) or }\\
&x_\theta(e)=\bar{a}_{\xi_e}(\theta)\text{ for some }\xi_e\in X^-_I\text{ or }x_\theta(e)=\bar{a}(\theta)\text{ (if $e\in D^-_{marg,v}$),}
\end{align*}
and
for every $e\in D^{\pm}_h$, we have
\begin{align*}
&y_\theta(e)=b_{\xi_e}(\theta)\text{ for some }\xi_e\in Y^+_I\text{ or }y_\theta(e)=b(\theta)\text{ (if $e\in D^+_{marg,h}$) or }\\
&y_\theta(e)=\bar{b}_{\xi_e}(\theta)\text{ for some }\xi_e\in Y^-_I\text{ or }y_\theta(e)=\bar{b}(\theta)\text{ (if $e\in D^-_{marg,h}$).}
\end{align*}
It follows that for $\theta\in I_1$,
\begin{align*}\label{eq:tautau1}
\begin{aligned}
\tau(\theta) e^{i\pi/4}&=\sum_{\xi\in X^+_I\setminus\{0\}}p_\xi a_\xi(\theta)+\sum_{\xi\in X^-_I}p_\xi\bar{a}_\xi(\theta)+p_{a} a(\theta)+ p_{\bar a} \bar{a}(\theta)\\
&\quad+i\Big(\sum_{\xi\in Y^+_I\setminus\{0\}}q_\xi b_\xi(\theta)+\sum_{\xi\in Y^-_I} q_\xi\bar{b}_\xi(\theta)+q_{b}  b(\theta)+ q_{\bar b}  \bar{b}(\theta)\Big),
\end{aligned}
\end{align*}
where
\begin{align*}
p_\xi&=\sum_{e\in D_v^+, \xi_e=\xi}2n_e-\sum_{e\in D_v^-, \xi_e=\xi}2n_e+\vep_1^b(v_+)\delta_{\xi_{e_v(v_+)},\xi}+\vep_1^e(v_-)\delta_{\xi_{e_v(v_-)},\xi},\\
q_\xi&=\sum_{e\in D_h^+, \xi_e=\xi}2n_e-\sum_{e\in D_h^-, \xi_e=\xi}2n_e+\vep_2^b(v_+)\delta_{\xi_{e_h(v_+)},\xi}+\vep_2^e(v_-)\delta_{\xi_{e_h(v_-)},\xi},\\
p_a&=\sum_{e\in D_{marg,v}^+}2n_e+\chi_{D_{marg,v}^+}(e_v(v_+))+\chi_{D_{marg,v}^+}(e_v(v_-))\geq 0,\\
p_{\bar a}&=\sum_{e\in D_{marg,v}^-}2n_e+\chi_{D_{marg,v}^-}(e_v(v_+))+\chi_{D_{marg,v}^-}(e_v(v_-))\geq 0,\\
q_b&=\sum_{e\in D_{marg,h}^+}2n_e+\chi_{D_{marg,h}^+}(e_h(v_+))+\chi_{D_{marg,h}^+}(e_h(v_-))\geq 0,\\
q_{\bar b}&=\sum_{e\in D_{marg,h}^-}2n_e+\chi_{D_{marg,h}^-}(e_h(v_+))+\chi_{D_{marg,h}^-}(e_h(v_-))\geq 0.
\end{align*}
Note that the absence of negative coefficients in the above four sums is because each marginal side is extreme, in which case we can use \eqref{eq:ext}, \eqref{eq:ext1}, \eqref{eq:ext2}, \eqref{eq:ext3}, and \eqref{eq:ext4}.
Hence, for $\theta\in I_1$,
\begin{align*}
\sum_{\xi\in X^+_I\setminus\{0\}}p_\xi a_\xi(\theta)&+\sum_{\xi\in X^-_I}p_\xi\bar{a}_\xi(\theta)+p_{a} a(\theta)+ p_{\bar a} \bar{a}(\theta)\\
&=\sum_{\xi\in Y^+_I\setminus\{0\}}q_\xi b_\xi(\theta)+\sum_{\xi\in Y^-_I} q_\xi\bar{b}_\xi(\theta)+q_{b}  b(\theta)+ q_{\bar b}  \bar{b}(\theta).
\end{align*}
As $I_1$ is uncountable, in view of Proposition~\ref{prop:indep0},
\begin{equation*} \label{eq:pqzero}
p_\xi=0\text{ for }\xi\in (X^+_I\setminus\{0\})\cup X^-_I\text{ and }q_\xi=0\text{ for }\xi\in (Y^+_I\setminus\{0\})\cup Y^-_I.
\end{equation*}
As $V_1$, $V_2$ are not even, by the proof of Proposition~\ref{prop:twobad}, we have $p_{a}=p_{\bar{a}}\neq 0$, $q_{b}=q_{\bar{b}}\neq 0$, and
\begin{equation}\label{eq:ga44}
p_a (a(\theta)+\bar{a}(\theta))=q_b(b_{E_0}(\theta)+\bar{b}_{E_0}(\theta))\text{ for all  }\theta \in (0,{E_0}).
\end{equation}
By the first part of the proof, $V_1,V_2\in\mathcal{SP}$. In view of Remark~\ref{rmk:SP}, this gives
\begin{equation}\label{eq:a+b}
a(\theta)+\bar{a}(\theta)=\frac{\pi}{\sqrt{V''_1(0)}},\quad
b_E(\theta)+\bar{b}_E(\theta)=\frac{\pi}{\sqrt{V''_2(0)}}\text{ for all  }E>0,\ \theta \in (0,E).
\end{equation}
Fix $\theta_0\in I_1$. Then, there exists $\vep>0$ such that for any pair $(E,\theta)$ with $|E-E_0|<\vep$ and $|\theta-\theta_0|<\vep$,
\begin{itemize}
\item the surfaces $(M_{E,\theta}, \omega_{E,\theta})$ and $(M_{E_0,\theta_0}, \omega_{E_0,\theta_0})$ have the same combinatorial data;
\item $(M_{E,\theta}, \omega_{E,\theta})$ has a saddle connection $\gamma(E,\theta)$ in a direction $\vartheta(E,\theta)$ (very close to $\pi/4$) which begins and ends at the same vertices as $\gamma(\theta_0)$ in $(M_{E_0,\theta_0}, \omega_{E_0,\theta_0})$;
\item $\gamma(E,\theta)$ and $\gamma(\theta_0)$ cross the same sides with the same multiplicities.
\end{itemize}
This follows from the fact that the parameters of the surface $(M_{E,\theta}, \omega_{E,\theta})$ change continuously around each pair $(E_0, \theta_0)$, provided that $\theta_0\in I\in \mathcal J_{E_0}$.

We will show that $\vartheta(E,\theta)=\pi/4$ for all $(E,\theta)$ with $|E-E_0|<\vep$ and $|\theta-\theta_0|<\vep$. This shows that $\mathcal{E}(P,V_1,V_2)$ is open.

As the saddle connections $\gamma(E,\theta)$ and $\gamma(\theta_0)$ have the same combinatorial data (i.e., they begin and end at the same vertices and pass through the same sides the same number of times), using \eqref{eq:sumsum} and the fact that the integer factors $p_\xi$, $p_a$, $p_{\bar a}$, $q_\xi$, $q_b$, $q_{\bar b}$ depend only on these combinatorial data, we get
\begin{align*}\label{eq:tautau2}
\begin{aligned}
\tau(E,\theta) e^{i\vartheta(E,\theta)}&=\sum_{\xi\in X^+_I\setminus\{0\}}p_\xi a_\xi(\theta)+\sum_{\xi\in X^-_I}p_\xi\bar{a}_\xi(\theta)+p_{a} a(\theta)+ p_{\bar a} \bar{a}(\theta)\\
&\quad+i\Big(\sum_{\xi\in Y^+_I\setminus\{0\}}q_\xi b_{E,\xi}(\theta)+\sum_{\xi\in Y^-_I} q_\xi\bar{b}_{E,\xi}(\theta)+q_{b}  b_E(\theta)+ q_{\bar b}  \bar{b}_E(\theta)\Big),
\end{aligned}
\end{align*}
where $\tau(E,\theta)>0$ is the length of $\gamma(E,\theta)$. Since $p_\xi=0$, $q_\xi=0$, $p_a=p_{\bar a}>0$, and $q_b=q_{\bar b}>0$, this gives
\[\tau(E,\theta) e^{i\vartheta(E,\theta)}=p_{a}(a(\theta)+ \bar{a}(\theta))
+iq_{b}  (b_E(\theta)+  \bar{b}_E(\theta)).\]
In view of \eqref{eq:ga44} and \eqref{eq:a+b}, it follows that
\[\tau(E,\theta) e^{i\vartheta(E,\theta)}=p_{a}(a(\theta)+ \bar{a}(\theta))(1+i),\]
so $\vartheta(E,\theta)=\pi/4$, which completes the proof of the openness of $\mathcal{E}(P,V_1,V_2)$.
\end{proof}
\appendix

\section{Examples of resonant energy levels for potentials outside  $\mathcal{SP}$}\label{sec:badpot}
The main purpose of this section is to show that option ($b$) in Theorem~\ref{thm:main2} can occur.
We construct a pair of potentials $V_1,V_2$ that are not in $\mathcal{SP}$ and for which a resonant energy level $E>0$ exists for certain polygons $P$. Then, by Theorem~\ref{thm:main2}, the set $\mathcal{E}(P,V_1,V_2)$ is a singleton.

In view of {Proposition~\ref{prop:twobad}}, if $E>0$ is a resonant energy level, then there exists a positive rational $\gamma$ such that
\begin{equation*}\label{eq:gaga4}
(a(\theta)+\bar{a}(\theta))=\gamma(b_E(\theta)+\bar{b}_E(\theta))\text{ for all  }\theta \in (0,E).
\end{equation*}
This condition is also sufficient to construct a rectilinear polygon $P$ for which $E\in\mathcal{E}(P,V_1,V_2)$. If we are not too ambitious, it is enough to take a sufficiently large rectangle as $P$, so that all orbits for the energy $E$ do not hit the sides of the rectangle. Then $P_{E,\theta}=[-\bar{a}(\theta),a(\theta)]\times[-\bar{b}_E(\theta),b_E(\theta)]$ and all its orbits are periodic for all $\theta\in(0,E)$, so $E\in\mathcal{E}(P,V_1,V_2)$. However, we can also construct more complicated vertically and horizontally symmetric rectilinear polygons $P$, for which $E$ is a resonant energy level. Then, we need to ensure that periodic orbits bouncing off the sides of $P$ also bounce off their symmetric counterparts, but that would require a more extensive discussion.

Another goal of this section is to show that the rationality of $\sqrt{V_2''(0)/V_1''(0)}$ is not a necessary condition for the existence of a resonant energy level $E>0$, as suggested by part ($c$) in Theorem~\ref{thm:main2} and part ($b$) in Theorem~\ref{thm:main1}.

\begin{proposition}
For every  $E>0$, there exists a pair of even non-quadratic $\mathcal{UM}$-potentials
$V_1,V_2:\R\to\R_{\geq 0}$ with $\deg(V_1,0)=\deg(V_2,0)=2$ such that $a(\theta)=b_E(\theta)$
for every $\theta\in[0,E]$. Moreover, $V_1$ and $V_2$ can be chosen so that $\sqrt{V_2''(0)/V_1''(0)}$ is irrational.
\end{proposition}

\begin{proof}
For any $N\geq 1$, let $P(x)=\sum_{n=0}^{2N} a_nx^n$ be any  polynomial with $a_{2N}>0$. For every $n\geq 0$, let
\[c_{2n}:=\int_0^1\frac{s^{2n}}{\sqrt{1-s^2}}ds>0.\]
We denote by $Q(x)=Q_{P,{E}}(x)=\sum_{n=0}^{2N} b_nx^n$ the polynomial determined by
\[\sum_{n=0}^{2N} b_nc_{2n}x^n=\sum_{n=0}^{2N} a_nc_{2n}(E-x)^n;\]
that is,
\begin{equation*}
b_n=(-1)^n\sum_{n\leq k\leq 2N}\binom{k}{n}E^{k-n}\frac{c_{2k}}{c_{2n}}a_k\quad\text{for all}\quad 0\leq n\leq 2N.
\end{equation*}
In particular, $b_{2N}=a_{2N}>0$.
Since the polynomials $P$ and $Q$ have even degree, there exists $d\in\R$ such that $P(x)+d>0$ and $Q(x)+d>0$ for all $x\in\R$.
Let $W_1$ and $W_2$ be polynomials such that $W_1(0)=W_2(0)=0$ and $W'_1(x)=P(x^2)+d$, $W'_2(x)=Q(x^2)+d$. As $W'_1(x)>0$, $W'_2(x)>0$, $W'_1(x)=W'_1(-x)$, $W'_2(x)=W'_2(-x)$ for all $x\in\R$, both $W_1$ and $W_2$ are odd bi-analytic maps {(in fact, they are polynomials of degree $4N+1$)}. Then, $V_1^*:=W_1^{-1}$, $V_2^*:=W_2^{-1}$ are also odd bi-analytic maps. Hence, $V_1:= (V_1^*)^2$ and $V_2:= (V_2^*)^2$ are even $\mathcal{UM}$-potentials. Then, the corresponding maps $a$ and $b_E$ are of the form
\begin{align*}
a(\theta)&=\frac{1}{\sqrt{2}}\int_0^1\frac{W'_1(\sqrt{\theta}s)}{\sqrt{1-s^2}}ds=\frac{1}{\sqrt{2}}\int_0^1\frac{P(\theta s^2)+d}{\sqrt{1-s^2}}ds=\frac{1}{\sqrt{2}}(dc_0+\sum_{n=0}^{2N}a_nc_{2n}\theta^n)\\
b_E(\theta)&=\frac{1}{\sqrt{2}}\int_0^1\frac{W'_2(\sqrt{E-\theta}s)}{\sqrt{1-s^2}}ds=\frac{1}{\sqrt{2}}\int_0^1\frac{Q((E-\theta) s^2)+d}{\sqrt{1-s^2}}ds\\
&=\frac{1}{\sqrt{2}}(dc_0+\sum_{n=0}^{2N}b_nc_{2n}(E-\theta)^n).
\end{align*}
By the definition of $Q$, it follows that $a(\theta)=b_E(\theta)$ for every $\theta\in[0,E]$.

Now, additionally, suppose that the polynomial $P$ is such that $\sum_{n=1}^{2N}a_nc_{2n}E^n\neq 0$. Then $b_0\neq a_0$.
Notice that
\[
\sqrt{\frac{V_2''(0)}{V_1''(0)}}=\frac{W_1'(0)}{W_2'(0)}=\frac{P(0)+d}{Q(0)+d}=\frac{a_0+d}{b_0+d}.\]
As $a_0\neq b_0$, we can choose $d$ sufficiently large so that $\frac{a_0+d}{b_0+d}$ is irrational. This completes the construction.
\end{proof}

{
\begin{remark}\label{rmk:relat}
In the proof of the previous theorem, we can take the polynomial $P(x)=\sum_{n=0}^{2N} a_nx^n$ so that $N$ is even and
\[\sum_{n=0}^{2N} a_nc_{2n}x^n=S(x)S(E-x), \]
where $S$ is any polynomial of degree $N$. Then $Q=P$ and $V_2=V_1$ are $\mathcal{UM}$-potentials that do not belong to $\mathcal{SP}$ such that $a(\theta)=a(E-\theta)=b_E(\theta)$.
Therefore, in this case we can also find polygons for which the set of resonant levels is a singleton.
\end{remark}}

\begin{remark}
Modifying slightly the above procedure, we can easily construct a pair of non-even potentials $V_1$, $V_2$ which do not belong to $\mathcal{SP}$ and such that
$a(\theta)+\bar{a}(\theta)=b_E(\theta)+\bar{b}_E(\theta)$ for all $\theta\in[0,E]$. Now we choose any non-zero real $d_1, \bar{d}_1$ and then $d_0$ large enough
so that $P(x^2)+d_1x+d_0>0$ and $Q(x^2)+\bar{d}_1x+d_0>0$ for every $x\in\R$. Then, we repeat the construction taking the polynomials $W_1$, $W_2$
so that $W_1(0)=W_2(0)=0$, $W'_1(x)=P(x^2)+d_1x+d_0$ and $W'_2(x)=Q(x^2)+\bar{d}_1x+d_0$. As $W'_1$ and $W'_2$ are not even, the corresponding potentials $V_1$, $V_2$ are also not even. As $W'_1(x)+{W}'_1(-x)=2P(x^2)+2d_0$ and $W'_2(x)+{W}'_2(-x)=2Q(x^2)+2d_0$ are not constant,
both $V_1$ and $V_2$ are not $\mathcal{SP}$-potentials.
Moreover,
\begin{align*}
a(\theta)+\bar{a}(\theta)&=\frac{1}{\sqrt{2}}\int_0^1\frac{W'_1(\sqrt{\theta}s)+\bar{W}'_1(\sqrt{\theta}s)}{\sqrt{1-s^2}}ds\\
&=\frac{1}{\sqrt{2}}\int_0^1\frac{2P(\theta s^2)+2d_0}{\sqrt{1-s^2}}ds=\frac{2}{\sqrt{2}}(d_0c_0+\sum_{n=0}^{2N}a_nc_{2n}\theta^n)\\
b_E(\theta)+\bar{b}_E(\theta)&=\frac{1}{\sqrt{2}}\int_0^1\frac{W'_2(\sqrt{E-\theta}s)+\bar{W}'_2(\sqrt{E-\theta}s)}{\sqrt{1-s^2}}ds\\
&=\frac{2}{\sqrt{2}}\int_0^1\frac{Q((E-\theta) s^2)+d_0}{\sqrt{1-s^2}}ds
=\frac{2}{\sqrt{2}}(d_0c_0+\sum_{n=0}^{2N}b_nc_{2n}(E-\theta)^n).
\end{align*}
By the definition of $Q$, it follows that $a(\theta)+\bar{a}(\theta)=b_E(\theta)+\bar{b}_E(\theta)$  for every $\theta\in[0,E]$.
\end{remark}

%\noindent
%\textbf{Data Availability Statement:} Data sharing not applicable to this article as no
%datasets were generated or analysed during the current study.

\end{document}